\pdfoutput=1

\PassOptionsToPackage{hidelinks,linktocpage,colorlinks=true, allcolors=blue, linktoc = all}{hyperref}
\PassOptionsToPackage{capitalise,nameinlink,noabbrev}{cleveref}

\documentclass[11pt,onecolumn,letterpaper,final,cleveref]{includes/colt2025}
\hypersetup{urlcolor=blue}
\title[]{The First-Order Oracle Complexity of Lipschitz Convex Optimization in Nondual Settings}

\usepackage{amsmath,amssymb,thmtools, mathtools, mathrsfs}

\usepackage{xcolor}
\usepackage{soul}

\usepackage{url}					%

\usepackage{enumitem}

\usepackage{xargs}
\usepackage{xfrac} %

\newcommand{\ones}{\mathbf{1}}

\usepackage{hhline}
\usepackage{booktabs} %
\usepackage{multirow}

\let\epsilon\varepsilon

\makeatother

\usepackage{tikz}

\usepackage{aliascnt}

\newcommand{\crefaliasappendix}{%
  \crefalias{section}{appendix} %
}

\crefname{equation}{}{}
\crefname{lemma}{Lemma}{Lemmas}
\crefname{example}{Example}{Examples}
\crefname{theorem}{Theorem}{Theorems}
\crefname{lemma}{Lemma}{Lemma}
\crefname{proposition}{Proposition}{Proposition}
\crefname{remark}{Remark}{Remark}
\crefname{claim}{Claim}{Claim}
\crefname{corollary}{Corollary}{Corollary}
\crefname{definition}{Definition}{Definition}
\crefname{conjecture}{Conjecture}{Conjecture}
\crefname{axiom}{Axiom}{Axiom}
\crefname{fact}{Fact}{Fact}
\crefname{assumption}{Assumption}{Assumptions}
\crefname{appendix}{Appendix}{Appendices}

\newtheorem{fact}[theorem]{Fact}
\newtheorem{assumption}[theorem]{Assumption}

\crefname{enumi}{Property}{Properties} %

\DeclareMathOperator*{\argmin}{arg\,min}

\newcommand{\E}{\mathbb{E}}

\newcommandx{\El}[2][1=, 2=, usedefault]{\mathbb{E}_{#2}\left[ #1 \right]}
\newcommand{\R}{\mathbb{R}}

\newcommand{\innp}[1]{\langle #1 \rangle}
\newcommand{\norm}[1]{\| #1 \|}

\newcommand{\abs}[1]{| #1 |}

\newcommand{\defi}{:=}

\newcommand{\Cube}{\mathcal C}

\coltauthor{
\Name{Brian Bullins}\Email{\href{mailto:bbullins@purdue.edu}{bbullins@purdue.edu}}\\
\addr Purdue University, West Lafayette, IN, USA
\AND
\Name{Cristóbal Guzmán}\Email{\href{mailto:crguzmanp@uc.cl}{crguzmanp@uc.cl}}\\
\addr Institute for Mathematical and Computational Engineering, Faculty of Mathematics and
School of Engineering, Pontificia Universidad Católica de Chile, Santiago, Chile
\AND
\Name{David Martínez-Rubio}\Email{\href{mailto:david.martinezrubio@imdea.org}{david.martinezrubio@imdea.org}}\\
\addr IMDEA Software Institute, Madrid, Spain
\AND
\Name{Mathieu Molina} \Email{\href{mailto:mathieu.molina.research@gmail.com}{mathieu.molina.research@gmail.com}}\\
\addr Tel Aviv University, Israel
}

\newcommand\blfootnote[1]{%
\begingroup
\renewcommand\thefootnote{}\renewcommand\footnoteseptext{}\footnote{#1}%
\addtocounter{footnote}{-1}%
\endgroup
}

\usepackage{geometry}
\editors{}         %
\usepackage{algorithm}
\usepackage{algcompatible}
\algnewcommand{\lst}{\texttt{lst}}
\algnewcommand{\slst}{\texttt{slst}}
\algnewcommand{\SEND}{\textbf{send}}

\newsavebox{\algleft}
\newsavebox{\algright}

\makeatletter
\newcounter{algorithmicH}
\let\oldalgorithmic\algorithmic
\renewcommand{\algorithmic}{%
  \stepcounter{algorithmicH}
  \oldalgorithmic}
\renewcommand{\theHALG@line}{ALG@line.\thealgorithmicH.\arabic{ALG@line}}
\makeatother



\usepackage[natbib, backend=biber, maxcitenames=3, minalphanames=3, maxbibnames=99, style=alphabetic, hyperref, backref, useprefix=true, uniquename=false, doi=false,url=false,eprint=false]{biblatex} 

\usepackage{csquotes}               
\bibliography{refs}        

\DeclareCiteCommand{\cite}
  {\usebibmacro{prenote}}
  {\usebibmacro{citeindex}%
   \printtext[bibhyperref]{\usebibmacro{cite}}}
  {\multicitedelim}
  {\usebibmacro{postnote}}

\DeclareCiteCommand*{\cite}
  {\usebibmacro{prenote}}
  {\usebibmacro{citeindex}%
   \printtext[bibhyperref]{\usebibmacro{citeyear}}}
  {\multicitedelim}
  {\usebibmacro{postnote}}

\DeclareCiteCommand{\parencite}[\mkbibparens]
  {\usebibmacro{prenote}}
  {\usebibmacro{citeindex}%
    \printtext[bibhyperref]{\usebibmacro{cite}}}
  {\multicitedelim}
  {\usebibmacro{postnote}}

\DeclareCiteCommand*{\parencite}[\mkbibparens]
  {\usebibmacro{prenote}}
  {\usebibmacro{citeindex}%
    \printtext[bibhyperref]{\usebibmacro{citeyear}}}
  {\multicitedelim}
  {\usebibmacro{postnote}}

\DeclareCiteCommand{\citeauthor}
  {\usebibmacro{prenote}}
  {\ifciteindex
     {\indexnames{labelname}}
     {}%
   \printtext[bibhyperref]{\printnames{labelname}}}
  {\multicitedelim}
  {\usebibmacro{postnote}}

\DeclareCiteCommand{\footcite}[\mkbibfootnote]
  {\usebibmacro{prenote}}
  {\usebibmacro{citeindex}%
  \printtext[bibhyperref]{ \usebibmacro{cite}}}
  {\multicitedelim}
  {\usebibmacro{postnote}}

\DeclareCiteCommand{\footcitetext}[\mkbibfootnotetext]
  {\usebibmacro{prenote}}
  {\usebibmacro{citeindex}%
   \printtext[bibhyperref]{\usebibmacro{cite}}}
  {\multicitedelim}
  {\usebibmacro{postnote}}

\DeclareCiteCommand{\textcite}
  {\boolfalse{cbx:parens}}
  {\usebibmacro{citeindex}%
   \printtext[bibhyperref]{\usebibmacro{textcite}}}
  {\ifbool{cbx:parens}
     {\bibcloseparen\global\boolfalse{cbx:parens}}
     {}%
   \multicitedelim}
  {\usebibmacro{textcite:postnote}}

\newbibmacro{string+doiurlisbn}[1]{%
  \iffieldundef{doi}{%
    \iffieldundef{url}{%
      \iffieldundef{isbn}{%
        \iffieldundef{issn}{%
          #1%
        }{%
          \href{http://books.google.com/books?vid=ISSN\thefield{issn}}{#1}%
        }%
      }{%
        \href{http://books.google.com/books?vid=ISBN\thefield{isbn}}{#1}%
      }%
    }{%
      \href{\thefield{url}}{#1}%
    }%
  }{%
    \href{https://doi.org/\thefield{doi}}{#1}%
  }%
}

\DeclareFieldFormat{title}{\usebibmacro{string+doiurlisbn}{\mkbibemph{#1}}}
\DeclareFieldFormat[article,incollection,inproceedings]{title}%
    {\usebibmacro{string+doiurlisbn}{#1}}

\input{includes/definitions.tex}
\renewcommand\nu{\oldnu}

\begin{document}
\maketitle
\thispagestyle{empty}

\begin{abstract}
    We study first-order black-box convex optimization over an $\ell_p$-ball for objectives Lipschitz in the $\ell_q$-norm, solving in the affirmative the nonsmooth version of the COLT open question \citep{pmlr-v40-Guzman15} on whether the geometry of a smaller feasible set ($p < q$) can improve convergence rates in convex optimization, and matching prior lower bounds up to logarithmic factors.
    Our rates include $\TildeNotation{O}(1/T)$ for convex Euclidean-Lipschitz optimization over the $\ell_1$-ball, improving on the $O(1/\sqrt{T})$ classical rate under general assumptions.
    The key technical device is a new online learning game, where the comparator is evaluated using the maximum of affine losses observed so far. We bound the value of this game above and below in terms of a combinatorial online learning quantity: the sequential fat-shattering dimension, which we characterize for the $\ell_p / \ell_q$ case. 
    Our results generally apply when the feasible set $\X$ and the set of possible subgradients $\H$ admit a type of minimax theorem, advancing on a fundamental question by \citet[Section 10.1.2, Q3]{sridharan2012learning}. 
    As a geometric consequence of our analysis, of independent interest, we obtain estimates for the expected distance of a convex hull of samples to their mean in several Banach geometries, a version of the celebrated Wendel's theorem \citep{wendel1962problem}, but quantitative and for bounded general distributions as opposed to centrally symmetric ones.
\end{abstract}

\blfootnote{Authors are listed in alphabetic order.}
\blfootnote{\color{darkgray}Most non-local notation in this work links to its definition, using \href{https://damaru2.github.io/general/notations_with_links/}{this code}; for example, ${\notationlink{def:feasible-set}{\color{darkgray}\mathcal X}}$ links its definition as the feasible set of the optimization problem we consider in this work.}

\clearpage
\setcounter{page}{1}

\section{Introduction}

First-order optimization methods arise in high-dimensional settings as a basic computational primitive of continuous optimization across the sciences, engineering, statistics, and machine learning. The \emph{first-order oracle} complexity model is the main framework to investigate optimal rates of convergence in high-dimensional convex optimization \citep{nemirovski1983problem}, with the Lipschitz setting being one of the most fundamental and natural regimes, widely satisfied across problems where local-search methods work \citep{nemirovski1983problem,nesterov2004introductory}. 

The framework asks an algorithm $\ALG$ to minimize a convex Lipschitz function $f$ accessed \emph{only} via a local first-order oracle, which on a point $x$ returns the value $f(x)$ and a subgradient $g\in\partial f(x)$, cf. \cref{eq:convex_subdifferentiable}. The minimax rate of the output $\widehat{x}_T$ of $\ALG$ after $T$ steps for a class of functions $\mathcal{F}$, namely 
\[
    \OptErr[\X,\H](T) \defi   \inf_{\ALG}\sup_{f\in\mathcal{F}} (f(\widehat x_T)-\min_{x}f(x)),
\] is dictated by the geometry of the feasible set $\X$ and of the set $\H$ in which subgradients of $f$ can lie.  In the so-called \emph{dual} setting, $\X$ is bounded in a given norm $\norm{\cdot}$, and $\H$ is bounded in the dual norm $\norm{\cdot}_\ast$, that is, the objective is Lipschitz in $\norm{\cdot}$ with a uniformly bounded Lipschitz constant. %
Dual settings were settled optimally in $\ell_p$-norms using the Mirror Descent (\newtarget{def:acronym_mirror_descent}{\MD{}}) algorithm \citep{nemirovski1983problem}, that is, $\X = \ball{p}{d}$ and $\H = \ball{p^\ast}{d}$, where $\ball{r}{d}$ denotes the unit $\ell_r$-ball in $d$ dimensions and $p^\ast$ satisfies $1/p + 1/p^\ast = 1$. More generally, it was shown that \MD{} is universally optimal in a fairly broad family of dual problems \citep{srebro2011universality}. A vast majority of developed algorithms in this and other optimization problems work in dual settings.

\paragraph{Nondual settings.} The nondual problem is far less understood. A combinatorial quantity from learning theory, known as the \emph{fat-shattering dimension}\footnote{This quantity, introduced by \citet{kearns1994efficient}, is a continuous scale-sensitive analog of the VC dimension, cf. \cref{def:fat-shattering-dimensions}.} for a class of linear functions depending on $\X$ and $\H$, provides a lower bound even in nondual settings \citep{srebro2012convex}, and  \citet[Section 10.1.2, Q3]{sridharan2012learning} conjectures the possibility that this quantity determines the minimax complexity in general, up to constants. %
\citet{guzman2015information} studies specifically the  high-dimensional $(p, q)$-norm case, $p, q \in [1, \infty]$, where $\X = \ball{p}{d}$, $\H = \ball{q^\ast}{d}$, $d \geq T$. The corresponding lower bounds, which are also fat-shattering lower bounds in hindsight, turn out to exhibit polynomial gaps with the best-known \MD{}-based algorithm when $p < \min\{q,2\}$ (up to logarithmic factors).
This gap motivates the open problem \citep{pmlr-v40-Guzman15}.\footnote{The case $p \geq 2$ when $p < q$, that was also stated as an open question \citep{pmlr-v40-Guzman15}, is in fact also solved optimally with mirror descent: a $G_q$-Lipschitz function in $\ell_q$ is also $G_q$-Lipschitz in $\ell_p$, so running mirror descent over $R \ball{p}{d}$ and using the $\ell_p$ geometry gives rates of $\bigo{R G_q / T^{1 / p}}$ for $p \geq 2$, which nearly matches the lower bound. The hard question is $p < \min\{q, 2\}$, where rates would be faster than $\bigo{1 / \sqrt{T}}$. Similarly, regular acceleration in the $p$-norm works near optimally in high dimensions when $2 \leq p < q$.} 
The question is whether the geometry of the smaller feasible set $\ball{p}{d} \subsetneq \ball{q}{d}$ than in the dual setting, would allow for new algorithms that match the lower bound, or whether better lower-bound techniques are needed to establish the optimality of current algorithms. The question was also posed for smooth settings where new lower bounds were provided, which nonetheless presented a similar gap. 

With the $(1, 2)$ case as an example, if we optimize a Euclidean-Lipschitz convex function $f(x)$ with $x$ constrained to an $\ell_1$-ball or a subset of it, the currently best known high-dimensional rate does not improve over $\bigo{1 / \sqrt{T}}$, which is what is also achieved over the much larger $\ell_2$-ball domain, whereas the lower bound only prevents a much faster decay of $O(1 / T)$. It is worth noting that despite the remaining gaps in these nondual settings in the literature, they naturally appear in several applications. Any problem with a probability distribution as variable and Lipschitzness in an $\ell_q$-norm $q> 1$, such as the natural Euclidean Lipschitzness condition, is a nondual problem.  In inverse problems, mixed-geometries result from the discrepancy between a signal structure and measurement device. For block MRI sampling, there are optimization formulations whose feasible domain is a high-dimensional $\ell_{\infty}$ box, and provably faster algorithms can be obtained by considering Lipschitzness measured in $\ell_p$-norms, where $1\leq p\leq 2$ \citep{BoyerWeissBigot2014}. %
Another example is convex relaxations of sparse PCA \citep{Dey:2017}, where the objective is a quadratic form given by the directional covariance, for which it is natural to regard the objective as Lipschitz in the $\ell_2$-norm, and the feasible set is given by an intersection of $\ell_2$- and $\ell_1$-balls of different radii.
In infinite dimensions, the Rudin-Osher-Fatemi model for total variation regularization in imaging leads to $L^2$ constraints with an objective given by a total variation seminorm \citep{Rudin:1992}.

\paragraph{Main results.} In this work, we identify a new high-dimensional phenomenon in convex optimization and achieve the $O( 1 / T)$ rate mentioned in the $(1, 2)$ example above, up to log factors. More generally, for the $(p, q)$-norm case with $1\leq p\leq q\leq \infty$, our rates are
\[
\OptErr[\ball{p}{d},\ball{q^\ast}{d}](T)=\bigotildepl{p, q}{\frac{1}{ T^{ \frac{1}{p} - \left( \frac{1}{q}-\frac{1}{2} \right)\pospart
        }}},
\] 
which nearly match the mentioned lower bounds up to logarithmic factors, %
solving the nonsmooth part of the COLT 2015 open question in \citep{pmlr-v40-Guzman15},
providing a polynomial improvement over the rate of Mirror Descent algorithms $\bigo{T^{ -1 / \max\{p\wedge q, 2\}}}$ %
when $p < \min\{q, 2\}$. Above, $1 / \infty$ should be interpreted as $0$. Additionally, we obtain an upper bound in the general setting where the feasible set $\X$ and the set of possible subgradients $\H$ are convex and the latter is centrally symmetric, showing that the minimax complexity is upper bounded by a combinatorial quantity known as the 
\emph{sequential fat-shattering dimension} (sfat) which is no less than the fat-shattering dimension (fat), and we show they are the same up to log factors in the $(p, q)$-case. This general result provides the first positive contribution towards understanding the question posed by \citet{sridharan2012learning}. %
Just as exciting as our results are the techniques behind them: we define and study a new online learning game, of independent interest, which we explain in the following.

\paragraph{Regret limitations, bundle methods, and our new max-regret.} To put our solution into context, we first note that the question of \citet{pmlr-v40-Guzman15} is especially puzzling if one takes into account that, despite decades of research, the \MD{} algorithm in \citet{nemirovski1983problem} is essentially the best known method for this problem, defaulting to the dual geometry when one could instead exploit the smaller geometry of the feasible set. Other families of algorithms have been devised under different names, like Dual Averaging or Follow-the-Regularized-Leader (FTRL) \citep{gordon1999regret,shalevshwartz2007primal-dual,nesterov2009primal-dual}, or Follow-the-Perturbed-Leader \citep{kalai2003efficient}, but all can be viewed through the common template and analysis of optimizing summed-up linearizations together with a regularizer \citep{abernethy2014online,mcmahan2017survey} and do not yield faster minimax rates. Most of the best known bounds in convex Lipschitz optimization are, in hindsight, regret bounds for the more general problem of online convex optimization \citep{Warmuth:1998,gordon1999regret,Zinkevich:2003,BeckTeboulle:2003,shalevshwartz2007online}. In this setting, a different convex function $\ell_t$ is chosen adversarially at every round even with knowledge of the competing algorithm, and the performance after $T$ steps is measured by the regret $R_T := \sum_{t=1}^T \ell_t(x_t) - \inf_{u\in\X} \sum_{t=1}^T \ell_t (u)$. If we set all $\ell_t$ to $f$ and output the average query, Jensen's inequality yields convex optimization rates of $R_T / T$. However, any analysis that first reduces the optimization problem to worst-case regret for convex Lipschitz losses is formally limited to an $\Omega(T^{-1/2})$ guarantee as per dual \citep{sridharan2010convexgames,johnson2025adaptive} and nondual \citep[Proposition~9]{rakhlin2015online}, \citep[Lemma~3 and Theorem~5]{cheng2025geometry} regret lower bounds, and can be worse for several geometries, which provides some evidence that \MD{} algorithms may not be able to match the nondual lower bounds we achieve in this work, although a \MD{} optimization analysis would not necessarily have to come from a reduction to regret.

Our nearly-optimal rates are achieved via introducing a new paradigm: a new online learning game, where the metric is what we call the \emph{max-regret}, where the comparator is measured via $T$ times the maximum of previous losses, rather than the sum:
\begin{equation}\label{eq:m-regret}
    \newtarget{def:max-regret}{\MaxRegret[\X,T]} \defi \sum_{t=1}^T \ell_t(x_t) - \inf_{u\in \X} T\max_{t\le T}\ell_t(u),
\end{equation}
or $\MaxRegret[T]$, if $\X$ is clear from context.
This regret is compatible with convex optimization in the sense that the average query $\bar{x}_T := T^{-1}\sum_{t=1}^T x_t$ can similarly be used to obtain optimization rates of $\MaxRegret[T] / T$, cf. \cref{lemma:m-online-to-batch}, by setting all $\ell_t$ to $f$, or alternatively to its first-order Taylor approximation $f(x_t) + \innp{g_t, \cdot - x_t}$ at the query $x_t$ using subgradient $g_t \in \partial f(x_t)$. Note that by the nature of the $\max$, we have $\MaxRegret[T] \leq R_T$. In the settings we study, $\MaxRegret[T]$ will be significantly lower. We lower and upper bound the minimax value of the Lipschitz convex game by expressions depending on sequential fat-shattering, which we show nearly match in several settings. In particular, they match for the $(p, q)$-norm case, where we characterize these quantities, up to log factors. Given the widespread applications of regret across areas of computer science, we believe the $\max$-regret and its analysis to be of great independent interest.

Our results go in the direction of finding oracle-efficient algorithms using online learning arguments \citep{abernethy2008optimal,abernethy2009stochastic,rakhlin2010online,sridharan2010convexgames,rakhlin2011online,rakhlin2012relax}. We establish the existence of a strategy minimizing the max-regret that leads to near-optimal first-order oracle optimization complexity. This strategy can be turned into an exponential time algorithm by dynamic programming and standard discretization arguments, leading to an algorithm with runtime exponential in $T$ and $d$. 
We leave the study of an algorithm that simultaneously achieves near-optimal first-order oracle rates and computational efficiency in the $(p,q)$-case for a forthcoming paper. The latter property should be understood as the required number of real operations being polynomial in the number of optimal first-order oracle queries and the dimension.

Finally, the minimization of the $\max$-regret above should remind the reader of bundle methods \citep{kelley1960cutting-plane,lemarechal1995new,ben-tal2005non-euclidean,lan2015bundle-level,drori2016optimal-kelley,diaz2023proximal-bundle}, which keep a maximum of affine minorants of the objective to compute their next query. Bundle methods have been reported to perform well in practice in terms of oracle queries \citep{ben-tal2005non-euclidean,lan2015bundle-level} but to the best of our knowledge, our algorithm is the first bundle method that provides an advantage in worst-case oracle optimization rates with respect to regret-based algorithms for convex Lipschitz optimization. In fact, we are not aware of any other problem class where bundle methods improve over other techniques. 

We refer to \cref{sec:further_related_work} for further related work. 
We next give a technical overview of our main results in the concrete mixed $\ell_p/\ell_q$ geometry and beyond, outlining the proof strategy and highlighting the core novel technical results that make the argument possible.

\subsection{Technical Overview} \label{sec:technical-overview}

The upper bound proof proceeds in four steps. First, we show that $\max$-regret controls the optimization error when applied to the supporting affine minorants returned by the first-order oracle. Second, a minimax reduction turns the $\max$-regret game into a sequential approximation problem from a convex-hull of samples to its mean of means. Third, we bound the value of this game above and below using sequential fat-shattering, a scale-sensitive complexity measure for adaptive function classes. Finally, we characterize this combinatorial dimension for the associated $(p, q)$ linear class up to logarithmic factors and show that it nearly matches its non-sequential counterpart, the ordinary fat-shattering dimension. Combined with the known oracle lower bounds based on ordinary fat-shattering, this yields our main optimization conclusion in \cref{cor:first-order-optimization-sfat}. We sketch some of the arguments of these steps for the $(p, q)$-norm case for $p<q$.  That is, the feasible set is $\ball{p}{d}$ and the subgradients live in $\ball{q^\ast}{d}$.

\paragraph{Max-regret online-to-batch.} We normalize $f(0) = 0$ without loss of generality. The first-order oracle returns affine minorants
$\ell_t(u)=f(x_t)+\langle g_t,u-x_t\rangle \leq f(u)$. Let $\Vmax[p,q^\ast,T]$ denote the minimax value of the $\max$-regret attainable against adaptive affine losses with
these slopes and also intercepts in $[-1,1]$\footnote{The minorant
intercepts lie in $[-2,0]$ and a common shift by $1$ places them
in $[-1,1]$ without changing $\max$-regret.}, which is a consequence of Lipschitzness and normalizing $f(0)=0$, as formalized in
\eqref{eqn:max_regret_value}. The max-regret online-to-batch reduction in \cref{lemma:m-online-to-batch} therefore gives the following for the minimax optimization error:
\[
    \OptErr[\ball{p}{d},\ball{q^\ast}{d}](T)
    \le \frac{\Vmax[p,q^\ast,T]}{T},
\]
so it is enough to study the minimax value of the max regret.  

\paragraph{Online reduction and stochastic processes.} The minimax reduction in \cref{thm:generalized_triplex_like_theorem} removes the player's decisions by choosing actions that minimize the conditional expected loss in the minimax-swapped game. Write $\pi$ for an adversarial
strategy that chooses predictable distributions $(r_t)_{t\leq T}$ from the previously sampled
losses, draws $\ell_t\sim r_t$, and has conditional mean
$m_t(u)=\mathbb E[\ell_t(u)\mid\ell_{1:t-1}]$. Then, we show in \cref{thm:generalized_triplex_like_theorem} that
\[
    \frac{\Vmax[p,q^\ast,T]}{T}
    \le
    \sup_\pi\mathbb E_\pi
    \sup_{u\in \ball{p}{d}}
    \left\{
        \frac1T\sum_{t=1}^T m_t(u)
        -\max_{t\le T}\ell_t(u)
    \right\}.
\]
Next we use $\max_t\ell_t(u)=\max_{\lambda\in\ProbSimplex{T}}\sum_t\lambda_t\ell_t(u)$,
where $\ProbSimplex{T}=\{\lambda\in[0,1]^T:\sum_t\lambda_t=1\}$.
For linear losses, these weights form convex combinations of the sampled
gradients. It is instructive to first consider the case of linear and iid losses, in which case duality identifies the inner supremum with the $\ell_{p^\ast}$-distance from their common mean to their convex hull (\cref{lem:iid-online-reduction-identity}), a problem of independent interest. Techniques for the general case follow from adaptations of this simpler case. That is, for a probability distribution $P$ supported on $\ball{q^\ast}{d}$, let $\mu_P\defi\E_{G\sim P}[G]$, and let $G_1,\ldots,G_T$ be independent samples from $P$. The expression above in this case becomes the worst-case expected distance from the mean to the convex hull of the samples, which is
\begin{equation} \label{eq:cvx-hull-dist-intro}
    \sup_{\operatorname{supp}(P)\subseteq \ball{q^\ast}{d}}
    \E
    \left[
        \operatorname{dist}_{p^\ast}
        \left(
            \mu_P,
            \conv\{G_1,\ldots,G_T\}
        \right)
    \right],
\end{equation}
where $ \operatorname{dist}_{p} (x,S)$ is the $p$-norm distance from point $x$ to set $S$.\footnote{We note in passing that the quantity \eqref{eq:cvx-hull-dist-intro} has some history in convex optimization. For example, \citep{nemirovski1983problem} (page 163 therein) prove their $\ell_1$-dual setting lower bound by building a distribution $P$ that lower bounds the convex-hull distance problem.} We show that the quantity above is of the order of $\TildeNotation{\Theta}[p,q] \big( 1 / T^{ \frac1p - \left( \frac1q-\frac12 \right)\pospart } \big)$ in high dimensions $d \geq T$, and in general it depends on the (nonsequential) fat-shattering dimension.
The bound is obtained using arguments from stochastic processes, that gracefully generalize to the adaptive non-iid case. One key idea consists of restricting the convex-hull weights $\lambda_t$ to a {\em capped simplex}
$\CappedSimplex{T}=\{\lambda\in\ProbSimplex{T}:\lambda_t\le2/T\}$. This can only increase the approximation error we must bound and allows for the use of stochastic processes tools. The intuition is that
we have a scalar empirical-process problem in each direction $u$, and just uniform weights $\lambda = \ones / T$ would provide strong averaging but only recover the usual $T^{-1/2}$ fluctuation scale, whereas the full simplex can adapt completely to the realized sample but may concentrate all its mass on a single observation, destroying the averaging structure needed for concentration. The capped simplex interpolates between these two extremes: imposing $\lambda_i\le \kappa/T$ forces every admissible combination to use at least $T/\kappa$ samples, while still allowing the weights to adapt to the realization.

Crucially, the approximation error over the capped simplex can be upper bounded by an offset empirical process of the form $\E[\phi(\langle u,G\rangle)]-\kappa T^{-1}\sum_i\phi(\langle u,G_i\rangle)$ for some non-negative function $\phi$. The extra factor $\kappa>1$ creates a negative drift, so overcoming this drift requires a constant fraction of the samples to deviate coherently in the same unfavorable direction. Such collective deviations have exponentially decaying tails, yielding finite-class bounds of order $\log N/T$ rather than the usual $\sqrt{\log N/T}$, for a covering of functions of size $N$. The (sequential) fat-shattering dimension then enters by controlling the sizes of the covers used to pass from finite classes to the full function class, which is approximated at successively finer scales in a chaining argument.

\paragraph{Bounding the rate via sfat.}
After generalizing the stochastic processes argument to the adaptive sequential case, the relevant class whose complexity we study for the $(p,q)$-norm case is
$\cL[p,q^\ast][d]=\{\ball{q^\ast}{d}\ni g\mapsto\langle u,g\rangle:
u\in \ball{p}{d}\}$. The final expression for the upper bound involves the so-called sequential covering numbers \citep{rakhlin2015martingale} which are known to be bounded by  $\sfat[\alpha](\cL[p,q^\ast][d])$.
A similar analysis in the iid case yields expressions depending on nonsequential covering numbers which are bounded by $\fat[\alpha](\cL[p,q^\ast][d])$. Lower bounds on the max-regret value depending on these quantities also arise naturally leading to \cref{thm:fully-sequential-bound}.

Finally, we compute bounds for $\sfat[\alpha](\cL[p,q^\ast][d])$.
Set $\theta_{p,q}=1/p-(1/q-1/2)\pospart$. For $p<q$ and
$0<\alpha\lesssim_{p,q}1$, \cref{cor:explicit_regimes} establishes
\[
    \sfat[\alpha](\cL[p,q^\ast][d])
    =
    \TildeNotation{\Theta}[p,q]
    \left(
        \min\left\{
            \alpha^{-1/\theta_{p,q}},
            \ d\left(1+\log\pospart\frac1{\alpha d^{\theta_{p,q}}}\right)
        \right\}
    \right)
    =
    \TildeNotation{\Theta}[p,q]
    \left(
        \fat[\alpha](\cL[p,q^\ast][d])
    \right).
\]
{\color{black}
For the upper bound we cover the feasible ball by smaller uniformly convex balls,
bound the complexity within each ball, and combine the classes using a
sequential union bound, cf. \cref{lem:seq_union_bound}. For the lower bound, ordinary fat-shattering
}
provides a cube of possible evaluation vectors. Binary search along each
coordinate turns this cube into an adaptively shattered tree. This also
explains why sequential shattering can continue growing logarithmically
after ordinary shattering saturates at dimension $d$.
When $T\le d$, the scale $\alpha\asymp_{p,q}T^{-\theta_{p,q}}$
lies in the first regime. Substitution into the game bound yields the
optimization upper bound
$\TildeNotation{O}[p,q](T^{-\theta_{p,q}})$
(\cref{cor:first-order-optimization-sfat}); the ordinary-fat oracle lower
bound of \citet[Theorem~1]{srebro2012convex} matches this rate up to logarithmic
factors. In particular, $(p,q)=(1,2)$ gives
$\TildeNotation{O}(1/T)$.

\subsection{Further Results}\label{sec:further-results}

\color{black}
\paragraph{The iid case and samples-convex-hull approximation to the mean.}

An intermediate expression used to bound the max-regret minimax value in the $(p, q)$-norm case, when specialized to the case where the losses are linear and iid, results in a problem of independent interest about how fast the convex hull of samples from a fixed distribution supported on $B_{q^\ast}$ approaches its mean in $\ell_{p^\ast}$ distance, cf. \cref{lem:iid-online-reduction-identity,eq:worst-case-iid-convex-hull-distance}. An analogous statement holds when working with general sets $\X$ and $\H$. 
Our techniques also yield upper and lower bounds for this problem, see \cref{thm:iid-centered-convex-hull,prop:iid-high-dimensional-instantiation}, which match in the $(p, q)$-norm case.
This result can be viewed as an early-time, quantitative analogue of Wendel-type convex-hull containment theorems \citep{wendel1962problem}: the latter gives the exact probability that the convex hull of iid points from a centrally symmetric distribution in general position contains the origin, and gives a sharp transition at around the dimension. We instead measure the distance of the convex hull to the mean in the $\ell_{p^\ast}$ norm and we use general \emph{bounded} distributions, rather than symmetric ones. Previous quantitative variants have been obtained in very specific cases, including Euclidean distance for the uniform distribution \citep{liu2014probabilities} and distribution-dependent Euclidean-distance settings \citep{hayakawa2023estimating}.
We note however that allowing general distributions changes the problem and answer substantially.

Our results show a decay as fast as $1 / T$ for the more benign geometries, but no faster. The 1D case already illustrates this phenomenon: if we have a symmetric distribution, we only need one sample on each side of the mean in order for the convex hull to contain it, and the probability of this not happening decreases as $\exp(-\Omega(T))$. However, for arbitrary distributions bounded in $[-1, 1]$ and $T$ steps, we have the distribution with atoms at $-1$ and $1 / T$ with probabilities $1 / (T+1)$ and $T / (T +1)$, respectively, which is bounded but would yield a decay in the expected distance of $\Theta(1 / T)$ after $T$ steps.
\color{black}

\color{black}
\paragraph{A general Banach-space optimization bound via sfat, and applications.} Our arguments extend beyond finite-dimensional mixed-norm geometry. The general bound in \cref{thm:fully-sequential-bound}, combined with \cref{lemma:m-online-to-batch}, controls optimization error through sequential fat-shattering under \cref{ass:minimax_theorem_consequence}. In \cref{sec:function-space-examples}, we obtain matching rates, up to logarithmic factors, for three infinite-dimensional settings.

First, for the Sobolev unit ball of $H^s([0,1]^m)$ with $L_2$-bounded subgradients and $0<s\le m/2$, \cref{prop:sobolev-function-space-example} gives the rate $\TildeNotation{\Theta}[s,m](T^{-(1/2+s/m)})$. By comparison, treating the feasible set only as a subset of the $L_2$ unit ball gives the standard $T^{-1/2}$ rate via the subgradient method \citep{Shor:1964}. Thus any positive Sobolev smoothness yields a polynomial improvement, reaching the $1/T$ rate at the critical smoothness $s=m/2$. For normalized $1$-Lipschitz functions on $[0,1]^m$ with signed point-evaluation subgradients, \cref{prop:lipschitz_bound} gives the rate $\TildeNotation{\Theta}[m](T^{-1/m})$, reaching the $1/T$ rate in one dimension. Here the feasible set is norm compact by the Arzelà-Ascoli theorem. Finally, for an $\ell_p$-ball of coefficients in an orthonormal basis, with $1<p<2$ and signed coordinate-evaluation subgradients, \cref{prop:basis_bound} gives the rate $\TildeNotation{\Theta}[p](T^{-1/p})$. The ambient space is reflexive, although the feasible ball is not norm compact. The corresponding problem with the full $L_2$ unit ball as feasible set has the standard $T^{-1/2}$ rate via
mirror descent \citep{nemirovski1983problem}, so the
improvement exploits the restricted feasible set.
We also discuss applications of these examples to functional regression, individual fairness, and wavelet representations for image compression and denoising.  Together, these results show that faster rates from nondual geometry are neither specific to the $\ell_p/\ell_q$ setting nor to finite dimension, but are governed more generally by the pairing between feasible points and admissible subgradients.
\color{black}

\section{\textcolor{black}{Notation and Preliminaries}}\label{sec:notation-preliminaries}

For $r\in[1,\infty]$, let $r^{\ast}\in[1,\infty]$ satisfy $1/r+1/r^{\ast}=1$, with the usual endpoint conventions, and denote the unit $r$-ball as $\newtarget{def:unit-norm-ball}{\ball{r}{d}}:=\{x\in\mathbb R^d:\|x\|_r\le1\}$. For $T\ge1$, we denote the simplex by $\newtarget{def:probability-simplex}{\ProbSimplex{T}}=\{\lambda\in\mathbb R\pospart^T:\sum_{t=1}^T\lambda_t=1\}$. For $a\in\mathbb R$, we use $\newtarget{def:positive-part}{a\pospart\defi\max\{a,0\}}$, and write $\log\pospart t\defi(\log t)\pospart$. For a condition $A$, we write $\mathbf 1_A$ for its $0/1$ indicator. We write $\newtarget{def:big_O_p}{O_{\vartheta}(\cdot)}$ when the omitted constant may depend on the parameter or tuple of parameters $\vartheta$, and $\newtarget{def:big_O_tilde}{\TildeNotation{O}(\cdot)}$ to omit logarithmic factors in the problem and algorithm's parameters such as the dimension $d$ or number of iterations $T$. We use $\vee$ and $\wedge$ for the $\max$ and $\min$ of two expressions, respectively.

\paragraph{\textcolor{black}{Optimization setting and oracle.}}

Let $E$ be a real Banach space, let $\newtarget{def:feasible-set}{\X}\subseteq E$ be our feasible set: a nonempty, closed, bounded, convex set containing $0$. Let $\newtarget{def:subgradient-set}{\H}\subseteq E^\ast$ be a nonempty bounded centrally symmetric set containing $0$.\footnote{We tacitly equip all spaces below with the natural
$\sigma$-algebras making the relevant evaluation maps and probability
kernels measurable, and assume that all function classes appearing in
stochastic suprema are pointwise measurable.} We write $\langle g,u\rangle\defi g(u)$ for the duality pairing between $E^\ast$ and $E$. {\color{black}On a first read, one may take $E=\mathbb R^d$, identify $E^\ast$ with $\mathbb R^d$, and interpret $\langle g,u\rangle$ as the usual inner product.}

A function $f:\X\to\mathbb R$ is convex and subdifferentiable in $\X$ if for every $x\in\X$ there is a functional $g$ in the dual $E^\ast$, such that:
\begin{equation}\label{eq:convex_subdifferentiable}
    f(y)\ge f(x)+\langle g,y-x\rangle
    \qquad \forall y\in\X.
\end{equation}
Then $g$ is called a subgradient and the set of all such subgradients is the subdifferential, $\partial f(x)$. Throughout, our objectives are real-valued convex subdifferentiable functions on their feasible sets.

We consider convex subdifferentiable functions $f:\X\to\mathbb R$, normalized by $f(0)=0$, accessed through a first-order oracle. At every query $x\in\X$, the oracle returns the value $f(x)$ and a subgradient $g\in\partial f(x)\cap\H$ in the sense of \cref{eq:convex_subdifferentiable}.

We define the minimax error after $T$ adaptive first-order oracle queries as
\begin{equation}\label{eq:general-minimax-error}
    \newtarget{def:minimax-optimization-error}{\OptErr[\X,\H](T)}
    \defi
    \inf_{\mathrm{ALG}}\sup_{f,\mathrm{Oracle}}
    \big(
        f(\widehat x_T)-\inf_{u\in\X}f(u)
    \big),
\end{equation}
where $\mathrm{ALG}$ outputs $\widehat x_T\in\X$ after $T$ queries, and the supremum ranges over admissible convex objectives $f$ and valid adaptive first-order oracles.
We define the cross-duality constant by
\begin{equation}\label{eq:cross-duality-constant}
    \newtarget{def:cross-duality-constant}{\LambdaPair[\X,\H]}
    \defi
    \sup_{g\in\H,\,u\in\X}|\langle g,u\rangle|.
\end{equation}
The boundedness assumptions imply $\LambdaPair[\X,\H]<\infty$. We often use $r\in[1,\infty]$ in a subscript to denote $\ball{r}{d}$. As a running concrete example, the reader may simply keep in mind $E=\mathbb R^d$, $\X=\ball{p}{d}$, and $\H=\ball{q^\ast}{d}$, which is the setting of the optimization open question. In particular, if $p\le q$, then $\LambdaPair[p,q^\ast]=1$, since $|\langle g,u\rangle|\le\|u\|_p\|g\|_{p^\ast}\le\|g\|_{q^\ast}\le1$.

\paragraph{\textcolor{black}{Shattering dimensions.}}

{\color{black} We define the associated linear pairing class by}
\begin{equation}\label{eq:linear_classes}
    \newtarget{def:linear-pairing-class}{\cL[\X,\H]}
    \defi
    \left\{
        \H\ni g\mapsto\langle g,u\rangle:
        u\in\X
    \right\}.
\end{equation}
Every function in $\cL[\X,\H]$ takes values in $[-\LambdaPair[\X,\H],\LambdaPair[\X,\H]]$. {\color{black}The class $\cL[\X,\H]$ consists of functions on $\H$, indexed by feasible points $u\in\X$.} That is, the roles of $\X$ and $\H$ are dualized in the shattering reduction, with respect to how we think about them in optimization.

{\color{black}For instance, $\cL[p,q^\ast][d]$ corresponds to \cref{eq:linear_classes} when $(\X,\H)=(\ball{p}{d},\ball{q^\ast}{d})$, where we made $d$ explicit to remove ambiguity.} {\color{black}Our bounds are expressed through the fat-shattering and sequential fat-shattering dimensions of this linear pairing class, defined below.}

\begin{definition}[Fat-shattering and sequential fat-shattering dimensions]
\label{def:fat-shattering-dimensions}
For a set $\mathcal Z$, let $\mathcal G\subseteq\mathbb R^{\mathcal Z}$ and $\alpha>0$. A set $z_1,\dots,z_m\in\mathcal Z$ is $\alpha$-shattered by $\mathcal G$ if there are thresholds $s_1,\dots,s_m\in\mathbb R$ such that, for every $\varepsilon\in\{\pm1\}^m$, some $h_\varepsilon\in\mathcal G$ satisfies
\[
    \varepsilon_i\bigl(h_\varepsilon(z_i)-s_i\bigr)
    \ge\frac{\alpha}{2},
    \qquad \forall i=1,\dots,m.
\]

A depth-$T$ predictable tree over $\mathcal Z$ is a sequence $\mathbf z=(z_t)_{t=1}^T$, with $z_t:\{\pm1\}^{t-1}\to\mathcal Z$. It is sequentially $\alpha$-shattered by $\mathcal G$ if there is a real-valued predictable tree $\mathbf s=(s_t)_{t=1}^T$ such that, for every $\varepsilon\in\{\pm1\}^T$, some $h_\varepsilon\in\mathcal G$ satisfies
\[
    \varepsilon_t
    \bigl(h_\varepsilon(z_t(\varepsilon_{<t}))
    -s_t(\varepsilon_{<t})\bigr)
    \ge\frac{\alpha}{2},
    \qquad \forall t=1,\dots,T.
\]
The fat-shattering dimension $\newtarget{def:fat-shattering-dimension}{\fat[\alpha](\mathcal G)}$ and its sequential version $\newtarget{def:sequential-fat-shattering-dimension}{\sfat[\alpha](\mathcal G)}$ are the largest possible such $m$ and $T$, respectively.
\end{definition}

Note that using the definition with a predictable tree that is constant across each level yields $\fat[\alpha](\mathcal G)\le\sfat[\alpha](\mathcal G)$; see also \citep[Section~2]{rakhlin2015online}.

\paragraph{\textcolor{black}{From optimization to max-regret.}}

Given a nonempty decision set $\X$, actions $x_t\in\X$, and losses $\ell_t:\X\to\mathbb R$ in some class, we define the $\max$-regret as in \cref{eq:m-regret} without explicitly writing the losses class. Recall that the $\max$-regret, despite its name, is always no greater than the regular regret, due to the $\max$ appearing with a minus sign in its definition.

The following lemma makes explicit how $\max$-regret allows for a simple reduction to minimizing a convex function, by repeatedly using it or surrogates of it like its linearizations at queried points $\ell_t(x)=f(x_t)+\innp{g_t,x-x_t} \leq f(x)$, $g_t\in\partial f(x_t)$. It works since the first-order information observed along the optimization trajectory must be consistent with affine minorants of a single convex function. The lemma is analogous to the classical online-to-batch conversion for regular regret for convex functions and $\mathcal{U}=\X$. 
We note the lemma also allows for optimizing a Lipschitz function under $\gamma$-quasar convexity, where for a $\gamma \in (0, 1]$ and any $x \in \X$, we have $f(x^\ast) \geq f(x) + \gamma^{-1} \innp{g_x, x^\ast-x}$ for a fixed minimizer $x^\ast$ of $f$ in $\X$: just use the losses $\ell_t(x) = f(x_t) + \innp{g_t / \gamma, x - x_t}$ and $u = x^\ast$ as the comparator point.

\begin{lemma}[Max-regret online-to-batch]\label{lemma:m-online-to-batch}
    For a convex set $\X$ and a set $\mathcal{U} \subseteq \X$, let $f:\X\to\R$ be a function and for a $\max$-regret game \cref{eq:m-regret} with query $x_t$, $t\in[T]$, let {$\ell_t:\X\to\R$} satisfy $\ell_t(u)\leq f(u)$ for all $u \in \mathcal{U}$, and $\ell_t(x_t)=f(x_t)$. Then we can compute $\widehat{x}_T$ from $(x_t)_{t \in [T]}$ such that
\[
    f(\widehat{x}_T)-f(u)\le\frac{\MaxRegret[\X,T]}{T}
\]
    for every $u \in \mathcal{U}$. For instance, if a minimizer $x^\ast$ of $f$ exists and $\mathcal{U} = \X$, we can set $u = x^\ast$.
\end{lemma}

\begin{proof}
    We use $\widehat{x}_T \in \argmin_{\{x_t : t \in [T]\}} f(x_t)$, which implies $\circled{1}$ below. If $f$ is convex, we can use  $\widehat{x}_T \defi \sum_{t=1}^T x_t / T$ and Jensen's inequality for $\circled{1}$. Use the assumption about $\ell_t$ in $\circled{2}$. For $\circled{3}$, we use the definition of $\max$-regret that compares with the best fixed action versus the fixed action $u$:
\[
\begin{aligned}
f(\widehat{x}_T)-f(u)
    \circled{1}[\le]
\sum_{t=1}^T\frac{f(x_t)}{T}-f(u)
    \circled{2}[\le]
\sum_{t=1}^T\frac{\ell_t(x_t)}{T}-\max_{t\le T}\ell_t(u)
    \circled{3}[\le]
\frac{\MaxRegret[\X,T]}{T}.
\end{aligned}
\tag*{\BlackBox}
\]
\renewcommand{\jmlrQED}{\par}
\end{proof}

\section{Reduction to {\color{black} a} convex-hull approximation problem}

We next upper bound the value of the game for affine losses by an expression involving only the sampled losses and their conditional means along the realized history.
{\color{black}Define the affine sample space and the associated affine pairing class by}
\begin{equation}\label{eq:affine_classes}
    \newtarget{def:affine-sample-space}{\cZ[\X,\H]}
    \defi \H\times[-\LambdaPair[\X,\H],\LambdaPair[\X,\H]], \qquad 
    \newtarget{def:affine-pairing-class}{\cA[\X,\H]}
    \defi \left\{
        \cZ[\X,\H]\ni(g,b)\mapsto\langle g,u\rangle+b:
        u\in\X
    \right\}.
\end{equation}
We similarly write $\cA[p,q^\ast][d]$ for this class when $(\X,\H)=(\ball{p}{d},\ball{q^\ast}{d})$, with $d$ made explicit to remove ambiguity.
Thinking about convex optimization, for a query $x_t\in\X$ with oracle output $g_t\in\H$, let $\ell_t(u)\defi\langle g_t,u\rangle+b_t$, where $b_t\defi f(x_t)-\langle g_t,x_t\rangle$. Since $f(0)=0$ and $f$ admits an $\H$-valued supporting functional at $0$, we have $b_t\in[-2\LambdaPair[\X,\H],0]$. Thus $Y_t\defi(g_t,b_t+\LambdaPair[\X,\H])$ belongs to $\cZ[\X,\H]$, and, for every $u\in\X$, the function $a_u(g,b)\defi\langle g,u\rangle+b$ in $\cA[\X,\H]$ satisfies $a_u(Y_t)=\ell_t(u)+\LambdaPair[\X,\H]$. This common shift leaves $\max$-regret unchanged.
{\color{black}Using \cref{lemma:m-online-to-batch}, this reduces the optimization problem to bounding the minimax value of the affine $\max$-regret game defined below.}

For a sequentially nested expression like the following, we use the compact notation:
\[
    (\sup_{r_t}\E_{Z_t\sim r_t})_{t=a}^b F(Z_{a:b})
    \defi
    \sup_{r_a}\E_{Z_a\sim r_a}\cdots
    \sup_{r_b}\E_{Z_b\sim r_b}F(Z_{a:b}),
\]
where $r_t$ is allowed to be a predictable kernel depending on the past $Z_{a:t-1}$.
We work in spaces where a minimax theorem works for the value of the $\max$-regret game $\Vmax[\X,\H,T]$, defined below.

\begin{assumption}\label{ass:minimax_theorem_consequence}
    Let $\X\subseteq E$, $\H\subseteq E^\ast$ be nonempty and convex, with $\H$ bounded and satisfying $\H = - \H$, $\LambdaPair[\X,\H]<\infty$, and such that the minimax theorem consequence $\circled{1}$ applies, where $r_t$ and $s_t$ are distributions over the affine loss class $\mathcal F:=\{x\mapsto\innp{g,x}+b:(g,b)\in\cZ[\X,\H]\}$ and over $\X$, chosen by the adversary and the player, respectively. The minimax value of the game is defined as:
\begin{align}
    \newtarget{def:max-regret-value}{\Vmax[\X,\H,T]}
    &:= \left(\inf_{x_t\in\X}\sup_{\ell_t\in\mathcal F}\right)_{t=1}^T
    \left\{
        \sum_{t=1}^T\ell_t(x_t)
        -\inf_{u\in\X}T\max_{t\in[T]}\ell_t(u)
    \right\}
    \label{eqn:max_regret_value}
    \\
    &\circled{1}[=]
    \left(
        \sup_{r_t}\inf_{s_t}
        \E_{\substack{x_t\sim s_t\\\ell_t\sim r_t}}
    \right)_{t=1}^T
    \left\{
        \sum_{t=1}^T\ell_t(x_t)
        -\inf_{u\in\X}T\max_{t\in[T]}\ell_t(u)
    \right\}.
    \notag
\end{align}
{\color{black}We further assume that, for every $\eta>0$, an $\eta$-minimizing action for $\inf_{x\in\X}\mathbb E_{\ell_t\sim r_t(\cdot\mid x_{1:t-1},\ell_{1:t-1})}\ell_t(x)$ can be selected measurably from the history and the announced loss distribution.}
\end{assumption}

The assumption above is mild, and in particular it holds for finite-dimensional compact {\color{black}convex} sets $\X,\H$, with $\H = - \H$. Indeed, see \citet{abernethy2009stochastic} that made this connection for the regular regret, between minimax regret and its analogous form after $\circled{1}$ by taking randomized strategies and applying the minimax theorem. 

The minimax assumption $\circled{1}$ holds more generally whenever there exists a Hausdorff locally convex topology $\tau$ for which $\X$ is compact and convex and every admissible affine loss $u\mapsto\langle g,u\rangle+b$ is $\tau$-continuous. In this case, the required minimax interchange follows from Sion's minimax theorem \citep{sion1958general}\footnote{Finitely supported adversarial laws suffice for the minimax interchanges, measurable approximate minimizers are explicitly assumed in \cref{ass:minimax_theorem_consequence}.}. This covers the standard settings used throughout the paper: closed and bounded feasible sets in finite-dimensional spaces, which are compact in the usual Euclidean topology \citep{abernethy2009stochastic,rakhlin2011online}; norm-compact feasible sets in arbitrary Banach spaces, such as closed uniformly bounded equicontinuous classes of continuous functions by the Arzel\`a-Ascoli theorem; and closed bounded convex feasible sets in reflexive Banach spaces, which are weakly compact, while continuous linear functionals, and hence admissible affine losses, are weakly continuous.
The $\eta$-optimal selection condition holds whenever $\X$ is separable in norm: conditional expected affine losses are norm-continuous, so one may choose the first $\eta$-optimal point in a fixed countable dense subset.

Some of the hypotheses in \cref{ass:minimax_theorem_consequence} are imposed only for simplicity and can be weakened. In particular, $\H$ need not be centrally symmetric or contain $0$ for the upper-bound argument: it is enough for $\H$ to be bounded. In finite-dimensional settings, more general loss classes may also be treated whenever the corresponding minimax interchange and existence of $\eta$-minimizing strategies are valid. In fact, \cref{thm:generalized_triplex_like_theorem} does not use convexity of the losses.

Motivated by the triplex-inequality framework of \citet{rakhlin2011online}, \textcolor{black}{and using the above assumption}, we remove the player's choices directly by minimizing conditional expected losses in the minimax-swapped game.

\begin{proposition}[Online reduction]\label{thm:generalized_triplex_like_theorem}\linktoproof{thm:generalized_triplex_like_theorem}
Under \cref{ass:minimax_theorem_consequence}, let $r_t(\cdot\mid\ell_{1:t-1})$ range over predictable distributions for admissible losses, and define $m_t(u)\defi\int\ell(u)\,r_t(d\ell\mid\ell_{1:t-1})$. Then
\begin{equation}\label{eq:online_reduced_value_game}
    \Vmax[\X,\H,T]
    \le
    \left(\sup_{r_t}\mathbb E_{\ell_t\sim r_t}\right)
    \!{\vphantom{\Big|}}_{t=1}^{T}
    \sup_{u\in\X}
    \left\{
        \sum_{t=1}^T m_t(u)
        -T\max_{t\le T}\ell_t(u)
    \right\}.
\end{equation}
\end{proposition}

Recall that the losses are $\innp{g_t,\cdot}+b_t$, with $(g_t,b_t)\in\cZ[\X,\H]$. The mixed-norm case with $p\le q$ is recovered using $g_t\in \ball{q^\ast}{d}$ and $b_t\in[-1,1]$. The offsets are important for the reduction from optimization but, as \cref{thm:fully-sequential-bound} shows, they only create a lower-order term.

{\color{black}
The convex-hull structure comes from the identity
\[
    \max_{t\le T}\ell_t(u)
    =
    \max_{\lambda\in\ProbSimplex{T}}\sum_{t=1}^T\lambda_t\ell_t(u)
    =
    \max_{\lambda\in\ProbSimplex{T}}
    \left\{
        \left\langle\sum_{t=1}^T\lambda_tg_t,u\right\rangle
        +\sum_{t=1}^T\lambda_tb_t
    \right\}.
\]
Thus the weights $\lambda$ form convex combinations of the observed affine samples $(g_t,b_t)$. In the next section, we bound the right-hand side of \eqref{eq:online_reduced_value_game} using sequential fat-shattering.
}

\section{Adaptive rate and sequential convex-hull approximation}
We tackle the problem yielded by \cref{thm:generalized_triplex_like_theorem}. Although the game itself involves affine losses, the associated linear class captures the order of the upper bound, and the offset coordinate will be shown to contribute only a low order term.  Consequently, the final upper bound is expressed in terms of the sequential fat-shattering dimension of the \emph{underlying linear class}.

Our final result can be compared in spirit to the one in \citep[Proposition 9]{rakhlin2015online}, where the authors obtained some bounds for the minimax value of regular online learning via Rademacher complexities and sequential fat-shattering expressions. However, the symmetrizations leading to the Rademacher complexity are possible in online learning due to all expressions involving sums. Bounding $\max$-regret requires a substantially different treatment. We now introduce some notation.

{
At each time $t=1,\ldots,T$, after observing the history
$
    y_{1:t-1}=((g_1,b_1),\ldots,(g_{t-1},b_{t-1})),
$
the adversary chooses a probability distribution
$P_t(\cdot\mid y_{1:t-1})$ supported on $\cZ[\X,\H]$. We write
$Y_t=(G_t,B_t)\sim P_t(\cdot\mid Y_{1:t-1})$. For a function
$f:\cZ[\X,\H]\to\mathbb R$, set
\[
    (P_tf)(y_{1:t-1})
    \defi
    \int_{\cZ[\X,\H]}f(g,b)\,dP_t(g,b\mid y_{1:t-1}),
\]
and abbreviate this quantity as $P_tf$ along a realized path. The predictable means of the two coordinates are\footnote{In the general Banach-space setting, $E^\ast$-valued integrals are understood weakly: $\mu_t\in E^\ast$ is defined by $\langle \mu_t,u\rangle=\int\langle g,u\rangle\,dP_t(g,b\mid y_{1:t-1})$ for every $u\in E$. We only use $\mu_t$ through such duality pairings.}
\[
    \mu_t(y_{1:t-1})
    \defi
    \int_{\cZ[\X,\H]}g\,dP_t(g,b\mid y_{1:t-1}),
    \qquad
    \beta_t(y_{1:t-1})
    \defi
    \int_{\cZ[\X,\H]}b\,dP_t(g,b\mid y_{1:t-1}).
\]
Each sample $Y_t=(G_t,B_t)$ defines the affine score
$u\mapsto\langle G_t,u\rangle+B_t$. Define the normalized one-sided approximation functional
\[
\newtarget{def:normalized-approximation-functional}{\Gf[\X,\H,T](Y_{1:T})}
\defi 
\sup_{u\in\X}\left\{\left\langle \frac1T\sum_{t=1}^T\mu_t,u\right\rangle+\frac1T\sum_{t=1}^T\beta_t
-\max_{\lambda\in\ProbSimplex{T}}\sum_{t=1}^T\lambda_t(\langle G_t,u\rangle+B_t)\right\}.
\]
For these affine losses, the conditional mean in \cref{thm:generalized_triplex_like_theorem} is $m_t(u)=\langle\mu_t,u\rangle+\beta_t$. Thus
\begin{equation}\label{eq:online_reduction_for_affine_losses}
    \Vmax[\X,\H,T]
    \le
    T
    \left(
        \sup_{r_t}\mathbb E_{Y_t\sim r_t}
    \right)\!{\vphantom{\Big|}}_{t=1}^{T}
    \Gf[\X,\H,T](Y_{1:T}),
\end{equation}
where each $r_t$ is supported on $\cZ[\X,\H]$. Every function in
$\cL[\X,\H]$ is bounded in
$[-\LambdaPair[\X,\H],\LambdaPair[\X,\H]]$. When $p\le q$, $(\X,\H)=(\ball{p}{d},\ball{q^\ast}{d})$ recovers the preceding
$\ball{q^\ast}{d}\times[-1,1]$ setup.
}
As discussed in the introduction, it is not clear how to nontrivially bound this process. The key insight is that further restricting the weights $\lambda$ to the {\em capped simplex}, $\CappedSimplex{T}=\{\lambda\in\ProbSimplex{T}:\lambda_t\leq 2/T,\,\,\forall t\in[T]\}$, provides a powerful tradeoff between greedy sample selection and averaging effects, resulting in an {\em offset empirical process} upper bound. This bound depends on the set of affine functions $\cA[\X,\H]$ and the subset $\EnvelopeClass$ of $1$-Lipschitz monotone functions in one dimension in \cref{def:offset-envelope-class}.
\begin{lemma}[Convex hull-distance to capped-simplex process]
\label{lem:deterministic-domination-general}\linktoproof{lem:deterministic-domination-general} %
Suppose that $\LambdaPair[\X,\H]=1$. For every realized path,
\begin{equation*}
    \Gf[\X,\H,T](y_{1:T})
    \le
    2
    \sup_{f\in\EnvelopeClass\circ(\tfrac12\cA[\X,\H])}
    \left[
        \frac1T
        \sum_{t=1}^T
        P_tf
        -
        \frac2T
        \sum_{t=1}^T
        f(y_t)
    \right].
\end{equation*}
\end{lemma}

With this upper bound at our disposal, we upper bound the resulting empirical process through a standard finite-class discretization and Dudley-style entropy bounds, along with a reduction from affine losses to linear losses, leading to our main result for the max-regret, in the following. Notice that it is the 2-factor offset that leads to the faster rates we report in our next results. For some discussion on why this offset yields faster rates, see \Cref{sec:technical-overview}.

\begin{theorem}\label{thm:fully-sequential-bound} \linktoproof{thm:fully-sequential-bound}
{
Under \cref{ass:minimax_theorem_consequence}, let $T\ge1$, set
$\Lambda\defi\LambdaPair[\X,\H]$, and write
$\sfat[\eta][\X,\H]\defi\sfat[\eta](\cL[\X,\H])$. There are
universal constants $C,c,\gamma>0$ such that
\[
\begin{aligned}
\sup_{\alpha>0}
\frac{\alpha}{2}
\left(
    \frac{\sfat[\alpha][\X,\H]}{T}
    \wedge 1
\right)
&\le \frac{\Vmax[\X,\H,T]}{T} \\
&\le
\inf_{0<\alpha\le\Lambda}
\Bigg\{C\alpha+\frac{C}{T}\Bigg[
\Lambda
\left(1+\log\frac{e\Lambda}{\alpha}\right)
+
\int_\alpha^{\Lambda}
\sfat[c\beta][\X,\H]
\log^\gamma\left(\frac{eT\Lambda}{\beta}\right)\,d\beta
\Bigg]\Bigg\}.
\end{aligned}
\]
}
Thus, if $p\le q$ and $T\le d$, then
\[
    \frac{1}{T^{\frac{1}{p}-(\frac{1}{q}-\frac{1}{2})\pospart}}
    \lesssim_{p,q}
    \frac{\Vmax[p,q^\ast,T]}{T}
    =
    \bigotildepl{p,q}{
        \frac{1}{T^{\frac{1}{p}-(\frac{1}{q}-\frac{1}{2})\pospart}}
    }.
\]
More generally, if $\sfat[\alpha](\cL[\X,\H])
=\TildeNotation{\Theta}((\Lambda/\alpha)^r)$ for some fixed $r\ge1$ and all $\alpha \in [\Lambda T^{-1/r}, \Lambda]$, then \begin{equation}\label{eq:general_sfat_rate}
    \frac{\Vmax[\X,\H,T]}{T}
=\TildeNotation{\Theta}[r]\Big(\frac{\Lambda}{T^{1/r}}\Big).
\end{equation}
\end{theorem}

In order to show the $(p, q)$-norm case, we use $\sfat$ bounds for the linear classes between finite-dimensional $p$-norm spaces $\cL[p,q^\ast][d]$, where $E=\mathbb R^d, \X=\ball{p}{d}, \H=\ball{q^\ast}{d}$, which we prove in the next subsection. For $p>q$, the last statement of \cref{thm:fully-sequential-bound}, along with the $\max$-regret online-to-batch conversion in \cref{lemma:m-online-to-batch}, recovers the known matching upper and lower bounds up to log factors from classical constructions and mirror descent \citep{guzman2015information,pmlr-v40-Guzman15}, while for $p \leq q$ and $T \leq d$, we obtain our main optimization result, nearly matching the lower bounds in \citep{guzman2015information,pmlr-v40-Guzman15}.
\begin{corollary}[First-order optimization bound]
\label{cor:first-order-optimization-sfat}\linktoproof{cor:first-order-optimization-sfat}
    Let $1\le p \le q\le\infty$, $d,T\in\mathbb N$, and ${\cal F}$ be the class of convex $1$-Lipschitz functions $\R^d \to \R$ in $\|\cdot\|_q$. There is a deterministic algorithm that for any $f\in {\cal F}$, after making $T$ queries to a first-order oracle, 
computes $\widehat{x}_T$ such that
\[
    f(\widehat x_T)-\min_{u\in \ball{p}{d}}f(u) = \bigotildepl{p,q}{
        \frac{1}{
            T^{\frac1p-(\frac1q-\frac12)\pospart}
        }
    }.
\]
\end{corollary}
Similarly, the general optimization oracle complexity upper bound derived from the last part of \cref{thm:fully-sequential-bound} and \cref{lemma:m-online-to-batch} gives a partial answer to \citet[Section 10.1.2, Q3]{sridharan2012learning}, who conjectured a similar upper bound depending on $\fat[\alpha](\cL[\X,\H])$ instead of $\sfat[\alpha](\cL[\X,\H])$, for any general value of $\fat[\alpha]$, instead of our polynomial decay; see \cref{sec:further_related_work}.

We note that since sfat is monotone in $\X$, the bounds in the first part of \cref{thm:fully-sequential-bound} are also monotone, so the upper bound does not increase if we evaluate it at a convex subset of $\X$ satisfying \cref{ass:minimax_theorem_consequence}. In particular, since we obtained $\bigotilde{T^{-1}}$ for the value of the max-regret game over $T$ for the $(1, 2)$-case, then the same upper bound is valid for $\X$ being the simplex and $\H$ still being the unit $\ell_2$-ball. The same applies to \cref{cor:first-order-optimization-sfat}.

\subsection{Sequential Fat-Shattering of Linear Classes between $\ball{p}{d}$ and $\ball{q^\ast}{d}$}

Here, we state our sfat bounds in the $(p,q)$-case that enable the second part of \cref{thm:fully-sequential-bound}. See \cref{app:sfat-profile} for a proof.
Set $L_\alpha\defi 1+\log\pospart(ed/\alpha)$.
Since every function in $\cL[p,q^\ast][d]$ takes values in
$[-\LambdaPair[p,q^\ast],\LambdaPair[p,q^\ast]]$, we have
$\sfat[\alpha](\cL[p,q^\ast][d])=0$ whenever
$\alpha>2\LambdaPair[p,q^\ast]$ since we cannot have function value separation greater than the range of the function class. Thus the next theorem is stated on the
nontrivial scale $0<\alpha\le2\LambdaPair[p,q^\ast]$. At
$p=q=\infty$, we interpret its second line as
$\sfat[\alpha](\cL[\infty,1][d])
\asymp d(1+\log\pospart(1/\alpha))$ on this same scale.
\begin{theorem}[Near-optimal sfat bounds]\label{cor:explicit_regimes} \linktoproof{cor:explicit_regimes}
    Assume $0<\alpha \le 2\LambdaPair[p,q^\ast] $. Then,
$\sfat[\alpha](\cL[p,q^\ast][d])$ lies, up to constants depending only on
$p,q$, between the following quantity and $L_\alpha$ times it:
\[
\begin{cases}
\displaystyle
\min\left\{
    \alpha^{-\left(\frac{1}{p}-(\frac{1}{q}-\frac{1}{2})\pospart\right)^{-1}},\;
    d\left(
        1+\log\pospart
        \frac{1}{\alpha d^{\frac{1}{p}-(\frac{1}{q}-\frac{1}{2})\pospart}}
    \right)
\right\},
& p<q,\\[4mm]
\displaystyle
\min\left\{
    \left(\frac{d^{1/q-1/p}}{\alpha}\right)^{\max\{q,2\}},\;
    d\left(
        1+\log\pospart
        \frac{d^{1/q-1/p}}{\alpha d^{1/\max\{q,2\}}}
    \right)
\right\},
& p\ge q.
\end{cases}
\]

\end{theorem}

The sfat upper bounds consist of covering the domain using smaller uniformly convex balls with known complexity, combining them to obtain sfat for the full ball, \cref{lem:seq_union_bound}, and trading off the radius of the balls with their sfat, so in the end we have 
$\sfat[\alpha](\cL[p,q^\ast][d]) \lesssim_r \inf_{R>0} \left[ \left(\frac{LR}{\alpha}\right)^{\max\{2,r\}} + \log \CoverN(\ball{p}{d},R\ball{r}{d}) \right]$ if $\ball{q^{\ast}}{d}\subseteq L \ball{r^{\ast}}{d}$ and $1 < r < \infty$, where $\newtarget{def:metric-covering-number}{\CoverN}(\cdot)$ is the covering number, cf. \cref{cor:barycentric_covering_bound}. The lower bounds are essentially the existing fat lower bounds \cref{eq:fat_shattering_values} amplified for the sequential case with a log factor via forcing a binary search over some coordinates.

\section{Further results}

\subsection{The iid rate and static convex-hull approximation to the mean}

\label{sec:iid-convex-hull}

We now specialize the online-reduction expression in \cref{thm:generalized_triplex_like_theorem} to iid linear losses. This yields a quantitative version of a classical convex-hull problem \citep{wendel1962problem}: we measure the distance from the mean of a general bounded distribution to the convex hull of its samples.

In this section, let $\X,\H\subseteq\mathbb R^d$ be origin-symmetric convex bodies. 
For the polar body and its associated norm, we write $ \newtarget{def:polar-body}{\Xpolar}
    \defi
    \left\{
        z\in\mathbb R^d:
        \sup_{u\in\X}
        |\langle u,z\rangle|
        \le1
    \right\}$, $
    \|z\|_{\Xpolar}
    \defi
    \sup_{u\in\X}
    |\langle u,z\rangle|$. We define $\operatorname{dist}_{\|\cdot\|}(x,S)=\inf_{y\in S}\|y-x\|$, and we use the shorthand notations $\operatorname{dist}_{p}=\operatorname{dist}_{\|\cdot\|_p}$, and  $\operatorname{dist}_{\Xpolar}=\operatorname{dist}_{\|\cdot\|_{\Xpolar}}$. Analogous terminology will be used for infinite-dimensional settings.
In this iid linear case, bounding the online-reduction term is exactly the
problem of approximating the mean of a distribution by the convex hull of its
samples in the polar norm $\|\cdot\|_{\Xpolar}$, as we show next. The iid property allows the bound to depend on $\fat$, rather than the $\sfat$ used in the adaptive case.
{
\begin{lemma}[iid specialization]
\label{lem:iid-online-reduction-identity}
\linktoproof{lem:iid-online-reduction-identity}
Let $P$ be supported on $\H$, with mean $\mu_P\defi\int g\,dP(g)$. If the
losses in the right-hand side of \eqref{eq:online_reduced_value_game} are the
iid linear losses $\ell_t(x)=\langle G_t,x\rangle$, with
    $G_{1:T}\stackrel{\mathrm{iid}}{\sim}P$, then the corresponding term,
divided by $T$ and before taking $\sup$ over $P$, satisfies
\[
    \frac{1}{T}
    \mathbb E_{G_{1:T}\stackrel{\mathrm{iid}}{\sim}P}
    \sup_{u\in\X}
    \left\{
        \sum_{t=1}^T m_t(u)
        -
        T\max_{t\le T}\langle G_t,u\rangle
    \right\}
    =
    \mathbb E_{G_{1:T}\stackrel{\mathrm{iid}}{\sim}P}
    \operatorname{dist}_{\Xpolar}
    \left(\mu_P,\conv\{G_{1:T}\}\right).
\]
\end{lemma}
}
Note that above $m_t(u)=\langle\mu_P,u\rangle$.
We now study worst-case bounds in this iid setting. Define the worst-case expected $\Xpolar$-distance from the convex hull of $T$ samples to their common mean:
\begin{equation}\label{eq:worst-case-iid-convex-hull-distance}
    \newtarget{def:worst-case-convex-hull-distance}{\ConvDist[\X,\H,T]}
    \defi
    \sup_{\operatorname{supp}(P)\subseteq\H}
    \mathbb E_{G_{1:T}\stackrel{\mathrm{iid}}{\sim}P}
    \operatorname{dist}_{\Xpolar}
    \left(\mu_P,\conv\{G_{1:T}\}\right).
\end{equation}

In this case, we can obtain upper and lower bounds depending on $\fat$.

{
\begin{theorem}[iid convex-hull approximation] \label{thm:iid-centered-convex-hull}\linktoproof{thm:iid-centered-convex-hull} Let $T\ge2$, set $\Lambda\defi\LambdaPair[\X,\H]$, and write
$\fat[\eta][\X,\H]\defi\fat[\eta](\cL[\X,\H])$. There are universal
constants $C,c,\gamma>0$ such that
\[
\begin{aligned}
    \sup_{\alpha>0}
    \frac{\alpha}{8}
    \Big(
        \frac{\fat[\alpha][\X,\H]}{T}
        \wedge 1
    \Big)
    \le
    \ConvDist[\X,\H,T]
    \le
    \inf_{0<\alpha\le\Lambda}
    \Bigg\{
        C\alpha
        +
        \frac{C}{T}
        \Bigg[
            \Lambda
            +
            \int_\alpha^{\Lambda}
            \fat[c\beta][\X,\H]
            \log^\gamma\!\left(
                \frac{eT\Lambda}{\beta}
            \right)
            \,d\beta
        \Bigg]
    \Bigg\}.
\end{aligned}
\]
\end{theorem}
}

Using the known values of $\fat$ for the $(p, q)$-norm case, cf. \cref{eq:fat_shattering_values} and \eqref{eq:fat-shattering-11}, we obtain that our previous bounds nearly match in this case.

\begin{corollary}%
\label{prop:iid-high-dimensional-instantiation}
\linktoproof{prop:iid-high-dimensional-instantiation}
Let $1\le p,q\le\infty$ and $T\ge2$. If $d\ge T$, then
\[
    \left(\frac{d}{T}\right)^{(\frac1q-\frac1p)\pospart}
    \frac{1}{T^{\frac1p-(\frac1q-\frac12)\pospart}}
    \lesssim_{p,q}
    \ConvDist[p,q^\ast,T]
    =
    \bigotildepl{p,q}{
        \left(\frac{d}{T}\right)^{(\frac1q-\frac1p)\pospart}
        \frac{1}{T^{\frac1p-(\frac1q-\frac12)\pospart}}
    }.
\]
If $d\le T$, then
\[
    \ConvDist[p,q^\ast,T]
    =
    \TildeNotation{\Theta}[p,q]\!\left(
        \frac{d^{1-\frac1p+(\frac1q-\frac12)\pospart}}{T}
    \right).
\]
\end{corollary}

Observe that, as we advanced earlier, when $d \leq T$, the decay is not exponential.
We now provide another lower bound, based on a coupon-collector argument. It is useful for $p=1$, $q \geq 2$ in the high-dimensional regime $d \geq T$ and regimes trivially close to it, where the corresponding lower bound in \cref{prop:iid-high-dimensional-instantiation} is $\bigomega{1 / T}$. In that case, it contains an additional logarithmic factor. In other cases the fat-shattering lower bound in \cref{prop:iid-high-dimensional-instantiation} dominates for $T$ greater than some constant.

\begin{proposition}[Coupon-collector lower bound for the iid case]\label{thm:coupon_lower_bound} \linktoproof{thm:coupon_lower_bound}
There exists a universal constant $c>0$ such that, for every
$p,q\in[1,\infty]$ and $T\ge3$, and every dimension $d\ge T$, we have
\[
    \ConvDist[p,q^\ast,T]
    \ge c\frac{\log T}{T}.
\]
The lower bound is witnessed by a mean-zero distribution supported on
$\ball{q^\ast}{d}$.
\end{proposition}
For any law $P$ supported on $\ball{q^\ast}{d}$, the law of
$(G-\mu_P)/2$ is centered and supported on the same ball. Against its iid
linear losses, with $G_{1:T} \sim P$, every player has $0$ expected loss. Thus the
duality calculation in \cref{lem:iid-online-reduction-identity}, followed by
taking the supremum over $P$, gives
\[
    \frac{\Vmax[p,q^\ast,T]}{T}
    \ge
    \sup_{\operatorname{supp}(P)\subseteq \ball{q^\ast}{d}}
    \mathbb E\operatorname{dist}_{p^\ast}
    \left(0,\conv\{(G_t-\mu_P)/2:t\in[T]\}\right)
    =\frac12\ConvDist[p,q^\ast,T].
\]
The factor $1/2$ ensures $(G_t-\mu_P)/2\in \ball{q^\ast}{d}$, since $\|G_t-\mu_P\|_{q^\ast}\le2$; translation invariance and homogeneity of distance give the equality.

The purpose of the current section is noting that a special case of our problem yields many interesting results, adjacent to already studied probability questions, like Wendel's theorem \citep{wendel1962problem}.  
The proofs for the adaptive and the iid cases are essentially the same, except that both at the end for the upper bounds proofs and for the lower bounds, we have to distinguish between using predictable trees and sequential coverings, leading to a sfat dependence, versus fixed settings and nonsequential covering numbers leading to a fat dependence.

\subsection{Optimization over Infinite-Dimensional Function Spaces}
\label{sec:function-space-examples}

In this section, we exploit the generality of our results and elaborate on three applications on infinite-dimensional spaces, in which the decision variable is a function on $\Omega\defi[0,1]^m$, by providing the optimization minimax rate via bounds on the corresponding fat-shattering dimensions.  Recall that $\H$ need not be convex, and in particular,
our results apply for the signed point and coordinate
evaluations used below.
The first application considers a Sobolev ball with $L_2$-bounded subgradients, the second a ball of Lipschitz functions with point-evaluation subgradients, and the third functions whose coefficients lie in an orthogonal-series representation belonging to an $\ell_p$ ball. The Sobolev example fits the reflexive Banach-space setting and is in fact norm compact in $L_2$, the Lipschitz example uses norm compactness in the nonreflexive space $C(\Omega)$, and the coefficient example has a norm-noncompact feasible ball whose weak compactness follows from reflexivity.
 To keep notation light, we reuse $\X$ and $\H$ for the feasible set and admissible subgradient set in each subsection.
We repeatedly use the last part of \cref{thm:fully-sequential-bound} and its implication of an optimization rate, after using \cref{lemma:m-online-to-batch}, and the lower bound of \citet{srebro2012convex} in terms of fat. That is, the minimax rates are
$
    \OptErr[\X,\H](T)
    =
    \TildeNotation{\Theta}
    \left(
       \LambdaPair[\X,\H] T^{-1/r}
    \right),
$
if $\sfat[\alpha](\cL[\X,\H]) \asymp \fat[\alpha](\cL[\X,\H])
=\TildeNotation{\Theta}((\LambdaPair[\X,\H]/\alpha)^r)$ for $r\ge1$ and all $\alpha \leq \LambdaPair[\X,\H]$.
{For the applications below, estimates at sufficiently small scales suffice since $\sfat[\alpha]$ is monotonous: $\sfat[c\beta](\cL[\X,\H])\le\sfat[c\alpha_0](\cL[\X,\H])<\infty$ for all $\beta \geq \alpha_0$, so the upper-tail integral in \cref{thm:fully-sequential-bound}, including its $1/T$ prefactor, contributes only $\TildeNotation{O}(\LambdaPair[\X,\H]/T)$ and leaves the stated rates unchanged.}
In each example below, it suffices to identify the sequential fat-shattering profile of the corresponding pairing class. We will use examples for which $\LambdaPair[\X,\H] \leq 1$.

A natural application of our results is constrained regression over infinite-dimensional function classes. The decision variable is a function $u$, while $\X$ encodes structural assumptions on the regression function (such as Sobolev smoothness, Lipschitz regularity, or sparse like representation) and $\H$ specifies the linear measurements through which the loss accesses $u$. Equivalently, $\H$ induces the observable metric $\sup_{g\in\H}|\langle g,u-v\rangle|$
which measures how distinguishable two candidate functions are under the available measurements. Thus the resulting optimization complexity depends jointly on the regularity class $\X$ and on the observation model $\H$.

\subsubsection{Optimization over Sobolev balls}

Fix $0<s\le m/2$ and let
\begin{equation*}
    \X
    \defi
    \left\{
        u\in H^s(\Omega):
        \|u\|_{H^s(\Omega)}
        \le1
    \right\}
    \subseteq
    L_2(\Omega),
    \qquad
    \H
    \defi
    B_{L_2(\Omega)}.
\end{equation*}
Here $H^s(\Omega)$ denotes the usual $L_2$-Sobolev space of smoothness $s$. For integer $s$, its norm controls the $L_2$ norms of the weak derivatives of $u$ up to order $s$.\footnote{For noninteger $s$, we can use the standard fractional Sobolev definition.} We normalize the norm so that $\|u\|_{L_2}\le\|u\|_{H^s}$. The Rellich-Kondrachov theorem implies that $\X$ is compact as a subset of $L_2(\Omega)$, while Cauchy-Schwarz gives $\LambdaPair[\X,\H]\le1$. Hence the result for a norm compact space applies.

\paragraph{Functional linear regression.}
A natural statistical application is functional linear regression with a smooth coefficient function, see \citep{Ramsay2005,Crambes2009}. Given functional covariates $Z_i\in L_2(\Omega)$ and responses $Y_i\in\R$, the goal is to find the best function $u \in \X$ that minimizes 
\begin{equation*}
    \frac1N\sum_{i=1}^N
    \ell\bigl(Y_i-\langle Z_i,u\rangle_{L_2}\bigr).
\end{equation*}
When $\|Z_i\|_{L_2}\le1$ and $\ell$ is convex and $1$-Lipschitz, this functional admits a subgradient in $B_{L_2(\Omega)}$ at every $u\in\X$, so the present result applies directly. This includes nonsmooth losses such as absolute-deviation and quantile regression. 
We now prove the shattering bounds that yield the tight minimax rate. 
\begin{proposition}
\label{prop:sobolev-function-space-example}\linktoproof{prop:sobolev-function-space-example}
For every $m\ge1$ and $0<s\le m/2$, for all sufficiently small $\alpha>0$,
\begin{equation*}
    \fat[\alpha]
    \left(
        \cL[\X,\H]
    \right)
    \asymp_{s,m}
    \sfat[\alpha]
    \left(
        \cL[\X,\H]
    \right)
    \asymp_{s,m}
    \alpha^{-\frac{2m}{m+2s}}.
\end{equation*}
Consequently,
\begin{equation*}
    \OptErr[\X,\H](T)
    =
    \TildeNotation{\Theta}[s,m]
    \left(
        T^{-\left(\frac12+\frac{s}{m}\right)}
    \right).
\end{equation*}
\end{proposition}

Thus every positive amount of Sobolev smoothness in this regime yields a polynomial improvement over the natural $T^{-1/2}$ rate obtained by the classical subgradient method. %
At the critical smoothness $s=m/2$, the rate reaches the $1/T$ scale. As can be seen by the exponent $s/m$, what matters for this class, is the amount of smoothness relative to the ambient dimension.

\subsubsection{Optimization over Lipschitz balls}

Work in $C(\Omega)$ endowed with the supremum norm, $\|u\|_{\infty}=\sup_{x\in \Omega}|u(x)|$, and let
\begin{equation*}
    \X
    \defi
    \left\{
        u\in C(\Omega):
        u(0)=0,\quad
        |u(x)-u(y)|
        \le
        \|x-y\|_\infty
        \ \text{for all }x,y\in\Omega
    \right\}.
\end{equation*}
For each $x\in\Omega$, let $\delta_x(u)\defi u(x)$ and let the point evaluations be $\H\defi\{0\}\cup\{\pm\delta_x:x\in\Omega\}$.
Thus $\cL[\X,\H]$ is simply the class of $0$-normalized $1$-Lipschitz functions on $\Omega$. Moreover, $\X$ is compact in the supremum norm by the Arzelà-Ascoli theorem and $\LambdaPair[\X,\H]\le1$.

\paragraph{Individual fairness.} We give an example application from fairness in machine learning. 
We consider a point $x\in[0,1]^m$ as the feature vector of an individual, and let $u(x)$ denote the score or decision assigned to that individual, such as the salary, credit score, or some priority score. A Lipschitz constraint then requires individuals with similar features to receive similar scores, this is a standard formulation that is called individual fairness \citep{Dwork2012}. Given a continuous baseline score $v$, such as a Bayes-optimal regression function, solving $\min_{u\in\X}\|u-v\|_\infty$ corresponds to finding an individually fair score with minimum worst-case deviation.%

\begin{proposition} \label{prop:lipschitz_bound}\linktoproof{prop:lipschitz_bound}
For all sufficiently small $\alpha>0$,
\begin{equation*}
    \fat[\alpha](\cL[\X,\H])
    \asymp_m
    \sfat[\alpha](\cL[\X,\H])
    \asymp_m
    \alpha^{-m}.
\end{equation*}
Consequently,
\begin{equation*}
    \OptErr[\X,\H](T)
    =
    \TildeNotation{\Theta}[m](T^{-1/m}).
\end{equation*}
\end{proposition}

In particular, when $m=1$, the minimax rate is $1/T$. When $m=2$, it is the $T^{-1/2}$ scale, while for $m>2$ this grows slower than the usual $T^{-1/2}$ rate.

\paragraph{Relation to the Sobolev example.} The two preceding examples encode different forms of regularity. The Sobolev constraint of $H^s$ controls smoothness through square-integrable derivatives and is paired above with the full $L_2$ unit ball of subgradients. The Lipschitz constraint instead controls pointwise variation and is paired with point evaluations. Sobolev embedding relates these classes for sufficiently large $s$, but using that embedding to control the Sobolev example would be lossy.

\subsubsection{Optimization over an $\ell_p$ Coefficients Ball}

Let $(\phi_j)_{j\ge1}$ be a real orthonormal basis of $L_2(\Omega)$ and fix $1<p<2$. Let
\begin{equation*}
    \X
    \defi
    \bigg\{
        u=\sum_{j\ge1}\theta_j\phi_j:
        \sum_{j\ge1}|\theta_j|^p\le1
    \bigg\},
\end{equation*}
where the ambient space is equipped with the coefficient norm $\|u\|
    \defi
    \big(
        \sum_{j\ge1}|\theta_j|^p
    \big)^{1/p}$.  For each $j\ge1$, define the coordinate functional $\langle g_j,u\rangle
    \defi
    \langle u,\phi_j\rangle_{L_2(\Omega)}
    =
    \theta_j$, 
and let $\H\defi\{g_j:j\ge1\}$.

Since $\ell_p\subseteq\ell_2$, these series converge in $L_2(\Omega)$. Under the coefficient map, the ambient space is isometric to $\ell_p$, and is therefore reflexive. Notice also that $\X$ is not  compact for the norm-topology %
as the functions $\phi_j$ belong to $\X$ and satisfy $\|\phi_j-\phi_k\|=2^{1/p}$ for $j\neq k$ which means that this sequence does not have a convergent sub-sequence.

\paragraph{Wavelet reconstruction.}
For $\Omega=[0,1]^2$, we may take $(\phi_j)$ to be an orthonormal wavelet basis. Sparse or $\ell_p$-structured wavelet representations are classical models for low-complexity signal and image representations, with applications to compression and denoising, see for instance \citet{Mallat2009}. Suppose that $v\in L_2(\Omega)$ is an observed image and that we seek a low-complexity approximation $u\in\X$. For a finite collection of wavelet coefficients, one natural criterion is to minimize $\max_{1\le j\le M}
    \left|
        \langle u,\phi_j\rangle_{L_2}
        -
        \langle v,\phi_j\rangle_{L_2}
    \right|$, the maximum discrepancy between the wavelet coefficients of the reconstruction and those of the observed image, subject to the $\ell_p$ constraint on the coefficients of $u$.

\begin{proposition}\label{prop:basis_bound}\linktoproof{prop:basis_bound}
For all sufficiently small $\alpha>0$,
\begin{equation*}
    \fat[\alpha](\cL[\X,\H])
    \asymp_p
    \sfat[\alpha](\cL[\X,\H])
    \asymp_p
    \alpha^{-p}.
\end{equation*}
Consequently, after symmetrizing the coefficient functionals,
\begin{equation*}
    \OptErr[\X,\H](T)
    =
    \TildeNotation{\Theta}[p](T^{-1/p}).
\end{equation*}
\end{proposition}

Unlike the Sobolev ball, which after diagonalization is a weighted $\ell_2$ ball with coefficients penalized by weights $w_j\asymp j^{s/m}$, this $\X$ is an unweighted $\ell_p$ ball, so it controls sparsity or compressibility of the coefficient sequence rather than imposing increasing decay with frequency.

\section{Conclusion}

We have solved the nonsmooth regime of the COLT open question \citep{pmlr-v40-Guzman15} by showing that the optimal minimax rate coincides, up to polylog factors, with the rate of the previously known lower bounds. 
The new substantial improvement occurs when $p<\min\{q,2\}$, while the $2 \le p < q$ regime can be solved optimally by mirror descent.
We expect that the techniques and results developed in this work will reach other applications. 
To the best of our knowledge, no previous general quantitative results for average convex-hull approximation to the mean were devised previously. An important open question for future work is that of finding efficient algorithms that realize the bounds found in this work, as well as extending the speed advantage to Hölder-smooth settings.

We also gave a partial answer to the open question in \citet[Section~10.1.2, Q3]{sridharan2012learning} about whether the complexity of convex Lipschitz optimization could be upper bounded in terms of fat-shattering dimension. We upper bounded it in terms of sequential fat-shattering dimension instead. These two quantities were only logarithmic factors away in our $(p, q)$-norm case, as well as in the Lipschitz, Sobolev, and $L_2$ orthonormal basis examples.

\clearpage

\begingroup
\sloppy

\acks{
Crist\'obal Guzm\'an was partially funded by ANID FONDECYT 1251029 grant, and ANID Basal FB210017 National Center for Artificial Intelligence CENIA.
David Martínez-Rubio was funded by a fellowship of the Bank Foundation ``la Caixa'' (ID 100010434) with code LCF/BQ/PI25/12100033. He thanks OpenAI for free access to their models. 
Mathieu Molina received funding from the European Research Council (ERC) under the European Union's Horizon Europe program (grant agreement No. 101170373), as a postdoctoral fellow at Tel Aviv University.
}

\AI{
This work was elaborated in combination with ChatGPT/Codex, that helped in a few places after a highly interactive workflow. Many of the main ideas (definition of the max-regret as a tool for solving the problem, online reduction removing player's actions, use of the capped simplex for the stochastic processes bound, iid simplified case, structure of lower bounds, among others) were from the authors, while AI was used to accelerate computations, checks and elaboration of some proofs. The computation of sequential fat-shattering bounds was mostly done by AI, and then checked and rewritten. The majority of this AI use, except for minor later refinements, was done with GPT 5.5 and earlier models (the newest models when we worked on the main results), except for the infinite-dimensional results. For the bounds of the latter results, many of the ingredients were classical but dispersed in the literature. We used GPT 5.6 and 6 to identify and synthesize these ingredients and derive short self-contained proofs tailored to our setting. 
The authors checked and rewrote the proofs that were automated, simplifying the presentation and making the main ideas more transparent. In particular, they are solely responsible for the contents of this manuscript. 
}

\clearpage
\appendix
\crefaliasappendix

\section{Further Related Work}\label{sec:further_related_work}

\textbf{Prior open questions.}
\citet{pmlr-v40-Guzman15} asked for the large-scale minimax risk of black-box
convex minimization over $\ell_p$-balls when regularity is measured in a
different $\ell_q$-norm, and whether the rate of the standard
$\ell_q/\ell_q$ method can be improved by exploiting the
geometry of the smaller feasible set.  Our results answer this question at the
$\kappa=1$ endpoint.  The
case $\kappa\to1$ is identified there with the nonsmooth Lipschitz limit. 

Our first reduction in \cref{thm:generalized_triplex_like_theorem} uses a direct conditional-mean argument, motivated by the triplex framework of \citet{rakhlin2011online}, who give a general framework for non-standard
notions of regret in which the learner's losses and the benchmark may be
aggregated in ways other than addition. In our setting, the benchmark is a
maximum over losses. Ordinary online convex optimization, with regret against
the sum of losses, is already pinned to the $\sqrt T$ cumulative scale in
worst case. This is visible through the online convex optimization reduction
of \citet[Lemma~13]{rakhlin2015online}, the sequential-Rademacher upper bound
of \citet[Proposition~16]{rakhlin2015martingale}, and the matching lower-bound
phenomena in
\citet[Lemma~3 and Theorem~5]{cheng2025geometry}.  The sequential entropy step
uses the sequential covering and fat-shattering machinery of
\citet{rakhlin2015online}. In particular, their Proposition~9 gives a lower
bound in terms of sequential fat-shattering for supervised absolute loss.  Our
lower bound in terms of sequential fat-shattering %
is for the value of the $\max$-regret game and although its analysis differs from any prior work, the final result should be viewed as an analog of \citep[Proposition 9]{rakhlin2015online} for our setting, in terms of its shape.

The relationship between convex optimization and fat-shattering goes back to
the lower-bound construction
of \citet[Section~4.4.2]{nemirovski1983problem}.  \citet{srebro2012convex}
observed that this construction yields an oracle-complexity lower bound
controlled by the fat-shattering dimension of the associated linear class.
The open question in \citet[Section~10.1.2]{sridharan2012learning}
asks, in this spirit, whether the offline oracle complexity
can always be upper bounded by the fat-shattering dimension of the
associated linear class, and whether such a bound can be achieved by an
optimization algorithm. That is, in our notation, whether there is $c$ such that if $T$ is the minimum iteration count such that $\OptErr[\X,\H](T) \leq \epsilon$ then
$ T \leq \fat[c\varepsilon](\cL[\X,\H])$. Taken literally, such a bound cannot work in low
dimensions, as oracle lower bounds of order $d\log(1/\varepsilon)$ are
known \citep[Chapter~4]{nemirovski1983problem} while fat-shattering is at most $d$.  One may instead ask for the right statement up to
logarithmic factors. Our contribution to the question is to show exactly that bound when using the sequential fat-shattering dimension instead, in cases where it decays according to the polynomial rate in \cref{thm:fully-sequential-bound}.

\paragraph{Upper bounds and sequential fat-shattering.}
As explained before, the Triplex inequality of \citep{rakhlin2011online} was a source of inspiration for our upper bounds, after realizing the max-regret would be a useful object of study for this problem. In the nondual settings, upper bounds depended on sequential fat-shattering bounds for classes of linear functions in mixed geometries, which were not known before to the best of our knowledge.
The geometric part about characterizing sequential fat-shattering up to log factors in the $(p, q)$-norm case is close in spirit to classical entropy estimates for
$\ell_p^d\to\ell_s^d$ embeddings, including Sch\"utt's bounds and the
sparse-vector proofs surveyed in
\citet[Theorem~2(a), Section~3]{kossaczka2020entropy}.  Fat-shattering was
introduced by \citet{kearns1994efficient}, whereas \citet{rakhlin2010online} introduced the sequential fat-shattering dimension. The martingale complexity framework of \citet[Proposition~16]{rakhlin2015martingale} gives
$\sfat[\alpha](\cL[s,s^\ast][d])$ for $s\ge2$. To the best of our knowledge, the general $p,q$ case was not previously known. In general, $\sfat$ and $\fat$ can be very far apart: in the binary
case this is the familiar gap between Littlestone and VC dimensions \citep{Littlestone:1987}, and
\citet{wu2022sequential} exhibit corresponding separations between $\fat$ and $\sfat$ for specific hypothesis classes; namely, for sequential prediction of linear threshold functions under log-loss. For the linear classes studied in this work, by contrast, our bounds show that $\fat$ and $\sfat$ are the same up to logarithmic factors. 

\paragraph{Lower bounds and fat shattering.}
The systematic study of lower bounds in first-order oracle models begins with
Nemirovski and Yudin \citep{nemirovski1983problem}, who formalized
local oracles and developed resisting-oracle constructions for the dual case. A later observation of
\citet[Theorem~1]{srebro2012convex} was that the same
construction gives a lower bound of $\fat[2\epsilon](\cL[\X,\H])$ for the oracle complexity over $\X$ for convex functions with subgradients in $\H$.
The values of fat-shattering dimension were studied up to constants in \citet{mendelson2004shattering}, which gave tight upper and lower bounds in several regimes and partial estimates in others. %
\citet{guzman2015information} provided lower bounds for convex Lipschitz optimization for all $p, q$ and although it was not noted in that work, the construction automatically serves for showing tight lower bounds for fat-shattering for all $p, q$, up to constants when $(p,q)\ne(1,1)$, and logarithmic factors at $(1,1)$, improving over \citet{mendelson2004shattering} in their missing case. In fact, the tightness is implied by matching oracle complexity upper bounds from mirror descent for this case $q < 2 < p$, which provides a proof that the lower bound on fat-shattering is tight. Thus, even though the pieces were not previously put together, for $(p,q)\notin\{(1,1),(\infty,\infty)\}$ and on the nontrivial scales it is known that
\begin{equation}\label{eq:fat_shattering_values}
\fat[\alpha](\cL[p,q^\ast][d])
\asymp_{p,q}
\begin{cases}
\min\left\{
    \alpha^{-\frac{1}{\frac{1}{p}-(\frac{1}{q}-\frac{1}{2})\pospart}},d
\right\},
& p<q,\\[1mm]
\min\!\left\{
    \left(d^{1/q-1/p}/\alpha\right)^{\max\{q,2\}},d
\right\},
& p\ge q,\ (p,q)\notin\{(1,1),(\infty,\infty)\}.
\end{cases}
\end{equation}
At $(p,q)=(1,1)$, the results of
\citet[Corollary~2.3 and Theorem~3.1]{mendelson2005geometry}
give the sharp characterization
\begin{equation}\label{eq:fat-shattering-11}
    \fat[\alpha](\cL[1,\infty][d])
    \asymp
    \min\left\{d,\frac{\log(2+d\alpha^2)}{\alpha^2}\right\},
    \qquad 0<\alpha\le1.
\end{equation}
And finally, $\fat[\alpha](\cL[\infty,1][d]) = d$ if $0<\alpha\le2$ and $0$ otherwise.
In all cases, the minimum with $d$ is the finite-dimensional cutoff, cf.
\citet[Section 11]{anthony1999neural}. Note that a simple consequence of the definition of $\fat[\alpha]$ for any class, is that it increases when $\alpha$ decreases and it tends to the pseudodimension  of the class \citep[Section~11.3]{anthony1999neural} when $\alpha \to 0$, which is a real-valued function generalization of the VC-dimension.

\paragraph{Bundle methods.}

The following is a non exhaustive discussion on the bundle-method literature. To the best of our knowledge, the first bundle method is from \citet{kelley1960cutting-plane}.  \citet{drori2016optimal-kelley} study a modified version of it, rather than the original unstabilized method. Bundle-level methods attain the same optimal nonsmooth rate
\citep{lan2015bundle-level} and can adapt to smoothness and be universal. The proximal-bundle analysis of
\citet{diaz2023proximal-bundle} recovers this rate and adapts to additional
smoothness or growth assumptions. \citet{ben-tal2005non-euclidean} designed bundle methods for non-Euclidean geometries, matching the rates of mirror descent methods.

\paragraph{Approximation of a convex hull of samples to its mean.} %
The classical Wendel Theorem \citep{wendel1962problem} provides an exact formula for the probability that the origin lies in the convex-hull of $n$ vectors drawn i.i.d.~from a centrally symmetric distribution in general position and in particular there is a sharp transition at $\Theta(d)$ samples. We study a more quantitative margin version guarantee, in the different setting of considering general distributions.
A few, but limited, quantitative variants of Wendel's theorem have been studied in
specific cases, such as  \citet{liu2014probabilities},
who treat the uniform distribution and Euclidean distance. \citet{hayakawa2023estimating} obtain distribution-dependent quantitative extensions of Wendel's theorem, controlling the probability that a point lies within distance $\varepsilon$ of the convex hull of iid samples. Their results concern Euclidean geometries and primarily the post-dimensional regime.

Notable applications of the Wendel Theorem include learning theory, as we discuss below, and compressed sensing: \citet{Donoho:2009} use it to show sharp transitions on the recovery of sparse vectors under random linear measurements. There is a broad literature on Wendel-type containment and absorption probabilities
\citep{wagner2001continuous,kabluchko2020absorption,hayakawa2023estimating,Tikhomirov:2023}.
To the best of our knowledge, these works mostly concern the regime in which
the number of samples is at least the ambient dimension, whereas our bounds
are most informative in the early-time regime $T<d$ and provide a more robust margin guarantee depending on a norm of choice.

The ideas behind Wendel's Theorem have been systematically applied to the study of linear (and nonlinear) data separability, starting from the work of \citet{Cover:1965}. In this regard, quantitative versions of Wendel's theorem as studied in our work are directly related to margin conditions, which are relevant for machine learning and statistical physics. For example, the Gardner problem in statistical physics asks precisely for margin conditions on randomly drawn (Gaussian) vectors \citep{Gardner:1988,Shcherbina:2002,Shcherbina:2003,Stojnic:2013}. There are also well known connections between margin conditions in different norms and guarantees for offline and online learning \citep{Littlestone:1987,Salehi:2020}.

\section{Proofs for the Online-Learning Reduction}\label{app:online-reduction}

\begin{proof}\linkofproof{thm:generalized_triplex_like_theorem}
For each realized past, write $\Eprime$ for the conditional expectation
over independent draws $x_t'\sim s_t$ and $\ell_t'\sim r_t$:
\[
    \newtarget{def:online-conditional-expectation}{\Eprime} F(x_t',\ell_t')
    \defi
    \mathbb E_{\substack{x_t'\sim s_t\\\ell_t'\sim r_t}}
    \left[F(x_t',\ell_t')\mid x_{1:t-1},\ell_{1:t-1}\right].
\]
These draws are used only inside the conditional expectation; future
distributions depend on the realized, unprimed history.

Using \cref{ass:minimax_theorem_consequence} gives $\circled{1}$.
Adding and subtracting the conditional expectations and splitting
$\sup_u$ gives $\circled{2}$:
{\small
\begin{align*}
    \Vmax[\X,\H,T]
    &\circled{1}[=]
    \left(\sup_{r_t}\inf_{s_t}
    \mathbb E_{\substack{x_t\sim s_t\\\ell_t\sim r_t}}\right)_{t=1}^T
    \sup_{u\in\X}
    \left\{\sum_{t=1}^T\ell_t(x_t)-T\max_{t\in[T]}\ell_t(u)\right\}\\
    &\circled{2}[\le]
    \left(\sup_{r_t}\inf_{s_t}
    \mathbb E_{\substack{x_t\sim s_t\\\ell_t\sim r_t}}\right)_{t=1}^T
    \Bigg[
        \sum_{t=1}^T\left(\ell_t(x_t)-\Eprime\ell_t'(x_t')\right)
    +
        \sup_{u\in\X}\left\{
            \sum_{t=1}^T\Eprime
            \left(\ell_t'(x_t')-\ell_t'(u)\right)
        \right\}\\
    &\qquad+
        \sup_{u\in\X}\left\{
            \sum_{t=1}^T\Eprime\ell_t'(u)
            -T\max_{t\in[T]}\ell_t(u)
        \right\}
    \Bigg].
\end{align*}
}

We break this expression into three terms in the spirit of the Triplex
Inequality \citep[Theorem~1]{rakhlin2011online}, allowing different loss
aggregations for the player and comparator. Using linearity of expectation
and, at each infimum,
\[
    \inf_a[C_1(a)+C_2(a)+C_3(a)]
    \le
    \sup_a C_1(a)+\inf_a C_2(a)+\sup_a C_3(a),
\]
we obtain
{\small
\begin{align*}
    \Vmax[\X,\H,T]
    &\le
    \left(\sup_{r_t}\sup_{s_t}
    \mathbb E_{\substack{x_t\sim s_t\\\ell_t\sim r_t}}\right)_{t=1}^T
    \left\{\sum_{t=1}^T
        \left(\ell_t(x_t)-\Eprime\ell_t'(x_t')\right)
    \right\}\\
    &\quad+
    \left(\sup_{r_t}\inf_{s_t}
    \mathbb E_{\substack{x_t\sim s_t\\\ell_t\sim r_t}}\right)_{t=1}^T
    \sup_{u\in\X}
    \left\{\sum_{t=1}^T\Eprime
        \left(\ell_t'(x_t')-\ell_t'(u)\right)
    \right\}\\
    &\quad+
    \left(\sup_{r_t}\sup_{s_t}
    \mathbb E_{\substack{x_t\sim s_t\\\ell_t\sim r_t}}\right)_{t=1}^T
    \sup_{u\in\X}
    \left\{\sum_{t=1}^T\Eprime\ell_t'(u)
        -T\max_{t\in[T]}\ell_t(u)
    \right\}.
\end{align*}
}
The first term is zero by the tower property, since
$\Eprime\ell_t'(x_t')
=\mathbb E[\ell_t(x_t)\mid x_{1:t-1},\ell_{1:t-1}]$.
For the second term, fix $\eta>0$. Since $r_t$ is announced before
$s_t$ in the minimax-swapped game, choose $s_t$ as the pure strategy of
playing a measurable $\eta$-optimal point of
\[
    \inf_{x\in\X}
    \mathbb E_{\ell_t'\sim r_t(\cdot\mid x_{1:t-1},\ell_{1:t-1})}
    \ell_t'(x).
\]
Then $\Eprime(\ell_t'(x_t')-\ell_t'(u))\le\eta$ for every $u\in\X$
along every realized path, so the second term is at most $T\eta$.

In the third term, the payoff depends only on the chosen loss kernels and
sampled losses. Player actions only provide auxiliary randomness for choosing
future kernels. By backward induction, averaging over these actions cannot
exceed the best continuation, so their distributions can be replaced by
deterministic choices. For any fixed resulting strategy, those choices can
be reconstructed from the loss history. Hence the kernels may be written
as $r_t(\cdot\mid\ell_{1:t-1})$.  Finally, substituting the notation
$
    \Eprime\ell_t'(u)
    =
    \int\ell(u)\,r_t(d\ell\mid\ell_{1:t-1})
    =m_t(u)
$
and letting $\eta\downarrow0$ proves the claim.
\end{proof}

\section{Proof of \cref{thm:fully-sequential-bound}: Sequential Bounds}\label{app:seq-upper}

It will be useful for a tight characterization of the $\max$-regret to establish lower bounds in terms of the sequential fat-shattering dimension, as we show next.
\begin{proposition}[Lower bound]
\label{prop:lower-bound-sfat}
{
Under \cref{ass:minimax_theorem_consequence}, for every $\alpha>0$,
\[
    \frac{\Vmax[\X,\H,T]}{T}
    \ge
    \frac{\alpha}{2}
    \left(
        \frac{\sfat[\alpha](\cL[\X,\H])}{T}
        \wedge 1
    \right).
\]
}
\end{proposition}

\begin{proof}
Fix $\alpha>0$ and set
\[
    m\defi\min\{\sfat[\alpha](\cL[\X,\H]),T\}.
\]
If $m=0$, the right-hand side is zero, and the identically zero adversarial
strategy shows that $\Vmax[\X,\H,T]\ge0$. Hence, suppose $m\ge1$.
Fix a depth-$m$ predictable tree $(g_t)_{t=1}^m$ over $\H$ and a
threshold tree $(s_t)_{t=1}^m$ that are sequentially $\alpha$-shattered
by $\cL[\X,\H]$. This class is bounded in
$[-\LambdaPair[\X,\H],\LambdaPair[\X,\H]]$. Since both children of every
shattered node are realizable,
\[
    -\LambdaPair[\X,\H]+\frac{\alpha}{2}
    \le s_t
    \le\LambdaPair[\X,\H]-\frac{\alpha}{2}.
\]
In particular, $\alpha/2\le\LambdaPair[\X,\H]$.

Let $\varepsilon_1,\ldots,\varepsilon_m$ be iid Rademacher signs and set
\[
    (G_t,B_t)
    \defi
    \begin{cases}
        \bigl(
            -\varepsilon_tg_t(\varepsilon_{<t}),
            \varepsilon_ts_t(\varepsilon_{<t})
        \bigr),
        &t\le m,\\
        (0,-\alpha/2),
        &m<t\le T,
    \end{cases}
    \qquad
    Y_t=(G_t,B_t).
\]
The symmetry assumption $\H=-\H$ ensures that $G_t\in\H$ for
$t\le m$, and by assumption $0\in\H$ on the padded rounds.
The bounds above give
$B_t\in[-\LambdaPair[\X,\H],\LambdaPair[\X,\H]]$ in both cases. Hence
$Y_t\in\cZ[\X,\H]$, so this is an admissible adaptive adversarial
strategy.

For $t\le m$, the sign $\varepsilon_t$ is independent of the past and has
mean zero; for $t>m$, $Y_t$ is deterministic. Therefore
\[
    \mu_t
    =
    \mathbb E[G_t\mid Y_{1:t-1}]
    =0,
    \qquad
    \beta_t
    =
    \mathbb E[B_t\mid Y_{1:t-1}]
    =
    \begin{cases}
        0,&t\le m,\\
        -\alpha/2,&m<t\le T.
    \end{cases}
\]

Using the witnesses $u_\varepsilon$ from
\cref{def:fat-shattering-dimensions}, for $t\le m$ we have
\[
    -\langle G_t,u_\varepsilon\rangle-B_t
    =
    \varepsilon_t
    \left(
        \langle g_t(\varepsilon_{<t}),u_\varepsilon\rangle
        -
        s_t(\varepsilon_{<t})
    \right)
    \ge
    \frac{\alpha}{2}.
\]
For $t>m$, the left-hand side equals $\alpha/2$. Moreover, against this
fixed adversarial strategy, every player strategy has expected cumulative loss
$\sum_{t=1}^T\beta_t=-(T-m)\alpha/2$. Therefore, after using \cref{ass:minimax_theorem_consequence}, we lower bound the minimax value by the one for our specific adversarial strategy and obtain:
\[
\begin{aligned}
    \frac{\Vmax[\X,\H,T]}{T}
    &\ge
    \mathbb E\sup_{u\in\X}
    \left\{
        -\frac{(T-m)\alpha}{2T}
        +
        \min_{t\in[T]}
        \bigl(-\langle G_t,u\rangle-B_t\bigr)
    \right\}\\
    &\ge
    -\frac{(T-m)\alpha}{2T}
    +
    \mathbb E
    \min_{t\in[T]}
    \bigl(-\langle G_t,u_\varepsilon\rangle-B_t\bigr)\\
    &\ge
    -\frac{(T-m)\alpha}{2T}
    +
    \frac{\alpha}{2}
    =
    \frac{m\alpha}{2T}.
\end{aligned}
\]
Substituting the definition of $m$ proves the claim.
\end{proof}

\color{black}
Our upper bounds make use of the following $\ell_\infty$ covering numbers, which in turn are 
controlled by $\fat$ for fixed designs \citep{alon1993scale} and by $\sfat$ for predictable trees \citep[Corollary~4]{rakhlin2010online}.

\begin{definition}[Sequential covers]\label{def:sequential-cover}
For a fixed depth-$T$ $\mathcal Z$-valued tree $\mathbf z$, an
    $\ell_\infty$ sequential $\alpha$-cover $\mathcal V_{\mathbf z}$ of $\mathcal G$  is a set of depth-$T$ trees over $\R$ such that
\[
    \forall h\in\mathcal G,\ \forall \varepsilon\in\{\pm1\}^T,\
    \exists \mathbf v\in\mathcal{V}_{\mathbf z}\ \text{ such that }\
    \max_{1\le t\le T}
    \left|
        h\bigl(z_t(\varepsilon_{<t})\bigr)
        -
        v_t(\varepsilon_{<t})
    \right|
    \le \alpha.
\]
\end{definition}
Note that above $\mathbf v$ may depend on both $h$ and $\varepsilon$, so it is a pathwise cover.
The size of a worst-case optimal cover over depth-$T$ trees is denoted by $\newtarget{def:sequential-covering-number}{\Ninfseq(\alpha,\mathcal G,T)} {\defi}\sup_{\mathbf z} \min_{\mathcal{V}_{\mathbf z}} \abs{\mathcal{V}_{\mathbf z}}
$.
We denote the nonsequential covering number by $\newtarget{def:covering-number}{\Ninf(\alpha,\mathcal G,z_{1:T})}$, which is the size of the worst-case optimal cover over trees where each level is constant.
\color{black}

The proof of \cref{thm:fully-sequential-bound} will use several lemmas that can be found after the proof.

\begin{proof}\linkofproof{thm:fully-sequential-bound}
{
The lower bound is \cref{prop:lower-bound-sfat}. For the upper bound, write $\Lambda\defi\LambdaPair[\X,\H]$ and $\widetilde\H
    \defi
    \Lambda^{-1}\H$.
Scaling every affine loss by $1/\Lambda$ gives
\begin{equation*}
    \Vmax[\X,\H,T]
    =
    \Lambda
     \Vmax[\X,\widetilde\H,T],
    \qquad
    \sfat[\alpha]
    \left(
        \cL[\X,\widetilde\H]
    \right)
    =
    \sfat[\Lambda\alpha]
    \left(
        \cL[\X,\H]
    \right).
\end{equation*}
It therefore suffices to prove the normalized upper bound when $\Lambda=1$, which we assume below.
In this normalization, recall the body-indexed version of the offset
envelope class from \cref{def:offset-envelope-class}:
\[
    \EnvelopeClass\circ(\tfrac12\cA[\X,\H])
    \defi
    \left\{
        (g,b)\mapsto
        \phi\!\left(\frac{\langle g,u\rangle+b}{2}\right):
        u\in\X,\ \phi\in\EnvelopeClass
    \right\}.
\]
Fix $0<\alpha\le1$. We obtain, for
every adaptive adversarial strategy $\pi$:
\begin{align*}
\begin{aligned}
    \mathbb{E}_{\pi}[\Gf[\X,\H,T](Y_{1:T})]
    &\circled{1}[\le] 
    2 \, \mathbb{E} _\pi \sup_{f\in\EnvelopeClass\circ(\tfrac12\cA[\X,\H])}
    \left[
        \frac1T\sum_{t=1}^T P_tf - \frac2T\sum_{t=1}^T f(Y_t)
    \right] \\
    &\circled{2}[\le]
    C\alpha + \frac{C}{T} \left[ 1+ \int_\alpha^1 \log \Ninfseq(c_{\mathrm{off}}\beta,\EnvelopeClass\circ(\tfrac12\cA[\X,\H]),2T) \,d\beta \right]  \\
    &\circled{3}[\le]
    C\alpha + \frac{C}{T}
    \left[
        1+ \int_\alpha^1 \log \Ninfseq(c_{\mathrm{off}}\beta/4,\cL[\X,\H],2T) \,d\beta + \frac{17\log 5}{c_{\mathrm{off}}} \int_\alpha^1\frac{d\beta}{\beta}
    \right] \\
&\circled{4}[\le]
    C\alpha + \frac{C}{T} \left[ 1+ \log\!\frac{e}{\alpha} + \int_\alpha^1 \sfat[c\beta](\cL[\X,\H]) \log^\gamma\!\left(\frac{eT}{\beta}\right) \,d\beta
    \right].
\end{aligned}
\end{align*}
    Where above, $\circled{1}$ uses \cref{lem:deterministic-domination-general}, and $\circled{2}$ uses \cref{prop:sequential-offset} with
$\mathcal Z=\cZ[\X,\H]$ and
$\mathcal C=\EnvelopeClass\circ(\tfrac12\cA[\X,\H])$, writing
    $c_{\mathrm{off}}$ for the scale constant in that proposition. We obtain $\circled{3}$ using
    \cref{lem:affine-offset-free-entropy} to make $\sfat$ of the linear class $\cL[\X,\H]$ appear as opposed to the one of the affine class, decreasing $c_{\mathrm{off}}$ if necessary so it is in $(0, 1]$. Finally, in $\circled{4}$, we used $\int_\alpha^1\beta^{-1}\,d\beta\le\log(e/\alpha)$, the bound of sequential covering numbers by $\sfat[c\beta]$ from
\citep[Corollary~1]{rakhlin2015online}, and we renamed $c_{\mathrm{off}}/4$ as the universal scale constant $c$.
The deterministic-domination and affine-entropy proofs use only the normalized
range bound and therefore apply verbatim to $(\X,\H)$.

The right-hand side does not depend on the adversarial strategy $\pi$. Hence,
for every $\varepsilon>0$, using \cref{eq:online_reduction_for_affine_losses} and choosing an $\varepsilon$-optimal history-dependent adversarial strategy
$\pi^\varepsilon$ for the nested problem yields
\begin{equation*}
    \frac{\Vmax[\X,\H,T]}{T}
    \le
    \mathbb E_{\pi^\varepsilon}
    \Gf[\X,\H,T](Y_{1:T})
    +
    \varepsilon.
\end{equation*}
Letting $\varepsilon\downarrow0$ and then taking the infimum over
$0<\alpha\le1$ proves the normalized upper bound. Multiplying it by
$\Lambda$, substituting $\alpha'=\Lambda\alpha$ and
$\beta'=\Lambda\beta$, and then dropping the primes gives the stated upper
bound for $(\X,\H)$.
}

\textbf{High-dimensional instantiation in the $(p, q)$-case}. When $T\le d$, the endpoint
$p=q=\infty$ follows by choosing $\alpha=1$ in the $\inf$ of the upper bound, combining the
endpoint case of \cref{cor:explicit_regimes} with \cref{prop:lower-bound-sfat}. For proving the remaining regimes, we study the general case where $\sfat[\alpha]$ decays like $\alpha^{-r}$ for $r \geq 1$.

In particular, for the remaining regimes with $p\le q$, write
$ 
    r
    \defi
    \left(
        \frac{1}{p}
        -
        \left(\frac{1}{q}-\frac{1}{2}\right)\pospart
    \right)^{-1},
$
so that
$
    T^{-1/r}
    =
    T^{-(1/p-(1/q-1/2)\pospart)}.
$

The class $\cL[p,q^\ast][d]$ is exactly the linear class whose sequential
fat-shattering dimension is characterized in \cref{cor:explicit_regimes}.
The upper side of that result gives
\[
    \sfat[\beta](\cL[p,q^\ast][d])
    \le
    C_{p,q}
    \beta^{-r}
    \left(1+\log\pospart\frac{ed}{\beta}\right),
    \qquad 0<\beta\le1.
\]
Substituting this estimate into the upper bound proved above gives
\[
\begin{aligned}
    \frac{\Vmax[p,q^\ast,T]}{T}
    \le
    \inf_{0<\alpha\le1}
    \Bigg\{
        C\alpha
        &+
        \frac{C}{T}\log\!\frac{e}{\alpha}
        +
        \frac{C}{T}
        +
        \frac{C_{p,q}}{T}
        \int_\alpha^1
        \beta^{-r}
        \left(1+\log\pospart\frac{ed}{\beta}\right)
        \log^\gamma\!\left(\frac{eT}{\beta}\right)
        \,d\beta
    \Bigg\}.
\end{aligned}
\]

First suppose $r=1$. Choose $\alpha=T^{-1}$. Then
\[
    1+\log\pospart\frac{ed}{\beta}\le C\log(edT),
    \qquad
    \log\!\left(\frac{eT}{\beta}\right)\le C\log(eT)
    \qquad(T^{-1}\le\beta\le1),
\]
and therefore for a constant $C_{p, q, \gamma}$:
\[
\begin{aligned}
    \frac{\Vmax[p,q^\ast,T]}{T}
    &\le
    \frac{C}{T}
    +
    \frac{C}{T}\log(eT)
    +
    \frac{C_{p,q,\gamma}}{T}
    \log(edT)
    \int_{T^{-1}}^1
    \frac{1}{\beta}
    \log^\gamma\!\left(\frac{eT}{\beta}\right)
    \,d\beta
    \\
    &\le
    \frac{C}{T}
    +
    \frac{C}{T}\log(eT)
    +
    \frac{C_{p,q,\gamma}}{T}
    \log(edT)
    \log^{\gamma+1}(eT) = \bigotildel{\frac{1}{T^{1 / r}}}.
\end{aligned}
\]
This is the asserted bound when $r=1$.
Now suppose $r>1$. If $T=1$, the claim is trivial after increasing the
constant, so assume $T\ge2$.  For every $\alpha\ge T^{-1}$,
\[
    \int_\alpha^1
    \beta^{-r}
    \left(1+\log\pospart\frac{ed}{\beta}\right)
    \log^\gamma\!\left(\frac{eT}{\beta}\right)
    \,d\beta
    \le
    C_{r,\gamma}
    \alpha^{1-r}
    \log(edT)
    \log^\gamma(eT/\alpha).
\]
Thus
\[
\begin{aligned}
    \frac{\Vmax[p,q^\ast,T]}{T}
    &\le
    \inf_{T^{-1}<\alpha\le1}
    \left\{
        C\alpha
        +
        \frac{C}{T}\log\!\frac{e}{\alpha}
        +
        \frac{C}{T}
        +
        \frac{C_{p,q,r,\gamma}}{T}
        \alpha^{1-r}
        \log(edT)
        \log^\gamma(eT/\alpha)
    \right\} \\
    &\circled{1}[\leq]
    C_{p,q,r,\gamma}
    \left(
        \frac{
            \log(edT)\log^\gamma(eT)
        }{T}
    \right)^{1/r}
    +
    \frac{C}{T}\log(eT) \\
    &= \bigotildel{\frac{1}{T^{1 / r}}},
\end{aligned}
\]
where $\circled{1}$ uses
$
    \alpha
    =
    \min\left\{
        1,\,
        \left(
            \frac{
                \log(edT)\log^\gamma(eT)
            }{T}
        \right)^{1/r}
    \right\} > \frac{1}{T},
$
since if the minimum is attained at $1$, the displayed claim is trivial after
increasing the constant. Otherwise, this choice satisfies $\log(eT/\alpha)\le C_r\log(eT)$.

For the lower bound, it is enough that the first term in
\cref{cor:explicit_regimes} be active, up to constants, at some scale
$\alpha\asymp T^{-1/r}$; thus $T\le d$ is sufficient but not necessary.
Indeed, set $\alpha_T\defi T^{-1/r}$. At this scale, the first term in that
theorem equals $T$, while its saturated term equals $d$ when $T\le d$,
because $\alpha_Td^{1/r}=(d/T)^{1/r}\ge1$. Hence, for some
$c_{p,q}>0$,
\[
    \sfat[\alpha_T](\cL[p,q^\ast][d])
    \ge c_{p,q}T.
\]
By \cref{prop:lower-bound-sfat},
\[
    \frac{\Vmax[p,q^\ast,T]}{T}
    \ge
    \frac{\alpha_T}{2}
    \left(
        \frac{
            \sfat[\alpha_T](\cL[p,q^\ast][d])
        }{T}
        \wedge 1
    \right)
    \ge
    \frac{c_{p,q}\wedge1}{2T^{1/r}},
\]
which is the stated lower bound after renaming the constant.

More generally, suppose
$\sfat[\alpha](\cL[\X,\H])=\TildeNotation{\Theta}((\Lambda/\alpha)^r)$
for $\Lambda T^{-1/r}\le\alpha\le\Lambda$, for a fixed $r\ge1$.
Normalize $\Lambda=1$, set $\alpha_T=T^{-1/r}$, and take $0<c\le1$
as the scale constant in the upper bound. Use the cutoff $\alpha_T/c$
when it is at most $1$, so every scale $c\beta$ in the integral lies
in the assumed range. The same integral estimates give
$\TildeNotation{O}[r](T^{-1/r})$, including the logarithmic integral when $r=1$.
If $\alpha_T/c>1$, the cutoff $1$ gives $O(1)=O(T^{-1/r})$.
For the lower bound, $\sfat[\alpha_T]=\TildeNotation{\Theta}(T)$, so
\cref{prop:lower-bound-sfat} gives $\TildeNotation{\Omega}(T^{-1/r})$,
allowing for the logarithmic losses in the assumed lower bound.
Renormalizing $\Lambda$ proves \eqref{eq:general_sfat_rate}.
\end{proof}

Now we prove the remaining lemmas needed in the proof of \cref{thm:fully-sequential-bound}. In order to do that, we only require the next lemma to apply conditional pushforwards along a fixed realized adaptive path. However, we show something more general without assuming any joint sampling or coupling of the measures. Let $Q_1,\dots,Q_T$ be probability measures on
$[-1,1]$, with distribution functions $F_1,\dots,F_T$ and means
$
    m_t\defi \int z\,dQ_t(z).
$
Let $z_1,\dots,z_T\in[-1,1]$ be arbitrary points.  Define
\[
    \bar m\defi \frac1T\sum_{t=1}^T m_t,
    \qquad
    \bar F(a)\defi \frac1T\sum_{t=1}^T F_t(a),
    \qquad
    \widehat F(a)\defi \frac2T\sum_{t=1}^T\mathbf{1}_{\{z_t\le a\}}.
\]
Also define the capped simplex
\[
    \newtarget{def:capped-simplex}{\CappedSimplex{T}}
    \defi
    \left\{
        \lambda\in\ProbSimplex{T}:
        \lambda_t\le\frac2T
        \text{ for every }t
    \right\}.
\]

\begin{lemma}[One-dimensional capped domination]
\label{lem:one-dimensional-general}
With the notation above and $\EnvelopeClass$ as in
\cref{def:offset-envelope-class},
\[
    \bar m
    -
    \max_{\lambda\in\CappedSimplex{T}}
    \sum_{t=1}^T\lambda_t z_t
    \le
    \sup_{\phi\in\EnvelopeClass}
    \left[
        \frac1T\sum_{t=1}^T\int \phi\,dQ_t
        -
        \frac2T\sum_{t=1}^T\phi(z_t)
    \right].
\]
\end{lemma}

\begin{proof}
Let $\nu\defi 2T^{-1}\sum_{t=1}^T\delta_{z_t}$. The measure $\nu$ has total mass
$2$.  Maximizing $\sum_{t=1}^T\lambda_t z_t$ over
$\lambda\in\CappedSimplex{T}$ is equivalent to maximizing
$\int z\,d\theta(z)$ over all submeasures $\theta\le\nu$ with total mass $1$,
where $\theta$ being a submeasure of $\nu$ means that $\theta(A)\le\nu(A)$ for every measurable set
$A$.
The maximizing submeasure fills mass from the largest observed points downward.
If $G$ denotes its distribution function, then
\[
    G(a)=\bigl(\widehat F(a)-1\bigr)\pospart.
\]
Let $\bar Q\defi T^{-1}\sum_{t=1}^T Q_t$. Then $\bar Q$ has distribution
function $\bar F$ and mean $\bar m$.  For any probability measure $\gamma$ on
$[-1,1]$ with distribution function $G_\gamma$, we have
$
    \int z\,d\gamma(z)
    =
    1-\int_{-1}^1G_\gamma(a)\,da.
$
Therefore
\[
\begin{aligned}
    \bar m
    -
    \max_{\lambda\in\CappedSimplex{T}}
    \sum_{t=1}^T\lambda_t z_t
    &= \bar m-\int z\,d\theta(z)
    =
    \left[1-\int_{-1}^1\bar F(a)\,da\right]
    -
    \left[1-\int_{-1}^1G(a)\,da\right] \\
    &= \int_{-1}^1\bigl(G(a)-\bar F(a)\bigr)\,da
    \le
    \int_{-1}^1
    \bigl(
        \widehat F(a)-\bar F(a)-1
    \bigr)\pospart\,da,
\end{aligned}
\]
where the inequality follows from
$G(a)=(\widehat F(a)-1)\pospart$ and $\bar F(a)\ge0$. It remains to
identify the last integral.
Every $\phi\in\EnvelopeClass$ admits a representation
$
    \phi(z)=\int_{-1}^z h(a)\,da
$
for some measurable $h:[-1,1]\to[0,1]$.  Conversely, every such $h$ defines a
function in $\EnvelopeClass$.  Hence
\[
    \int \phi\,dQ_t
    =
    \int_{-1}^1 h(a)\bigl(1-F_t(a)\bigr)\,da,
    \qquad
    \text{and}
    \qquad
    \phi(z_t)
    =
    \int_{-1}^1 h(a)\mathbf{1}_{\{a\le z_t\}}\,da.
\]
Therefore
\[
    \frac1T\sum_{t=1}^T\int \phi\,dQ_t
    -
    \frac2T\sum_{t=1}^T\phi(z_t)
    =
    \int_{-1}^1
        h(a)
        \bigl(
            \widehat F(a)-\bar F(a)-1
        \bigr)
    \,da.
\]
Taking the supremum over measurable $h:[-1,1]\to[0,1]$ gives the positive part
integral, and therefore the claim.
\end{proof}

Now fix a realized path $y_t=(g_t,b_t)\in\cZ[\X,\H]$, $t=1,\ldots,T$, and the corresponding kernels $P_t(\cdot\mid y_{1:t-1})$. We are ready to provide a proof of \cref{lem:deterministic-domination-general}.

\begin{proof}\linkofproof{lem:deterministic-domination-general}
Fix $u\in\X$ and define $z_t^u
    \defi
    \langle g_t,u\rangle+b_t$. 
Its predictable mean is
\begin{equation*}
    m_t^u
    \defi
    \int_{\cZ[\X,\H]}
    \left(
        \langle g,u\rangle+b
    \right)
    \,dP_t(g,b\mid y_{1:t-1})
    =
    \langle \mu_t,u\rangle+\beta_t,
\end{equation*}
where we abbreviate $\mu_t=\mu_t(y_{1:t-1})$ and $\beta_t=\beta_t(y_{1:t-1})$. Since $\LambdaPair[\X,\H]=1$, we have $|\langle g_t,u\rangle|\le1$. Together with $b_t\in[-1,1]$, this gives $z_t^u\in[-2,2]$. Let $Q_t^u$ be the pushforward of $P_t(\cdot\mid y_{1:t-1})$ under
\begin{equation*}
    (g,b)
    \mapsto
    \frac{\langle g,u\rangle+b}{2}.
\end{equation*}
Then $Q_t^u$ is supported on $[-1,1]$ and has mean $m_t^u/2$. Therefore,
\begin{align*}
    \frac1T\sum_{t=1}^T m_t^u
    -
    \max_{\lambda\in\ProbSimplex{T}}
    \sum_{t=1}^T\lambda_t z_t^u
    &\le
    \frac1T\sum_{t=1}^T m_t^u
    -
    \max_{\lambda\in\CappedSimplex{T}}
    \sum_{t=1}^T\lambda_t z_t^u \\
    &=
    2
    \left[
        \frac1T\sum_{t=1}^T\frac{m_t^u}{2}
        -
        \max_{\lambda\in\CappedSimplex{T}}
        \sum_{t=1}^T\lambda_t\frac{z_t^u}{2}
    \right] \\
    &\circled{1}[\le]
    2
    \sup_{\phi\in\EnvelopeClass}
    \left[
        \frac1T\sum_{t=1}^T
        P_t
        \phi\left(
            \frac{\langle \cdot,u\rangle+\cdot}{2}
        \right)
        -
        \frac2T\sum_{t=1}^T
        \phi\left(
            \frac{\langle g_t,u\rangle+b_t}{2}
        \right)
    \right].
\end{align*}
Here $\circled{1}$ applies \cref{lem:one-dimensional-general} to $Q_1^u,\ldots,Q_T^u$ and $z_1^u/2,\ldots,z_T^u/2$. Taking the supremum over $u\in\X$ proves the claim.
\end{proof}

We first control the offset process for a finite class of pair-trees (trees that have a pair of numbers at each node), and then
apply this estimate scale by scale to sequential covers in the chaining argument.
\begin{lemma}[Finite-class offset estimate]\label{lem:finite-class-offset}
Let $b>0$, and let $\mathcal A$ be a finite nonempty family of nonnegative
pair-trees
$
    a
    =
    \left\{
        a_t^+(\epsilon_{1:t-1}),
        a_t^-(\epsilon_{1:t-1})
    \right\}_{t\le T}
$
with values in $[0,b]$.  For
\begin{equation}
    M_a(\epsilon)
    \defi
    \sum_{t=1}^T
    \left[
        a_t^{-\epsilon_t}(\epsilon_{1:t-1})
        -
        2a_t^{\epsilon_t}(\epsilon_{1:t-1})
    \right] \quad \text{we have} \quad
    \mathbb E_\epsilon
    \left[
        \sup_{a\in\mathcal A}
        \frac1T M_a(\epsilon)
    \right]\pospart
    \le
    \frac{4b}{T}
    \bigl(1+\log|\mathcal A|\bigr).
\end{equation}
\end{lemma}

\begin{proof}
Indeed, fix $a\in\mathcal A$.  Conditionally on $\epsilon_{1:t-1}$, write
$
    r=\frac{a_t^+(\epsilon_{1:t-1})}{b},
    s=\frac{a_t^-(\epsilon_{1:t-1})}{b}.
$
Then $r,s\in[0,1]$.  For a sufficiently small universal $\lambda>0$,
\[
    \frac12 e^{\lambda(s-2r)}
    +
    \frac12 e^{\lambda(r-2s)}
    \le 1.
\]
Indeed, for fixed $s$, the left-hand side is convex in $r$, so its maximum
over $r\in[0,1]$ is attained at $r=0$ or $r=1$.  For each of these two
choices, the resulting function of $s$ is again convex, so it is enough to
check the four corners of $[0,1]^2$.  The claim holds, for example, for
$\lambda=1/4$.  Therefore
$
    \mathbb E_\epsilon \exp(\lambda M_a(\epsilon)/b)\le 1
$.
Hence
$
    \mathbb P_\epsilon(M_a(\epsilon)>u)\le e^{-\lambda u/b}
$.
By a union bound,
\[
    \mathbb P_\epsilon
    \left(
        \sup_{a\in\mathcal A}M_a(\epsilon)>u
    \right)
    \le
    \min\{1,|\mathcal A|e^{-\lambda u/b}\}.
\]
Writing $N=|\mathcal A|$ and integrating at the cutoff
$u_0=4b\log N$ gives
\[
    \mathbb E_\epsilon\left[\sup_{a\in\mathcal A}M_a(\epsilon)\right]\pospart
    \le
    u_0+\int_{u_0}^\infty Ne^{-u/(4b)}\,du
    =4b(1+\log N).
\]
Dividing by $T$ proves the lemma.
\end{proof}

Assume that every $h\in\mathcal C$ is measurable and that $\mathcal C$ is
pointwise measurable: there is a countable $\mathcal C_0\subseteq\mathcal C$
such that every $h\in\mathcal C$ is the pointwise limit of a sequence in
$\mathcal C_0$. Thus the supremum below equals a countable supremum and is
measurable. Since $0\le h\le2$, the quantity inside the supremum lies in
$[-4,2]$, so it is integrable without any further assumption.
\begin{proposition}[Sequential one-sided offset bound]
\label{prop:sequential-offset}
Let $\mathcal Z$ be any measurable space and let $\mathcal C$ be a nonempty
pointwise-measurable class of measurable functions $\mathcal Z\to[0,2]$.
For every adaptive strategy of
kernels $P_t(\cdot\mid X_{1:t-1})$ supported on $\mathcal Z$, and every
$0<\delta\le 1$,
\[
\begin{aligned}
    &\mathbb E
    \sup_{h\in\mathcal C}
    \left[
        \frac1T\sum_{t=1}^T P_th
        -
        \frac2T\sum_{t=1}^T h(X_t)
    \right]
    \le
    4\delta
    +
    \frac{16}{T}
    +
    \frac{48}{T}
        \int_\delta^1
        \log
        \Ninfseq(\alpha/4,\mathcal C,2T)
        \,d\alpha.
\end{aligned}
\]
\end{proposition}

\begin{proof}
    We are going to use a sequential symmetrization technique, at the level of predictable trees. The inspiration is \citep{rakhlin2015martingale} which also considers symmetrization for different adaptive processes. At every node of the adaptive process, draw two independent children from the distribution chosen at that node, obtaining a random pair-tree
\[
    \mathbf Z
    =
    \left\{
        Z_t^+(\epsilon_{1:t-1}),
        Z_t^-(\epsilon_{1:t-1})
    \right\}_{t\le T}.
\]
For a sign path $\epsilon\in\{\pm1\}^T$, the selected point at time $t$ is
$Z_t^{\epsilon_t}(\epsilon_{1:t-1})$, and the tangent ghost point is
$Z_t^{-\epsilon_t}(\epsilon_{1:t-1})$.
Let $\mathcal F_{t-1}$ be the sigma-field generated by the past signs and the sampled pairs along the selected path up to time $t-1$, and write $P_t$ for the kernel evaluated at the selected history. Then
\[
    P_th
    =\mathbb E[h(Z_t^{\epsilon_t})\mid\mathcal F_{t-1}]
    =\mathbb E[h(Z_t^{-\epsilon_t})\mid\mathcal F_{t-1}].
\]
Conditionally on the entire selected path and the signs, each ghost point still has law $P_t$, since future selected points depend only on the selected history, not on the ghosts. The selected path has the law of $X_{1:T}$, so conditional Jensen's inequality gives:
\[
    \mathbb E\sup_{h\in\mathcal C}\sum_{t=1}^T(P_th-2h(X_t))
    \le
    \mathbb E_{\mathbf Z,\epsilon}\sup_{h\in\mathcal C}
    \sum_{t=1}^T
    \left[
        h(Z_t^{-\epsilon_t}(\epsilon_{1:t-1}))
        -
        2h(Z_t^{\epsilon_t}(\epsilon_{1:t-1}))
    \right].
\]
We average over the random sign and the two
terms always use different children of the same node. It remains to bound the right-hand side conditionally on $\mathbf Z$ for which we perform a chaining argument.

The left-hand side of the proposition is at most $2$. Hence the result is
immediate when $\delta\ge1/2$, and we assume below that $0<\delta<1/2$.
Fix the pair-tree $\mathbf Z$. Let
\begin{equation}\label{eq:choice_of_S}
    S\defi \max\{s\ge 0:2^{-s}\ge\delta\},
    \qquad
    \alpha_s\defi 2^{-s}.
\end{equation}
At scale $\alpha_s$, take a pair-tree sequential cover $\mathcal V_s$
of $\mathcal C$ at accuracy $\alpha_s/2$.  A pair-tree of depth $T$ can be
encoded as an ordinary tree of depth $2T$ by listing, at each time $t$, first
the ``$+$'' child and then the ``$-$'' child.  Therefore we may choose
$\mathcal V_s$ with
\begin{equation}\label{eq:cover_size_bound}
    \log|\mathcal V_s|
    \le
    \log
    \Ninfseq(\alpha_s/2,\mathcal C,2T).
\end{equation}
We also clip the cover trees to $[0,2]$, which does not increase the covering
error.

By \cref{def:sequential-cover}, the approximation guarantee is pathwise. Fix
$h\in\mathcal C$ and a sign path $\epsilon$.  For each scale $s$, choose
$v_s\in\mathcal V_s$ approximating $h$ along the path on both children.  Define
the notation $v_{s,t}^{+}$ for the cover value approximating
$h(Z_t^+(\epsilon_{<t}))$, that is, the plus child of $v_s$ at time $t$. Similarly $v_{s,t}^{-}$ for
$h(Z_t^-(\epsilon_{<t}))$. We now bound both quantities at once using the
notation $v_{s,t}^{\pm}$. By definition of the covers at scales $\alpha_r/2$, along the chosen path we have
the approximations
\begin{equation}\label{eq:max_lb_bounds}
   0 \leq  c_{s,t}^{\pm}
    \defi 
    \max_{0\le r\le s} \bigl(v_{r,t}^{\pm}-\alpha_r/2\bigr)\pospart \leq h(Z_t^{\pm}) \leq c_{s, t}^{\pm} + \alpha_s \quad \text{ for all }s.
\end{equation}
Define increments by
\begin{equation} \label{eq:bounding_g_st}
    g_{0,t}^{\pm}\defi c_{0,t}^{\pm} \in [0, 2], \ \ \text{ and, for } s\ge 1,
    \ \  g_{s,t}^{\pm}
    \defi 
    \min\left\{
        c_{s,t}^{\pm}-c_{s-1,t}^{\pm},
        \alpha_{s-1}
    \right\} \in [0, \alpha_{s-1}].
\end{equation}
The truncation is needed only off the chosen path: it gives every possible
increment tree the uniform envelope $\alpha_{s-1}$ required by the
finite-class estimate in \cref{lem:finite-class-offset}. Along the chosen path it is harmless, since
\begin{equation}\label{eq:sum_of_increments}
    c_{s,t}^{\pm}-c_{s-1,t}^{\pm}
    \le
    h(Z_t^\pm)-c_{s-1,t}^{\pm}
    \le
    \alpha_{s-1} \quad \text{ and so } \quad
    c_{S,t}^{\pm}
    =
    \sum_{s=0}^S g_{s,t}^{\pm}.
\end{equation}
Using that by \cref{eq:choice_of_S} we have $\delta\le\alpha_S<2\delta$ and using \cref{eq:max_lb_bounds}, we have
\[
    c_{S,t}^{\pm} \leq
    h(Z_t^\pm)
    \le
    c_{S,t}^{\pm}+2\delta,
\] which along with \cref{eq:sum_of_increments}, yields $\circled{1}$ below:
\[
\begin{aligned}
    \mathbb E_\epsilon
    \sup_{h\in\mathcal C}
    \frac1T
    \sum_{t=1}^T
    \left[
        h(Z_t^{-\epsilon_t})
        -
        2h(Z_t^{\epsilon_t})
    \right]
    &\circled{1}[\le]
    2\delta
    +
    \sum_{s=0}^{S}
    \mathbb E_\epsilon
    \left[
        \sup_{g\in\mathcal G_s}
        \frac1T
        \sum_{t=1}^T
        \left[
            g_t^{-\epsilon_t}(\epsilon_{<t})
            -
            2g_t^{\epsilon_t}(\epsilon_{<t})
        \right]
    \right]\pospart \\
    &\circled{2}[\leq]
    2\delta
    +
    \frac{8}{T}
    \sum_{s=0}^{S}
        \alpha_s
        \left(
            1+
            \sum_{r=0}^s \log|\mathcal V_r|
        \right)
    \\
    &\circled{3}[\le]
    2\delta
    +
    \frac{16}{T}
    \left[
        1+
        \sum_{s=0}^{S}
        \alpha_s\log|\mathcal V_s|
    \right] \\
    &\circled{4}[\le]
    2\delta
    +
    \frac{16}{T}
    +
    \frac{48}{T}
        \int_\delta^1
        \log \Ninfseq(\alpha/4,\mathcal C,2T)\,d\alpha .
\end{aligned}
\]
where in $\circled{2}$, we used the envelopes in \cref{eq:bounding_g_st} and
\cref{lem:finite-class-offset}. For $s=0$, the envelope $2$ gives the
factor $8/T=8\alpha_0/T$. For $s\ge1$, the envelope
$\alpha_{s-1}=2\alpha_s$ gives $8\alpha_s/T$. The possible increment
pair-trees at level $s$
form a finite class $\mathcal G_s$ satisfying
$
    \log|\mathcal G_s|
    \le
    \sum_{r=0}^s\log|\mathcal V_r|
$. In $\circled{3}$ we used that $\sum_{s=0}^S\alpha_s \leq 2$ and that with  $A_r=\log|\mathcal V_r|$, we have
\[
    \sum_{s=0}^{S}
    \alpha_s
    \sum_{r=0}^s A_r
    =
    \sum_{r=0}^{S}
    A_r
    \sum_{s=r}^{S}\alpha_s
    \le
    \sum_{r=0}^{S}
    A_r
    \sum_{s=r}^{\infty}2^{-s}
    =
    2\sum_{r=0}^{S}\alpha_r A_r .
\]
Finally, in $\circled{4}$ we used the cover-size bound
\cref{eq:cover_size_bound} and monotonicity of the covering numbers. Indeed,
the term $s=0$ is at most twice the integral over $[1/2,1]$ with covering
scale $\alpha/4$, while each term $s\ge1$ is at most the corresponding
integral over $[\alpha_s,\alpha_{s-1}]$. Since $S\ge1$ and
$\delta\le\alpha_S$, their sum is at most three times the integral over
$[\delta,1]$.
Averaging over $\mathbf Z$ proves the proposition.
\end{proof}

\begin{proof}\linkofproof{cor:first-order-optimization-sfat}
Finite-dimensional compact balls satisfy
\cref{ass:minimax_theorem_consequence}. Apply
\cref{lemma:m-online-to-batch} to a minimax strategy for the $\max$-regret
game and output its average query. The rate is a consequence of
\cref{thm:fully-sequential-bound}.

We do not seek a computationally efficient strategy in this paper, but a valid
deterministic strategy can be obtained by taking sufficiently fine
discretizations of the iterated decision spaces $\ball{p}{d}$ and
$\ball{q^\ast}{d}$, using the Lipschitz continuity of $\max$-regret, and solving
    the resulting discrete problem coming from the deterministic definition of $\Vmax[p,q^\ast,T]$ by enumeration. This establishes first-order oracle efficiency, but the resulting computational time is exponential in $d$ and $T$.
\end{proof}

\subsection{Reducing the affine class to the linear one}
Here, we show that the covering number for the affine class is not much larger than the one for the linear class. Throughout this section, for simplicity, we assume that $\LambdaPair[\X,\H]=1$.

\begin{definition}[Offset envelope class]\label{def:offset-envelope-class}
Define
\begin{equation*}
    \newtarget{def:offset-envelope-class}{\EnvelopeClass}
    \defi
    \left\{
        \phi:[-1,1]\to[0,2]:
        \phi(-1)=0,\;
        \phi\ge0,\;
        \phi\text{ is nondecreasing and }1\text{-Lipschitz}
    \right\}.
\end{equation*}
The affine-composed class appearing in our analysis is defined as:
\begin{equation*}
    \EnvelopeClass\circ(\tfrac12\cA[\X,\H])
    \defi
    \left\{
        (g,b)
        \mapsto
        \phi
        \left(
            \frac{\langle g,u\rangle+b}{2}
        \right):
        u\in\X,\;
        \phi\in\EnvelopeClass
    \right\}.
\end{equation*}
\end{definition}
Since $|\langle g,u\rangle|\le1$ and $b\in[-1,1]$, every argument of $\phi$ belongs to $[-1,1]$. Therefore, every function in $\EnvelopeClass\circ(\tfrac12\cA[\X,\H])$ is well-defined and bounded in $[0,2]$.

\begin{lemma}[Affine offsets do not increase the sequential entropy]
\label{lem:affine-offset-free-entropy}
With $\EnvelopeClass$ and $\EnvelopeClass\circ(\tfrac12\cA[\X,\H])$ as in
\cref{def:offset-envelope-class}, for every $n\ge1$ and every $0<\alpha\le1$,
\begin{equation*}
    \log
    \Ninfseq
    \left(
        \alpha,
        \EnvelopeClass\circ(\tfrac12\cA[\X,\H]),
        n
    \right)
    \le
    \log
    \Ninfseq
    \left(
        \alpha/4,
        \cL[\X,\H],
        n
    \right)
    +
    \frac{17\log5}{\alpha}.
\end{equation*}
\end{lemma}

\begin{proof}
Fix a depth-$n$ tree
\begin{equation*}
    \mathbf z=(\mathbf g,\mathbf b),
    \qquad
    z_t(\varepsilon_{<t})
    =
    \left(
        g_t(\varepsilon_{<t}),
        b_t(\varepsilon_{<t})
    \right)
    \in
    \H\times[-1,1].
\end{equation*}
Let $\mathcal V$ be a sequential $\alpha/4$-cover of the linear class $\cL[\X,\H]$ on the $\H$-valued tree $\mathbf g$. Clip every value of every tree in $\mathcal V$ to $[-1,1]$; since the true linear evaluations lie in $[-1,1]$, this does not increase the covering error.

We also use a uniform $\alpha/4$-cover $\mathcal U$ of $\EnvelopeClass$ in
$\ell_\infty([-1,1])$, by real-valued functions on $[-1,1]$.  The
construction is elementary.  Put
$\eta=\alpha/8$, choose a grid
\[
    -1=t_0<t_1<\cdots<t_M=1,
    \qquad
    t_j-t_{j-1}\le \eta,
    \qquad
    M\le \left\lceil\frac{2}{\eta}\right\rceil\le \frac{17}{\alpha}.
\]
For a given $\phi\in\EnvelopeClass$, quantize each $\phi(t_j)$ to a nearest point
$a_j\in\eta\mathbb Z$, set $\psi(-1)=a_0$, and let $\psi$ be
constant equal to $a_j$ on $(t_{j-1},t_j]$ for $1\le j\le M$. Then
$\|\phi-\psi\|_\infty\le \eta/2+\eta\le\alpha/4$.  Moreover $a_0=0$,
because $\phi(-1)=0$, and the $1$-Lipschitz property gives
\[
    |a_j-a_{j-1}|
    \le
    |\phi(t_j)-\phi(t_{j-1})|+\eta
    \le 2\eta .
\]
Thus, after $a_{j-1}$ is fixed, $a_j$ has at most five possible
successors, namely the admissible grid values in
$a_{j-1}+\{-2\eta,-\eta,0,\eta,2\eta\}$. Since $a_0=0$ is fixed, only
$a_1,\ldots,a_M$ contribute choices. Hence one may choose $\mathcal U$ with
\[
    \log|\mathcal U|\le M\log 5\le \frac{17\log 5}{\alpha}.
\]
For each $v\in\mathcal V$ and $\psi\in\mathcal U$, define a real-valued tree on
$\mathbf z$ by
\[
    w_t^{v,\psi}(\varepsilon_{<t})
    \defi 
    \psi\!\left(
        \frac{
            v_t(\varepsilon_{<t})
            +
            b_t(\varepsilon_{<t})
        }{2}
    \right).
\]
The collection of all such trees has cardinality at most
$|\mathcal V||\mathcal U|$. Since $v_t\in[-1,1]$ and $b_t\in[-1,1]$, the argument of $\psi$ belongs to $[-1,1]$.

Now fix $u\in\X$, $\phi\in\EnvelopeClass$, and the corresponding function
\begin{equation*}
    f_{u,\phi}(g,b)
    \defi
    \phi
    \left(
        \frac{\langle g,u\rangle+b}{2}
    \right)
    \in
    \EnvelopeClass\circ(\tfrac12\cA[\X,\H]).
\end{equation*}
Fix a root-to-leaf path $\varepsilon$.  Since $\mathcal V$ covers the linear
class, choose a tree $v\in\mathcal V$ which approximates the linear function
$g\mapsto\langle g,u\rangle$ along this path:
\begin{equation*}
    \max_{t\le n}
    \left|
        \langle g_t(\varepsilon_{<t}),u\rangle
        -
        v_t(\varepsilon_{<t})
    \right|
    \le
    \alpha/4.
\end{equation*}
Choose $\psi\in\mathcal U$ with
$\|\phi-\psi\|_\infty\le \alpha/4$. Along this path, abbreviating
$g_t=g_t(\varepsilon_{<t})$, $b_t=b_t(\varepsilon_{<t})$, and
$v_t=v_t(\varepsilon_{<t})$,
\begin{align*}
    \left|
        \phi
        \left(
            \frac{\langle g_t,u\rangle+b_t}{2}
        \right)
        -
        \psi
        \left(
            \frac{v_t+b_t}{2}
        \right)
    \right|
    &\le
    \left|
        \phi
        \left(
            \frac{\langle g_t,u\rangle+b_t}{2}
        \right)
        -
        \phi
        \left(
            \frac{v_t+b_t}{2}
        \right)
    \right|
    +
    \|\phi-\psi\|_\infty \\
    &\le
    \frac12
    \left|
        \langle g_t,u\rangle-v_t
    \right|
    +
    \frac{\alpha}{4}
    \le
    \frac{3\alpha}{8}
    \le
    \alpha,
\end{align*}
for every $t\le n$. Hence
\begin{equation*}
    \Ninfseq
    \left(
        \alpha,
        \EnvelopeClass\circ(\tfrac12\cA[\X,\H]),
        \mathbf z
    \right)
    \le
    \Ninfseq
    \left(
        \alpha/4,
        \cL[\X,\H],
        \mathbf g
    \right)
    |\mathcal U|.
\end{equation*}
Taking logarithms and then the supremum over $\mathbf z$ proves the claim.
\end{proof}

\section{Proofs for the iid case}

\begin{proof}\linkofproof{lem:iid-online-reduction-identity}
{
Since the gradients are iid, the conditional
means in \eqref{eq:online_reduced_value_game} satisfy
$m_t(u)=\langle\mu_P,u\rangle$. Consequently, the corresponding term
divided by $T$ is
\[
\begin{aligned}
&\mathbb E\sup_{u\in\X}
    \left\{\langle\mu_P,u\rangle-\max_{t\le T}\langle G_t,u\rangle\right\}
\circled{1}[=]
\mathbb E\sup_{u\in\X}\inf_{\lambda\in\ProbSimplex{T}}
\left\langle \sum_{t=1}^T\lambda_tG_t-\mu_P,u\right\rangle
\\
&\quad\circled{2}[=]
\mathbb E\inf_{\lambda\in\ProbSimplex{T}}\sup_{u\in\X}
\left\langle \sum_{t=1}^T\lambda_tG_t-\mu_P,u\right\rangle
\circled{3}[=]
\mathbb E\inf_{\lambda\in\ProbSimplex{T}}
\left\|\sum_{t=1}^T\lambda_tG_t-\mu_P\right\|_{\Xpolar}
\\
&= \mathbb E\operatorname{dist}_{\Xpolar}
\bigl(\mu_P,\conv\{G_1,\ldots,G_T\}\bigr).
\end{aligned}
\]
Here $\circled{1}$ renames $u$ as $-u$ (recall that $\X$ is symmetric) and writes the minimum as an
infimum over the simplex, $\circled{2}$ is von Neumann's minimax theorem, and
$\circled{3}$ is the definition of the polar norm.
}
\end{proof}

The following proof is essentially the same as the one for \cref{thm:fully-sequential-bound}, except that we exploit the iid case for obtaining bounds depending on $\fat$ instead of $\sfat$. We are using the same structure of proof for simplicity, although the usage of the class $\EnvelopeClass$ could likely be avoided. 

\begin{proof}\linkofproof{thm:iid-centered-convex-hull}
The lower bound follows by taking the supremum over $\alpha>0$ in
    \cref{prop:iid-lower-bound-fat}.
For the upper bound, we may focus on zero-mean distributions. Indeed, for an arbitrary
distribution $P$ supported on $\H$, convexity gives $\mu_P\in\H$, and hence
$(G_i-\mu_P)/2\in\H$. The law of these centered random vectors has mean zero,
and
\[
    \operatorname{dist}_{\Xpolar}
    \left(\mu_P,\conv\{G_{1:T}\}\right)
    =
    2\operatorname{dist}_{\Xpolar}
    \left(0,\conv\{(G_i-\mu_P)/2:i\in[T]\}\right).
\]
Thus a uniform bound for mean-zero distributions implies the desired bound
after absorbing the factor of $2$ into $C$ and taking the supremum over
$P$.

Now let $\widetilde\H\defi\Lambda^{-1}\H$. Homogeneity gives
\[
\begin{aligned}
    \ConvDist[\X,\H,T]
    =
    \Lambda
    \ConvDist[\X,\widetilde\H,T],
    \qquad\text{ and }\qquad
    \fat[\beta](\cL[\X,\widetilde\H])
    =
    \fat[\Lambda\beta](\cL[\X,\H]).
\end{aligned}
\]
It therefore suffices to prove the claim when $\Lambda=1$. Assume this for
the remainder of the argument. Accordingly, fix a distribution $P$ supported
on $\H$ with $\mu_P=0$. By \cref{lem:iid-online-reduction-identity}, it
remains to prove the centered estimate. Using the linear class in
\cref{eq:linear_classes}, write
\[
    \EnvelopeClass\circ\left(\tfrac12\cL[\X,\H]\right)
    \defi
    \left\{
        g\mapsto\phi\!\left(\frac{\ell(g)}{2}\right):
        \ell\in\cL[\X,\H],\ \phi\in\EnvelopeClass
    \right\},
\]
and set $\widehat P_Tf=T^{-1}\sum_{t=1}^T f(G_t)$. Fix
$0<\alpha\le1$, let $\eta=\max\{\alpha,T^{-1}\}$, and write
$c_{\mathrm{off}}$ for the scale constant in the fixed-design chaining
bound. Below, for a function class $\mathcal G$ on a domain $\mathcal Z$, we write $\Ninf(\alpha,\mathcal G,2T)\defi\sup_{z_{1:2T}\in\mathcal Z^{2T}}\Ninf(\alpha,\mathcal G,z_{1:2T})$ for the worst-case fixed-design covering number. Then
\begin{align*}
\begin{aligned}
    &\mathbb E\operatorname{dist}_{\Xpolar}
    \bigl(0,\conv\{G_1,\ldots,G_T\}\bigr)
    \\
    &\quad\circled{1}[\le]
    2\mathbb E\sup_{f\in
    \EnvelopeClass\circ(\tfrac12\cL[\X,\H])}
    \bigl(Pf-2\widehat P_Tf\bigr)
    \\
    &\quad\circled{2}[\le]
    C\eta+
    \frac{C}{T}\left[
        1+
        \int_\eta^1
        \log \Ninf
        \bigl(c_{\mathrm{off}}\beta,
        \EnvelopeClass\circ(\tfrac12\cL[\X,\H]),2T\bigr)
        \,d\beta
    \right]
    \\
    &\quad\circled{3}[\le]
    C\eta+
    \frac{C}{T}\left[
        1+
        \int_\eta^1
        \log \Ninf
        \bigl(c_{\mathrm{off}}\beta/4,
        \cL[\X,\H],2T\bigr)
        \,d\beta
        +
        \frac{17\log5}{c_{\mathrm{off}}}
        \int_\eta^1\frac{d\beta}{\beta}
    \right]
    \\
    &\quad\circled{4}[\le]
    C\alpha+
    \frac{C}{T}\left[
        1+
        \int_\alpha^1
        \fat[c\beta](\cL[\X,\H])
        \log^\gamma\!\left(\frac{eT}{\beta}\right)
        \,d\beta
    \right].
\end{aligned}
\end{align*}
Here $\circled{1}$ applies
\cref{lem:deterministic-domination-general} with $b_t=0$ and $P_t=P$.
The fixed-design chaining argument in the proof of \cref{prop:iid-offset},
before bounding covering numbers by fat-shattering dimension, gives
$\circled{2}$. Note these are non-sequential covering numbers. For $\circled{3}$, the construction in the proof of \cref{lem:affine-offset-free-entropy}, using ordinary fixed-design covers and
    setting the offset coordinate to zero, gives
\[
    \log \Ninf
    \bigl(\beta,\EnvelopeClass\circ(\tfrac12\cL[\X,\H]),n\bigr)
    \le
    \log \Ninf
    \bigl(\beta/4,\cL[\X,\H],n\bigr)
    +\frac{17\log5}{\beta}.
\]
Finally, $\circled{4}$ uses \citet[Theorem~12.8]{anthony1999neural}, which
bounds fixed-design $\ell_\infty$ covering numbers by fat-shattering
dimension up to logarithmic factors, together with
$\eta\le\alpha+T^{-1}$ and monotonicity of the integral. We may take
$c\le1$ and $\gamma\ge1$. Choose $g_0\in\H$ and $u_0\in\X$ that
attain $\LambdaPair[\X,\H]=1$. By symmetry, $g_0$ is
$\beta$-shattered by $\cL[\X,\H]$ for
$0<\beta\le2$. Thus, when $\alpha\le1/2$, the integral over
$[1/2,1]$ absorbs the term $T^{-1}\log(e/\eta)$ from the last integral in
$\circled{3}$. When $\alpha>1/2$, that term is absorbed by $C\alpha$.
The deterministic and entropy arguments cited above use only convexity,
symmetry, and the normalized range bound, so their proofs apply verbatim to
$\cL[\X,\H]$. Taking the infimum over $\alpha$ proves the
normalized upper bound. Finally, multiply that bound by $\Lambda$, set
$\alpha'=\Lambda\alpha$ and $\beta'=\Lambda\beta$, and use the two
homogeneity identities above. Dropping the primes gives the stated bound for
$(\X,\H)$.
\end{proof}

The following proposition is an analogous result to \cref{prop:sequential-offset} for the iid case, which is the only part where the analysis of the iid and the adaptive cases slightly differ, which makes $\fat$ appear in the former case as opposed to $\sfat$ in the latter.
\begin{proposition}[iid centered one-sided offset bound]
\label{prop:iid-offset}
Let $\mathcal Z\subseteq\mathbb R^d$ be measurable, let $P$ be a
probability distribution supported on $\mathcal Z$, and let
$G_1,\ldots,G_T\stackrel{\mathrm{iid}}{\sim}P$. Let $\mathcal C$ be a
nonempty pointwise-measurable class of measurable functions
$\mathcal Z\to[0,2]$, and write
$
    Ph\defi\int h\,dP
$
and
$
    \widehat P_T h\defi T^{-1}\sum_{t=1}^T h(G_t)
$.
There are universal constants $C,c,\gamma>0$ such that, for every
$0<\delta\le1$,
\[
    \mathbb E\sup_{h\in\mathcal C}\bigl(Ph-2\widehat P_T h\bigr)
    \le
    C\delta+\frac{C}{T}\left[1+\int_\delta^1
    \fat[c\alpha](\mathcal C)\log^\gamma\!\left(\frac{eT}{\alpha}\right)
    \,d\alpha\right].
\]
\end{proposition}

\begin{proof}
For each $t$, draw an independent pair
$Z_t^+,Z_t^-\stackrel{\mathrm{iid}}{\sim}P$, and let
$\epsilon_1,\ldots,\epsilon_T$ be independent Rademacher signs. Ghost-sample
symmetrization and a random swap within each pair give
\[
\begin{aligned}
    &\mathbb E
    \sup_{h\in\mathcal C}
    \sum_{t=1}^T\bigl(Ph-2h(G_t)\bigr)
    \le
    \mathbb E_{Z,\epsilon}
    \sup_{h\in\mathcal C}
    \sum_{t=1}^T
    \left[
        h(Z_t^{-\epsilon_t})
        -
        2h(Z_t^{\epsilon_t})
    \right].
\end{aligned}
\]
Condition on the $2T$ points $Z_{1:T}^+,Z_{1:T}^-$. Unlike
in the adaptive case, these points do not depend on the sign history, so the
dyadic lower-approximation argument in the proof of
\cref{prop:sequential-offset} uses ordinary fixed-design $\ell_\infty$
covers. That argument and \cref{lem:finite-class-offset} yield
\[
\begin{aligned}
    &\mathbb E_\epsilon
    \sup_{h\in\mathcal C}
    \frac1T\sum_{t=1}^T
    \left[
        h(Z_t^{-\epsilon_t})
        -
        2h(Z_t^{\epsilon_t})
    \right]
    \le
    C\delta
    +
    \frac{C}{T}
    \left[
        1
        +
        \int_\delta^1
        \log \Ninf(c\alpha,\mathcal C,Z_{1:2T})
        \,d\alpha
    \right].
\end{aligned}
\]
Applying \citet[Theorem~12.8]{anthony1999neural} with sample size $2T$ and
range bound $2$ gives the following bound:
\[
    \log (\sup_{z_{1:2T}\in\mathcal Z^{2T}}\Ninf(\alpha,\mathcal C,z_{1:2T}))
    \le
    C\fat[c'\alpha](\mathcal C)
    \log^\gamma\!\left(\frac{eT}{\alpha}\right).
\]
Averaging over the pairs and renaming the universal scale constant proves the claim.
\end{proof}

\begin{proposition}[iid convex-hull lower bounds]
\label{prop:iid-lower-bound-fat}
{
Let $\X,\H$ be centrally symmetric convex bodies, and
let $T\ge2$ and $\alpha>0$. There is a distribution $P$ over $\H$ satisfying
\[
    \ConvDist[\X,\H,T]
    \ge
    \mathbb E_{G_{1:T}\stackrel{\mathrm{iid}}{\sim}P}
    \operatorname{dist}_{\Xpolar}
    \left(\mu_P,\conv\{G_{1:T}\}\right)
    \ge
    \frac{\alpha}{8}
    \left(
        \frac{\fat[\alpha](\cL[\X,\H])}{T}
        \wedge 1
    \right).
\]
}
\end{proposition}

\begin{proof}
{
Set
$
    m
    \defi
    \min\left\{
        \fat[\alpha](\cL[\X,\H]),T
    \right\}.
$
If $m=0$, take $P$ to be the point mass at $0$. Hence, suppose
$m\ge1$. By
\cref{fact:fat-shattering-evaluation-cubes}, there are
$g_1,\ldots,g_m\in\H$ whose evaluation image contains
$(\alpha/2)\ball{\infty}{m}$. Set $\eta\defi(2T)^{-1}$, and let $P$ assign
mass $\eta$ to every $g_i$ and the remaining mass to $0$. This is a
probability distribution because $m\le T$.

Given a sample from $P$, let $S\subseteq[m]$ contain the indices that
appear and set $M\defi m-|S|$. The fact that the image contains $(\alpha/2)\ball{\infty}{m}$ gives $u_S\in\X$ such that $\langle u_S,g_i\rangle=0$ for $i\in S$ and
$\langle u_S,g_i\rangle=\alpha/2$ otherwise. Thus every point in the
sampled convex hull has pairing zero with $u_S$, whereas
\[
    \left\langle u_S,\mu_P\right\rangle
    =
    \frac{\alpha\eta M}{2}.
\]
Consequently, the distance from $\mu_P$ to the sampled convex hull is at
least $\alpha\eta M/2$. Since
$\mathbb E M=m(1-\eta)^T\ge m/2$, taking expectations gives
\[
    \mathbb E
    \operatorname{dist}_{\Xpolar}
    \left(\mu_P,\conv\{G_{1:T}\}\right)
    \ge
    \frac{\alpha m}{8T}
    =
    \frac{\alpha}{8}
    \left(
        \frac{\fat[\alpha](\cL[\X,\H])}{T}
        \wedge 1
    \right).
\]
The definition of
$\ConvDist[\X,\H,T]$ now gives the final
claim.
}
\end{proof}

\begin{proof}\linkofproof{thm:coupon_lower_bound}
Fix $p,q\in[1,\infty]$ and $T\ge3$. After decreasing the universal constant, if necessary, it is
enough to consider large $T$.  Let $k\ge2$ be the largest integer satisfying
$k\log k\le T$, and set $d=T$.  Then $k\le d$ and
$k\asymp T/\log T$.  In the first $k$ coordinates of $\mathbb R^d$,
consider the centered simplex vertices
\[
    v_i \defi \frac12\left(e_i-\frac1k\mathbf 1\right),
\qquad i=1,\dots,k,
\]
where $e_i$ is the $i$-th standard basis vector in $\mathbb R^k$ and
$\mathbf 1=(1,\ldots,1)$.  Let $P$ be the uniform distribution on $\{v_1,\ldots,v_k\}$. $P$
has mean $0$, since
$
    \frac1k\sum_{i=1}^k v_i
    =
    \frac12\left(\frac1k\mathbf 1-\frac1k\mathbf 1\right)
    =
    0.
$
Moreover, the distribution $P$ is supported on
$\ball{q^\ast}{d}$, because
$
    \|v_i\|_{q^\ast}
    \le
    \|v_i\|_1
    =
    1-\frac1k
    <1.
$

Draw $G_1,\ldots,G_T\stackrel{\mathrm{iid}}{\sim}P$, and write
\[
    S\defi \{i\in[k]: v_i \text{ appears among } G_1,\ldots,G_T\}.
\]
If some index $j$ is missing, then every
$y\in\conv\{v_i:i\in S\}$ has $j$-th coordinate
$
    y_j=-\frac1{2k},
$
since $(v_i)_j=-1/(2k)$ for all $i\neq j$. Consequently, for every
$p\in[1,\infty]$,
\[
    \operatorname{dist}_{p^\ast}\bigl(0,\conv\{G_1,\ldots,G_T\}\bigr)
    =
    \inf_{y\in\conv\{G_1,\ldots,G_T\}}\|y\|_{p^\ast}
    \ge
    \frac1{2k}
    \qquad\text{on the event } \{S\neq[k]\}.
\]

It remains to show that this event has probability bounded below by a universal
constant.  Let $M\defi k-|S|$ be the number of missing support points.  If
$X_i$ is the indicator that $v_i$ is not sampled, then $M=\sum_i X_i$,
and the standard second-moment computation gives
\[
    \mathbb E M
    =
    k\Big(1-\frac1k\Big)^T,
    \qquad
    \mathbb E M^2
    =
    \mathbb E M
    +
    k(k-1)\Big(1-\frac2k\Big)^T .
\]
The maximality of $k$ gives $k\log k\le T<(k+1)\log(k+1)$.  Hence
$(1-\frac1k)^T\asymp k^{-1}$ and
$(1-\frac2k)^T\lesssim k^{-2}$, with universal constants.  Therefore
there are universal constants $0<c_1<c_2<\infty$ such that
$
    c_1\le \mathbb E M\le c_2,
    \mathbb E M^2\le c_2
$.
Combining the geometric observation above with the Paley-Zygmund inequality,
and using $k\asymp T/\log T$, yields
\[
    \mathbb E\operatorname{dist}_{p^\ast}
    \bigl(0,\conv\{G_1,\ldots,G_T\}\bigr)
    \circled{1}[\geq]
    \frac{1}{2k}\mathbb P(M>0)
    \circled{2}[\ge]
    \frac{(\mathbb E M)^2}{2k\mathbb E M^2}
    \ge
    \frac{c_0}{k} \geq \frac{c\log T}{T}.
\]
Here $\circled{1}$ is the geometric observation above and $\circled{2}$
is the Paley-Zygmund inequality. The constants $c_0,c>0$ are universal.
\end{proof}

\begin{proof}\linkofproof{prop:iid-high-dimensional-instantiation}
Set
\[
    a\defi\left(\frac1q-\frac1p\right)\pospart,
    \qquad
    s\defi
    a+\frac1p-\left(\frac1q-\frac12\right)\pospart.
\]
Then $\LambdaPair[p,q^\ast]=d^a$.
If $p=q=\infty$, the endpoint ordinary fat-shattering profile,
\cref{thm:iid-centered-convex-hull,prop:iid-lower-bound-fat}, and the trivial
diameter bound give
\[
    \ConvDist[\infty,1,T]
    =
    \TildeNotation{\Theta}\left(\frac dT\wedge1\right),
\]
which proves both claims at this endpoint. Hence, assume
$(p,q)\ne(\infty,\infty)$, so $s>0$.

First suppose $d\ge T$. The ordinary fat-shattering bounds
\cref{eq:fat_shattering_values,eq:fat-shattering-11} give, uniformly in the
remaining $(p,q)$,
\[
    \fat[\beta](\cL[p,q^\ast][d])
    \lesssim_{p,q} L_\beta
    \left(\frac{\LambdaPair[p,q^\ast]}{\beta}\right)^{1/s}
    \qquad
    (0<\beta\le\LambdaPair[p,q^\ast]).
\]
For the upper bound, choose $\alpha=\LambdaPair[p,q^\ast]T^{-s}$ in
\cref{thm:iid-centered-convex-hull}. Integrating the preceding estimate gives
(the exponent $\gamma+2$ allows for $L_{c\beta}\lesssim_{p,q}\log(eTd)$
on this interval and the extra logarithm when $s=1$)
\[
    \int_\alpha^{\LambdaPair[p,q^\ast]}
    \fat[c\beta](\cL[p,q^\ast][d])
    \log^\gamma\!\left(\frac{eT\LambdaPair[p,q^\ast]}{\beta}\right)
    \,d\beta
    \lesssim_{p,q}
    \LambdaPair[p,q^\ast]T^{1-s}\log^{\gamma+2}(eTd),
\]
and hence
$
    \ConvDist[p,q^\ast,T]
    =
    \TildeNotation{O}[p,q](\LambdaPair[p,q^\ast]T^{-s})
$.
For the lower bound, choose
$\alpha\asymp_{p,q}\LambdaPair[p,q^\ast]T^{-s}$ sufficiently small
that \cref{eq:fat_shattering_values,eq:fat-shattering-11} give
$\fat[\alpha](\cL[p,q^\ast][d])\ge c_{p,q}T$, with $0<c_{p,q}\le1$.
We only need a constant fraction of $T$, including when $d=T$.
Applying \cref{prop:iid-lower-bound-fat} therefore yields
\[
    \ConvDist[p,q^\ast,T]
    \ge \frac{\alpha}{8}
    \left(\frac{\fat[\alpha](\cL[p,q^\ast][d])}{T}\wedge1\right)
    \ge \frac{c_{p,q}\alpha}{8}
    \gtrsim_{p,q}\LambdaPair[p,q^\ast]T^{-s}.
\]
Finally,
\[
    \LambdaPair[p,q^\ast]T^{-s}
    =
    \left(\frac dT\right)^a
    \frac{1}{T^{\frac1p-(\frac1q-\frac12)\pospart}},
\]
which proves the first claim.

Suppose now that $d\le T$.
Write $\rho\defi\rhopqd[d]$, with $\rhopqd$ as in
\cref{eq:rho-pqd}. We first prove
\[
    \ConvDist[p,q^\ast,T]
    =
    \TildeNotation{\Theta}[p,q]\!\left(\frac{d\rho}{T}\right).
\]

For the upper bound, the ordinary fat-shattering estimates of
\citet{mendelson2004shattering} and
\citet{guzman2015information}, summarized in the related-work discussion
and supplemented by \eqref{eq:fat-shattering-11} at $(1,1)$,
give
\[
    \fat[\beta](\cL[p,q^\ast][d])
    \lesssim_{p,q}
    \min\left\{
        d,
        L_\beta\left(\frac{\LambdaPair[p,q^\ast]}{\beta}\right)^{1/s}
    \right\},
\]
and $\rho=\LambdaPair[p,q^\ast]d^{-s}$. Apply
\cref{thm:iid-centered-convex-hull} with $\alpha=\rho/T$. Splitting its
integral at $\rho$ gives
\[
\begin{aligned}
    \int_\alpha^{\LambdaPair[p,q^\ast]}
    \fat[c\beta](\cL[p,q^\ast][d])
    \log^\gamma\!\left(\frac{eT\LambdaPair[p,q^\ast]}{\beta}\right)
    \,d\beta
    &\lesssim_{p,q}
    d\rho\,\log^{\gamma+2}(eTd).
\end{aligned}
\]
Indeed, on $[\alpha,\rho]$ use $\fat[c\beta]\le d$, which contributes at
most $d\rho$ times the logarithmic factor. On
$[\rho,\LambdaPair[p,q^\ast]]$, use
$L_{c\beta}\lesssim_{p,q}\log(eTd)$. Integrating
$(\LambdaPair[p,q^\ast]/\beta)^{1/s}$ gives $O_{p,q}(\LambdaPair[p,q^\ast][1/s]\rho^{1-1/s}) =  O_{p,q}(d\rho)$ when $s<1$
and one additional logarithm when $s=1$. Since $0<s\le1$ and
$\rho=\LambdaPair[p,q^\ast]d^{-s}$, we also have
$\LambdaPair[p,q^\ast]\le d\rho$, so the theorem yields the claimed upper
bound.

For the lower bound, for all sufficiently large $d$, choose an
integer $m\asymp_{p,q}d$ with $m\le c_{p,q}d$ in
\cref{lem:evaluation_cubes}. Since $\rhopqd[m]$ is nonincreasing in
$m$, cf. \cref{eq:rho-pqd}, the resulting cube has radius
$R\gtrsim_{p,q}\rhopqd[d]=\rho$. For the finitely many remaining
dimensions, the one-coordinate cube gives the same conclusion after
decreasing the $(p,q)$-dependent constant. By
\cref{fact:fat-shattering-evaluation-cubes}, this gives
$\fat[2R](\cL[p,q^\ast][d])\ge m$. Since $m\le d\le T$,
\cref{prop:iid-lower-bound-fat} with $\alpha=2R$ gives
\[
    \ConvDist[p,q^\ast,T]
    \ge
    \frac{\alpha}{8}
    \left(
        \frac{\fat[\alpha](\cL[p,q^\ast][d])}{T}
        \wedge 1
    \right)
    \ge
    \frac{Rm}{4T}
    \gtrsim_{p,q}
    \frac{d\rho}{T}.
\]
Finally, substituting $m\asymp_{p,q}d$ into \cref{eq:rho-pqd} and using
\[
    \frac1q-\frac1{\max\{q,2\}}
    =
    \left(\frac1q-\frac12\right)\pospart
\]
gives the expression in the statement in both regimes.
\end{proof}

\section{Proof of the Sequential Fat-Shattering Bounds}\label{app:sfat-profile}

\begin{proof}\linkofproof{cor:explicit_regimes}
The $p<q$ and $p\ge q$ cases follow by combining the corresponding upper
bound in \cref{thm:norm-specific-upper-bound} and lower bound in
\cref{thm:norm-specific-lower-bound}, both proved below.
\end{proof}

Recall the linear class $\cL[p,q^\ast][d]$ in
\cref{eq:linear_classes}. %
For $1\le m\le d$, define
\begin{equation}\label{eq:rho-pqd}
\newtarget{def:rho-profile}{\rhopqd[m]}\defi
\begin{cases}
m^{-(1/p-(1/q-1/2)\pospart)},
& p < q,\\[1mm]
d^{1/q-1/p}m^{-1/\max\{q,2\}},
& p\ge q.
\end{cases}
\end{equation}
Set
$
L_\alpha\defi 1+\log\pospart\frac{ed}{\alpha}.
$

The proof of the upper bound consists of several parts. First, we recall that uniform convexity controls sfat for uniformly convex balls. For balls like $\ball{p}{d}$ that do not induce a uniformly convex norm, we cover them with other balls and compute sfat with \cref{lem:seq_union_bound} that allows to compute sfat for a union of sets of functions. Trading off radius of the covering balls vs their sfat, we can compute upper bounds that are near optimal. 

For the lower bound, a classical technique for fat shattering is amplified by sequential binary search. The latter
turns each of $m$ ordinary shattering directions into
$\log(1+\rhopqd[m]/\alpha)$ adaptive levels and explains the logarithmic growth after the ordinary dimension has saturated at $d$.

\subsection{Upper bound}

We first introduce a few classical notions.

\begin{definition}[Uniform convexity and function range]
\label{def:uniform-convexity-and-range}
Let $(E,\|\cdot\|)$ be a normed space, and  $K\subseteq E$ be convex. %
For $c>0$ and $\kappa\ge2$, a function
$\Phi:K\to\mathbb R$ is \emph{$(c,\kappa)$-uniformly convex on $K$}
with respect to $\|\cdot\|$ if
\[
    \frac{\Phi(u)+\Phi(v)}2
    \ge
    \Phi\!\left(\frac{u+v}{2}\right)
    +c\|u-v\|^\kappa,
    \qquad u,v\in K.
\]
Its range on $K$ is
$
    \operatorname{range}_K(\Phi)
    \defi
    \sup_{u\in K}\Phi(u)-\inf_{u\in K}\Phi(u).
$
\end{definition}

The next lemma says $\sfat$ can be bounded in terms of the constants of a uniformly convex potential in our set, if it exists. We note below the lemma, that this result is essentially known, since it is a consequence of concatenating two known results. For convenience, however, we instead provide a direct simple proof of the lemma in \cref{app:sfat-profile-known-proofs}.

\begin{lemma}[Uniformly convex localization]
\label{lem:uniformly-convex-localization}\linktoproof{lem:uniformly-convex-localization}
Let $(E,\|\cdot\|)$ be a real normed space and suppose
$\|g\|_*\le L$ for every $g\in\H\subseteq E^\ast$.
Let $K\subseteq E$ be convex, and let $\Phi:K\to\mathbb R$ be
$(c,\kappa)$-uniformly convex on $K$ with respect to $\|\cdot\|$
and range $D$. Then
\[
    \sfat[\alpha](\cL[K,\H])
    \le
    \frac{D}{c}\left(\frac{L}{\alpha}\right)^\kappa .
\]
\end{lemma}

For $1<r<\infty$, the classical uniform-convexity inequalities for
$\ell_r$ show that $R\ball{r}{d}$ admits a potential as in
\cref{lem:uniformly-convex-localization} with
$\kappa=\max\{2,r\}$ and $D/c\lesssim_r R^{\max\{2,r\}}$. 

The result in \cref{lem:uniformly-convex-localization} directly follows by concatenating the sequential Rademacher comparison and uniformly convex potential bound of \citet[Lemma~8 and Proposition~16]{rakhlin2015martingale}. Note that in their notation,
$\fat[\alpha]$ denotes the sequential dimension: Lemma~8 shows that
$\sfat[\alpha](\mathcal G)\ge n$ implies
$\alpha\le2\mathfrak R_n(\mathcal G)$, and Proposition~16 bounds this
sequential Rademacher complexity $\mathfrak R_n$ using a uniformly convex potential. For convenience, we instead provide a direct simple proof of the lemma in \cref{app:sfat-profile-known-proofs}.

We next combine localized classes. The statement is the real-valued,
scale-sensitive analogue of the multiple-union bound for Littlestone
dimension in \citet[Appendix~A.1, Proposition~23]{pmlr-v125-alon20a}. We
retain the short multiplicative-weights proof because it also fixes the
margin and tree conventions used below.

\begin{lemma}[Sequential union bound]\label{lem:seq_union_bound}\linktoproof{lem:seq_union_bound}
Let $\mathcal G_1,\ldots,\mathcal G_N$ be real-valued function classes on
the same domain, and set
$
\mathcal G\defi \bigcup_{i=1}^N \mathcal G_i .
$
If
$
\sfat[\alpha](\mathcal G_i)\le m
\text{ for every }i\in[N],
$
then
$
\sfat[\alpha](\mathcal G)
\le
C\bigl(m+\log N\bigr),
$
for a universal constant $C$.
\end{lemma}

\begin{proof}\linkofproof{lem:seq_union_bound}
For a subset of functions $\mathcal V\subseteq\mathcal G_i$, write
$\mathrm{rk}(\mathcal V)\defi\sfat[\alpha](\mathcal V)$, with
$\mathrm{rk}(\varnothing)=-1$. We will have an expert predicting on each subset, and once the subset is empty the corresponding expert
predicts arbitrarily. At a node with instance-threshold pair $(x,s)$, define
\[
    \mathcal V\pospart
    \defi
    \{f\in\mathcal V:f(x)\ge s+\alpha/2\},
    \qquad
    \mathcal V_-
    \defi
    \{f\in\mathcal V:f(x)\le s-\alpha/2\}.
\]
If $\mathrm{rk}(\mathcal V)=R$, the two children cannot both have rank at
least $R$, since otherwise attaching their depth-$R$ shattered trees below $(x,s)$
would give a depth-$(R+1)$ tree shattered by $\mathcal V$. Hence
\[
    \min\{\mathrm{rk}(\mathcal V\pospart),\mathrm{rk}(\mathcal V_-)\}
    \le R-1.
\]
The rank algorithm predicts the sign of the child with larger rank, breaking
ties arbitrarily, and updates to the child selected by the observed sign. On
an $\alpha$-realizable path, every mistaken prediction therefore decreases
the rank by at least one, and thus the rank algorithm for $\mathcal G_i$ makes at most
$\sfat[\alpha](\mathcal G_i)\le m$ mistakes.

Now run the $N$ rank algorithms as experts. Give every expert initial
weight $1$, predict by weighted majority, and after each round multiply the
weight of each mistaken expert by $1/2$. If the master makes a mistake, then
at least half of the current total weight is on mistaken experts, and hence
the total weight $Z_t$ is multiplied by at most $3/4$. After $M$ master
    mistakes, we have $\circled{2}$ below, where $\circled{1}$ holds if the sequence is $\alpha$-realizable by $\mathcal G_i$, since expert $i$
makes at most $m$ mistakes and thus its final weight is at least $2^{-m}$: 
\[
2^{-m}
\circled{1}[\le]
Z_T
\circled{2}[\le]
    N\left(\frac34\right)^M, \qquad \text{ and hence} \qquad
M
\le
\frac{m+\log_2 N}{-\log_2(3/4)}
\le
4(m+\log_2 N).
\]

Finally, let $n$ be the depth of an arbitrary tree shattered by
$\mathcal G$, and run the $N$ rank algorithms and their weighted-majority
master at its nodes.
The two edges below each node represent the two possible signs of the margin
constraint at that node. After the master predicts one sign, the adversary selects the edge
with the other sign. Repeating this rule constructs a root-to-leaf path on
which the master makes exactly $n$ mistakes. Since the tree is shattered,
    this complete path is witnessed by some function $f\in\mathcal G$, and
the union representation $\mathcal G=\bigcup_{i=1}^N\mathcal G_i$ places
that function in some $\mathcal G_i$. The mistake bound above therefore
applies to the constructed path and yields
$n\le4(m+\log_2N)$, proving the lemma.
\end{proof}

Covering the parameter body by local balls now costs only the logarithm of the covering number. This is the key argument for computing near optimal upper bounds on $\sfat[\alpha](\cL[p,q^\ast][d])$, as we simply take the infimum over $R$ over covering a ball with another type of ball of radius $R$, the latter coming from a uniformly convex geometry.

\begin{proposition}[Uniform convexity vs covering]\label{cor:barycentric_covering_bound}%
Let $1<r<\infty$ and suppose
$\ball{q^{\ast}}{d}\subseteq L \ball{r^{\ast}}{d}$. Then
\[
\sfat[\alpha](\cL[p,q^\ast][d])
\lesssim_r
\inf_{R>0}
\left[
\left(\frac{LR}{\alpha}\right)^{\max\{2,r\}}
+
\log \CoverN(\ball{p}{d},R\ball{r}{d})
\right].
\]
\end{proposition}

\begin{proof}%
Fix $R>0$, and cover $\ball{p}{d}$ by
\[
\ball{p}{d}\subseteq \bigcup_{i=1}^\CoverN(a_i+R\ball{r}{d}),
\qquad
N=\CoverN(\ball{p}{d},R\ball{r}{d}).
\]
Set
\[
C_i\defi \ball{p}{d}\cap(a_i+R\ball{r}{d}),
\qquad
\mathcal F_i\defi
\{\ball{q^\ast}{d}\ni g\mapsto \langle w,g\rangle:w\in C_i\}.
\]
Since $C_i\subseteq a_i+R\ball{r}{d}$, the class $\mathcal F_i$ is contained in
the corresponding local ball class. Moreover, shifting the ball by $a_i$ does not change $\sfat$. 
Thus, using
$
\cL[p,q^\ast][d]
\subseteq
\bigcup_{i=1}^N \mathcal F_i,
$
\cref{lem:uniformly-convex-localization} and \cref{lem:seq_union_bound}, we have
\[
\sfat[\alpha](\cL[p,q^\ast][d])
\lesssim_r
\left(\frac{LR}{\alpha}\right)^{\max\{2,r\}}
    +\log \CoverN(\ball{p}{d},R\ball{r}{d}) \quad \text{ for all } R > 0.
\]
\end{proof}

The next estimate is classical. The first one
is the packing-volume bound for finite-dimensional norm balls
\citep[Chapter~5]{pisier1989volume} and the second is the intermediate entropy regime of \citet{schutt1984entropy}. An explicit modern formulation is \citet[Theorem~2(a)]{kossaczka2020entropy}.

\begin{fact}[Covering numbers]
\label{prop:entropy-inputs}
For $1\le p,s\le\infty$ and $R>0$,
\begin{align}
    \log \CoverN(\ball{p}{d},R\ball{s}{d})
    &\lesssim_{p,s}
    d\left(1+\log\pospart\frac{d^{1/s-1/p}}R\right),
    \label{eq:volumetric-entropy-input}
    \\
    \log \CoverN(\ball{p}{d},R\ball{s}{d})
    &\lesssim_{p,s}
    R^{-1/(1/p-1/s)}\log\frac{Ced}{R},
    \qquad 1\le p<s\le2,\quad 0<R<1.
    \label{eq:schutt-entropy-input}
\end{align}
\end{fact}

\begin{theorem}[sfat upper bound]
\label{thm:norm-specific-upper-bound}\linktoproof{thm:norm-specific-upper-bound}
For every $1\le p,q\le\infty$, $(p,q)\ne(\infty,\infty)$, and
$\alpha>0$,
\[
    \sfat[\alpha](\cL[p,q^\ast][d])
    \lesssim_{p,q}
    L_\alpha
    \begin{cases}
    \displaystyle
    \min\left\{
        \alpha^{-\left(\frac{1}{p}-(\frac{1}{q}-\frac{1}{2})\pospart\right)^{-1}},
        \ d\left(
            1+\log\pospart
            \frac{1}{\alpha d^{\frac{1}{p}-(\frac{1}{q}-\frac{1}{2})\pospart}}
        \right)
    \right\},
    & p < q,\\[4mm]
    \displaystyle
    \min\left\{
        \left(\frac{d^{1/q-1/p}}{\alpha}\right)^{\max\{q,2\}},
        \ d\left(
            1+\log\pospart
            \frac{d^{1/q-1/p}}{\alpha d^{1/\max\{q,2\}}}
        \right)
    \right\},
    & p \geq q.
    \end{cases}
\]
At $p=q=\infty$,
$
    \sfat[\alpha](\cL[\infty,1][d])
    \lesssim
    d\left(1+\log\pospart\frac1\alpha\right).
$
\end{theorem}

\begin{proof}\linkofproof{thm:norm-specific-upper-bound}
When $p<q$, every function in the class takes values in $[-1,1]$. We can restrict to $0<\alpha\le 1$. Indeed, since $\sfat[\alpha]$ is nonincreasing in $\alpha$, we can map the case $1<\alpha\le2$ to the case $\alpha =1$, which just changes the constant.
If $\alpha>2$, even a depth-one shattered
tree would require two function values separated by at least $\alpha$,
which is impossible for a class taking values in $[-1,1]$, so $\sfat[\alpha]$ would be $0$.

\paragraph{The case $p < q$ and $p\ge2$.}
Apply
\cref{cor:barycentric_covering_bound} with $r=p$ and $L=1$.
Taking $R=1$, for which $\CoverN(\ball{p}{d},\ball{p}{d})=1$, gives $\sfat[\alpha](\cL[p,q^\ast][d])\lesssim_p\alpha^{-p}$.
Using \eqref{eq:volumetric-entropy-input} instead and taking
$R=\alpha d^{1/p}$ gives
\[
    \sfat[\alpha](\cL[p,q^\ast][d])
    \lesssim_p
    d\left(1+\log\pospart\frac1{\alpha d^{1/p}}\right).
\]
\paragraph{The case $p < q$ and $p<2$.}
Set
\[
    s\defi\min\{2,q\},
    \qquad
    b\defi\frac1p-\frac1s.
\]
Then $p<s\le2$, $\ball{q^{\ast}}{d}\subseteq \ball{s^{\ast}}{d}$, and
$\theta=b+1/2$. Applying \cref{cor:barycentric_covering_bound} with
$r=s$ and $L=1$, followed by \eqref{eq:schutt-entropy-input}, gives for
every $0<R<1$
\[
    \sfat[\alpha](\cL[p,q^\ast][d])
    \lesssim_{p,q}
    \frac{R^2}{\alpha^2}
    +R^{-1/b}\log\frac{Ced}{R}.
\]
For $0<\alpha<1$, take $R=\alpha^{2b/(2b+1)}$. Since
$\log(Ced/R)\lesssim_{p,q}L_\alpha$, we obtain
\[
    \sfat[\alpha](\cL[p,q^\ast][d])
    \lesssim_{p,q}
    L_\alpha\alpha^{-1/\theta}.
\]
The case $\alpha=1$ follows by monotonicity.
For the $d$-dependent term, use \eqref{eq:volumetric-entropy-input} in
\cref{cor:barycentric_covering_bound} with $r=s$, $L=1$, and
$R=\alpha\sqrt d$. Since $\theta=1/p-1/s+1/2$, this gives
\[
    \sfat[\alpha](\cL[p,q^\ast][d])
    \lesssim_{p,q}
    d\left(1+\log\pospart\frac1{\alpha d^\theta}\right).
\]
Since
$
\theta= b + 1/2 = 1/p-(1/q-1/2)\pospart,
$
we have the result.

\paragraph{The case $p \ge q$ and $q<\infty$.}
Set
$
    r\defi\max\{q,2\}
$
 and 
$
    A\defi d^{1/q-1/p}.
$

If $q\ge2$, use $r=q$, $L=1$,
and $R=A$, the inclusion $\ball{p}{d}\subseteq A \ball{q}{d}$ makes the covering
number equal to one and gives
$
    \sfat[\alpha](\cL[p,q^\ast][d])
    \lesssim_q
    \left(\frac A\alpha\right)^q.
$

If $q<2$ and $p>1$, set $u=\min\{p,2\}$. The inclusions
\[
    \ball{p}{d}\subseteq d^{1/u-1/p}\ball{u}{d},
    \qquad
    \ball{q^\ast}{d}\subseteq d^{1/q-1/u}\ball{u^\ast}{d}
\]
have product distortion $A$. Applying the same covering bound with
$r=u$ and a one-element cover therefore gives
$
    \sfat[\alpha](\cL[p,q^\ast][d])
    \lesssim_{p,q}
    \left(\frac A\alpha\right)^2.
$

For $p=q=1$, it suffices to consider the nontrivial range
$\alpha\le2$. Take $u=1+1/L_\alpha$. Then
$\ball{1}{d}\subseteq \ball{u}{d}$, $\ball{\infty}{d}\subseteq e\ball{u^\ast}{d}$, and the
quadratic modulus constant from \citet{ball1994sharp} is of order $u-1$.
Thus the same argument gives
$
    \sfat[\alpha](\cL[1,\infty][d])
    \lesssim
    L_\alpha\alpha^{-2}.
$

For the dimension dependent term, set
$
    \lambda\defi d^{1/r^\ast-1/q^\ast}=d^{1/q-1/r}.
$
Then $\ball{q^\ast}{d}\subseteq\lambda \ball{r^\ast}{d}$. Hence
\cref{cor:barycentric_covering_bound} and
\eqref{eq:volumetric-entropy-input} give
\[
    \sfat[\alpha]
    \lesssim_{p,q}
    \inf_{R>0}
    \left[
        \left(\frac{\lambda R}{\alpha}\right)^r
        +d\left(1+\log\pospart\frac{d^{1/r-1/p}}R\right)
    \right].
\]
Taking $R=\alpha\lambda^{-1}d^{1/r}$ makes the first term equal to $d$,
and
$
    \lambda d^{-1/p}=A d^{-1/r}.
$
Thus
\[
    \sfat[\alpha]
    \lesssim_{p,q}
    d\left(1+\log\pospart\frac{A}{\alpha d^{1/r}}\right).
\]

\paragraph{The case $p=q=\infty$.}
Apply \eqref{eq:volumetric-entropy-input} to an
$\alpha/4$-net of $\ball{\infty}{d}$ in $\ell_\infty$. 
Let $\mathcal V$ be such a net. For every $w\in \ball{\infty}{d}$, choose
$v\in\mathcal V$ with $\lVert w-v\rVert_\infty\le\alpha/4$. Uniformly over
$g\in \ball{1}{d}$, the corresponding linear functions then differ by at most
$\alpha/4$. Replacing the witnesses of an $\alpha$-shattered tree by
their nearest net points therefore gives an $(\alpha/2)$-shattered tree for
the finite class indexed by $\mathcal V$. A depth-$n$ shattered tree
requires a distinct function for each of its $2^n$ root-to-leaf paths, so
$n\le\log_2|\mathcal V|$. Finally,
\eqref{eq:volumetric-entropy-input} bounds this logarithm by
$O\bigl(d(1+\log\pospart(1/\alpha))\bigr)$. Thus
\[
    \sfat[\alpha](\cL[\infty,1][d])
    \lesssim
    d\left(1+\log\pospart\frac1\alpha\right).
\]
\end{proof}

\subsection{Lower bound}

We start by providing a proof of the following fact from \citet[Corollary~3.2]{mendelson2004shattering} for completeness and the reader's convenience. It involves the equivalence between fat shattering and a certain cube containment in some cases.
Let $\X,\H\subseteq\mathbb R^d$ be centrally symmetric convex bodies and,
for $g_1,\ldots,g_m\in\H$, define the evaluation image
\[
    \newtarget{def:evaluation-image}{\EvalCube[\X,\H]{g_{1:m}}}
    \defi
    \left\{
        (\langle u,g_1\rangle,\ldots,\langle u,g_m\rangle):
        u\in\X
    \right\}.
\]

\begin{fact}[Fat-shattering and evaluation cubes]
\label{fact:fat-shattering-evaluation-cubes}\linktoproof{fact:fat-shattering-evaluation-cubes}
Let $\X,\H\subseteq\mathbb R^d$ be centrally symmetric convex bodies, let
$m\in\mathbb N$, and let $g_1,\ldots,g_m\in\H$. For every $\alpha>0$,
\begin{equation}\label{eq:equivalence_of_fat_shattering_to_cubes}
    \{g_1,\ldots,g_m\}\text{ is $\alpha$-shattered by }
    \cL[\X,\H]
    \quad\Longleftrightarrow\quad
    \frac{\alpha}{2}\ball{\infty}{m}
    \subseteq
    \EvalCube[\X,\H]{g_{1:m}}.
\end{equation}
\end{fact}

As explained in the introduction,
sharp fat-shattering lower bounds were known for
$(\X,\H)=(\ball{p}{d},\ball{q^\ast}{d})$. Through
\cref{fact:fat-shattering-evaluation-cubes}, they give the following
evaluation cubes. We include a direct proof for completeness.

\begin{fact}[Evaluation cubes in $\ell_p/\ell_q$ cases]
\label{lem:evaluation_cubes}\linktoproof{lem:evaluation_cubes}
There is a constant $c_{p,q}>0$ such that, for every
$1\le m\le c_{p,q}d$, there exist
$g_1,\dots,g_m\in \ball{q^{\ast}}{d}$ satisfying
\[
\left\{
(\langle w,g_1\rangle,\dots,\langle w,g_m\rangle):
w\in \ball{p}{d}
\right\}
\supseteq
c_{p,q}{\rhopqd[m]}[-1,1]^m .
\]
\end{fact}

The new sequential step is to query each coordinate of such a cube repeatedly
by binary search.

\begin{lemma}[Cube-to-tree amplification]\label{lem:cube_to_tree}\linktoproof{lem:cube_to_tree}
Let $\X,\H\subseteq\mathbb R^d$ be centrally symmetric convex bodies and let $g_1,\ldots,g_m\in\H$. We have
\[
    R \ball{\infty}{m}\subseteq \EvalCube[\X,\H]{g_{1:m}} \qquad \implies \qquad
    \sfat[\alpha](\cL[\X,\H])
    \ge
    m\left\lfloor\log_2\left(1+\frac{2R}{\alpha}\right)\right\rfloor.
\]
\end{lemma}

\begin{proof}\linkofproof{lem:cube_to_tree}
A single interval $[-R,R]$ supports the usual binary-search tree. For a
current interval $[a,b]$, use the midpoint $s=(a+b)/2$ as threshold. The
minus child retains $[a,s-\alpha/2]$, and the plus child retains
$[s+\alpha/2,b]$. After $t$ levels, every surviving interval has length
$
    2^{-t}(2R+\alpha)-\alpha.
$
It is therefore nonempty through
$
    h\defi
    \left\lfloor\log_2\left(1+\frac{2R}{\alpha}\right)\right\rfloor
$
levels. Every point in a terminal interval satisfies the margin inequalities
along the corresponding branch.

Concatenate this construction in $m$ blocks. In block $j$, label every
node by $g_j$ and use the binary-search thresholds for coordinate $j$.
A branch of depth $mh$ determines terminal intervals
$I_1,\dots,I_m\subseteq[-R,R]$. Choose
$
    z_j\in I_j, j\in[m].
$
By the cube containment, for a $u\in\X$, we have
\[
    \langle u,g_j\rangle=z_j,
    \qquad j\in[m].
\]
This witness satisfies every margin constraint on the branch, so the
concatenated tree is $\alpha$-shattered to depth $mh$.
\end{proof}

\begin{theorem}[sfat lower bound]
\label{thm:norm-specific-lower-bound}\linktoproof{thm:norm-specific-lower-bound}
Let $1\le p,q\le\infty$, $(p,q)\ne(\infty,\infty)$, and assume
$0<\alpha\le2\LambdaPair[p,q^\ast]$. Then
\[
    \sfat[\alpha](\cL[p,q^\ast][d])
    \gtrsim_{p,q}
    \begin{cases}
    \displaystyle
    \min\left\{
        \alpha^{-\left(\frac{1}{p}-(\frac{1}{q}-\frac{1}{2})\pospart\right)^{-1}},
        \ d\left(
            1+\log\pospart
            \frac{1}{\alpha d^{\frac{1}{p}-(\frac{1}{q}-\frac{1}{2})\pospart}}
        \right)
    \right\},
    & p < q,\\[4mm]
    \displaystyle
    \min\left\{
        \left(\frac{d^{1/q-1/p}}{\alpha}\right)^{\max\{q,2\}},
        \ d\left(
            1+\log\pospart
            \frac{d^{1/q-1/p}}{\alpha d^{1/\max\{q,2\}}}
        \right)
    \right\},
    & p \geq q.
    \end{cases}
\]
For $p=q=\infty$ and $0<\alpha\le2$,
$
    \sfat[\alpha](\cL[\infty,1][d])
    \gtrsim
    d\left(1+\log\pospart\frac1\alpha\right).
$
\end{theorem}

\begin{proof}\linkofproof{thm:norm-specific-lower-bound}
First assume $0<\alpha\le c_{p,q}\LambdaPair[p,q^\ast]$ for a sufficiently small constant $0<c_{p,q}\le1$.
Fix $1\le m\le c_{p,q}d$. By
\cref{lem:evaluation_cubes}, there are $g_1,\ldots,g_m\in \ball{q^\ast}{d}$
whose evaluation image contains a cube of radius
$c_{p,q}\rhopqd[m]$. Applying \cref{lem:cube_to_tree} with
$(\X,\H)=(\ball{p}{d},\ball{q^\ast}{d})$, and then taking the supremum over such
$m$, gives the profile bound below with the supremum restricted to
$m\le c_{p,q}d$. For $m>c_{p,q}d$, choose
$m'\le c_{p,q}d$ with $m'\asymp_{p,q}m$. Since
$\rhopqd$ is nonincreasing, the profile term at $m'$ is at least a
constant multiple of the term at $m$. The finitely many dimensions for
which such an $m'$ does not exist are absorbed into the constants.
Consequently,
\[
    \sfat[\alpha](\cL[p,q^\ast][d])
    \gtrsim_{p,q}
    \sup_{1\le m\le d}m
    \left\lfloor
        \log_2\left(
            1+\frac{c_{p,q}\rhopqd[m]}{\alpha}
        \right)
    \right\rfloor.
\]

We evaluate this profile directly. Fix $K>0$, $A>0$, and
$0<\beta\le1$, and write $x=A/\alpha$. On the nontrivial range
$0<\alpha\lesssim_K A$, choosing $m$ comparable to
$\min\{x^{1/\beta},d\}$ gives
\[
    \sup_{1\le m\le d}
    m\left\lfloor
        \log_2\left(1+Kxm^{-\beta}\right)
    \right\rfloor
    \gtrsim_{\beta,K}
    \min\left\{
        x^{1/\beta},
        \ d\left(1+\log\pospart\frac{x}{d^\beta}\right)
    \right\}.
\]
Indeed, if $x<d^\beta$, choose $m\asymp_{\beta,K}x^{1/\beta}$, so the
logarithm is bounded below by a positive constant. If $x\ge d^\beta$,
choose $m$ to be a sufficiently small fixed multiple of $d$, depending
on $\beta,K$. This yields the second term. Adjusting the comparison
constants absorbs the floor.
First consider $p=q=\infty$. Then
$\rho_{\infty,\infty,d}(m)=1$, and taking $m=d$ gives
$d(1+\log\pospart(1/\alpha))$ up to universal constants.
Otherwise, apply the estimate with $K=c_{p,q}$ to
the two cases in
\eqref{eq:rho-pqd}. We obtain the result with:
\[
    (A,\beta)
    =
    \left(
        1,
        \frac1p-\left(\frac1q-\frac12\right)\pospart
    \right)
\ \text{ for } \ p<q, \ \text{ and } \
    (A,\beta)
    =
    \left(
        d^{\frac{1}{q}-\frac{1}{p}},
        \frac1{\max\{q,2\}}
    \right)
\ \text{ for } \ p \geq q.
\]
For $c_{p,q}\LambdaPair[p,q^\ast]<\alpha\le2\LambdaPair[p,q^\ast]$ and
$(p,q)\ne(\infty,\infty)$, the claimed lower bound is $O_{p,q}(1)$,
whereas $\sfat[\alpha]\ge1$: choose $g\in \ball{q^\ast}{d}$ and $u\in \ball{p}{d}$
attaining $\langle u,g\rangle=\LambdaPair[p,q^\ast]$, and use $\pm u$ with threshold zero.
At $p=q=\infty$, coordinate evaluations of $\ball{\infty}{d}$ give
$\sfat[\alpha]\ge d$ for every $0<\alpha\le2$, which also covers the remaining scales.
\end{proof}

\subsection{Proofs of essentially-known facts provided for completeness or elegance}
\label{app:sfat-profile-known-proofs}

\begin{proof}\linkofproof{lem:uniformly-convex-localization}
Suppose the class $\alpha$-shatters a depth-$n$ tree, and choose a
witness $w_\varepsilon\in K$ for every leaf $\varepsilon$. For a node
$u$, let $L(u)$ be the leaves below it and set
\[
    m_u\defi\frac1{|L(u)|}\sum_{\varepsilon\in L(u)}w_\varepsilon.
\]
Convexity of $K$ gives $m_u\in K$. If $g_u\in\H$ and $s_u$
are the instance and threshold at an internal node, averaging the shattering
inequalities over its two child subtrees gives
\[
    \langle m_{u+},g_u\rangle\ge s_u+\frac\alpha2,
    \qquad
    \langle m_{u-},g_u\rangle\le s_u-\frac\alpha2.
\]
Consequently,
\[
    \|m_{u+}-m_{u-}\|
    \ge
    \frac{\langle m_{u+}-m_{u-},g_u\rangle}{\|g_u\|_*}
    \ge
    \frac\alpha L.
\]

For each level $t$, define
\[
    \overline\Phi_t\defi 2^{-t}\sum_{|u|=t}\Phi(m_u).
\]
Since $m_u=(m_{u+}+m_{u-})/2$, uniform convexity yields
\[
    \overline\Phi_{t+1}
    \ge
    \overline\Phi_t+c\left(\frac\alpha L\right)^\kappa.
\]
Thus $nc(\alpha/L)^\kappa\le
\overline\Phi_n-\overline\Phi_0\le D$, which proves the lemma.
\end{proof}

\begin{proof}\linkofproof{fact:fat-shattering-evaluation-cubes}
First suppose that
$\frac{\alpha}{2}\ball{\infty}{m}\subseteq \EvalCube[\X,\H]{g_{1:m}}$. For every
$\varepsilon\in\{\pm1\}^m$, choose $u_\varepsilon\in\X$ whose evaluation
vector is $(\alpha/2)\varepsilon$. The corresponding functions shatter
$g_1,\ldots,g_m$ with all thresholds equal to zero.

Conversely, suppose that $g_1,\ldots,g_m$ are $\alpha$-shattered with
threshold vector $s\in\mathbb R^m$, and write
$\Cube=\EvalCube[\X,\H]{g_{1:m}}$. Fix $a\in\mathbb R^m$, and choose the two
shattering functions corresponding to the sign patterns of $a$ and $-a$,
with arbitrary signs in the zero coordinates. Their evaluation vectors
$y^+,y^-\in \Cube$ satisfy
\[
    \langle a,y^+\rangle
    \ge
    \langle a,s\rangle+\frac{\alpha}{2}\|a\|_1,
    \qquad
    \langle-a,y^-\rangle
    \ge
    \langle-a,s\rangle+\frac{\alpha}{2}\|a\|_1.
\]
Since $\Cube$ is centrally symmetric, its support function satisfies
\[
    h_{\Cube}(a)=h_{\Cube}(-a)
    \ge
    \frac{\alpha}{2}\|a\|_1+|\langle a,s\rangle|
    \ge
    \frac{\alpha}{2}\|a\|_1.
\]
The rightmost expression is the support function of
$(\alpha/2)\ball{\infty}{m}$. Both sets are closed and convex, so support-function
duality gives
$(\alpha/2)\ball{\infty}{m}\subseteq \Cube$. This cube spans $\mathbb R^m$, so the
evaluation map $u\mapsto(\langle u,g_i\rangle)_{i=1}^m$ has rank $m$,
which also proves that the $g_i$ are linearly independent.
\end{proof}

\begin{proof}\linkofproof{lem:evaluation_cubes}
Given $g_1,\dots,g_m$, define $T:\mathbb R^m\to\mathbb R^d$ by
$Ta=\sum_{i=1}^ma_ig_i$, and let $\Cube$ be the evaluation image in the
statement. Its support function is
\[
    \sup_{y\in \Cube}\langle a,y\rangle
    =
    \sup_{w\in \ball{p}{d}}\langle w,Ta\rangle
    =
    \|Ta\|_{p^\ast}.
\]
Thus it suffices in every case to prove
$\|Ta\|_{p^\ast}\gtrsim_{p,q}\rhopqd[m]\|a\|_1$.

\paragraph{The range $p\le q$.}
If $q\ge2$, take $g_i=e_i$. Then
\[
    \|Ta\|_{p^\ast}=\|a\|_{p^\ast}
    \ge m^{-1/p}\|a\|_1
    =\rhopqd[m]\|a\|_1.
\]

If $q<2$, choose an orthonormal basis
$u_1,\dots,u_m\in\mathbb R^m$ with
\[
    \|u_i\|_\infty\lesssim m^{-1/2}
\]
and embed it in the first $m$ coordinates of $\mathbb R^d$. Set
$
    g_i\defi c_qm^{1/q-1/2}u_i.
$
Then $g_i\in \ball{q^\ast}{d}$, and, since $p^\ast\ge q^\ast>2$,
\[
    \|Ta\|_{p^\ast}
    \gtrsim_{p,q}
    m^{1/q-1/2}m^{1/p^\ast-1/2}\|a\|_2
    \ge
    m^{-(1/p+1/2-1/q)}\|a\|_1
    =
    \rhopqd[m]\|a\|_1.
\]

\paragraph{The range $q<p$.}
If $q>2$, partition a subset of $[d]$ into $m$ disjoint blocks
$B_1,\dots,B_m$ of common size $b\asymp d/m$, and set
$
    g_i\defi b^{-1/q^\ast}\mathbf 1_{B_i}.
$
Then $\|g_i\|_{q^\ast}=1$, and the disjoint supports give
\[
    \|Ta\|_{p^\ast}
    =b^{1/q-1/p}\|a\|_{p^\ast}
    \ge
    b^{1/q-1/p}m^{-1/p}\|a\|_1
    \asymp
    d^{1/q-1/p}m^{-1/q}\|a\|_1.
\]

Finally, suppose $q\le2$. Kashin's random-sign construction
\citep{kashin1977diameters} gives, for $m\le c_pd$, a matrix
$G\in\{\pm1\}^{d\times m}$ such that
\[
    \|Ga\|_{p^\ast}\gtrsim_p d^{1/p^\ast}\|a\|_2,
    \qquad a\in\mathbb R^m.
\]
For $p^\ast\le2$, this is the Kashin $\ell_1$-estimate followed by norm
comparison. For $p^\ast\ge2$, it follows from the smallest singular-value
bound and the comparison of $\ell_{p^\ast}$ with $\ell_2$. If $G_i$
is column $i$, define
$
    g_i\defi d^{-1/q^\ast}G_i.
$
Then $\|g_i\|_{q^\ast}=1$, and
\[
    \|Ta\|_{p^\ast}
    =d^{-1/q^\ast}\|Ga\|_{p^\ast}
    \gtrsim_{p,q}
    d^{1/q-1/p}m^{-1/2}\|a\|_1.
\]
In each case the support-function criterion gives the asserted cube.
\end{proof}

\section{Proofs of minimax rate for some infinite-dimensional cases}

\begin{proof}\linkofproof{prop:sobolev-function-space-example}
We first prove the fat-shattering lower bound. 
By the Rellich-Kondrachov compact embedding theorem, the inclusion $J:H^s(\Omega)\hookrightarrow L_2(\Omega)$
is compact. Since both spaces are Hilbert and $J$ is injective with dense range in $L_2(\Omega)$, the singular-value decomposition for compact Hilbert-space operators gives an orthonormal basis $(e_j)_{j\ge1}$ of $H^s(\Omega)$, an orthonormal basis $(\phi_j)_{j\ge1}$ of $L_2(\Omega)$, and positive weights $(w_j)_{j\ge1}$ satisfying $Je_j=w_j^{-1}\phi_j$. 
Thus, if $u=\sum_{j\ge1}a_je_j\in H^s(\Omega)$, then, viewing $u$ as an element of $L_2(\Omega)$,  $ u=Ju=\sum_{j\ge1}a_jw_j^{-1}\phi_j$.  Writing $\theta_j=a_jw_j^{-1}$ gives us $\|u\|_{H^s(\Omega)}^2
    =
    \sum_{j\ge1}a_j^2
    =
    \sum_{j\ge1}w_j^2\theta_j^2$.  
Hence, we have the following representation 
\begin{equation*}
    \X
    =
    \bigg\{
        \sum_{j\ge1}\theta_j\phi_j:
        \sum_{j\ge1}w_j^2\theta_j^2\le1
    \bigg\}.
\end{equation*}
The classical spectral estimate for the Sobolev embedding gives $w_j\asymp_{s,m}j^{s/m}$,
see \citet[Section~3.3.4]{edmunds1996function}.\footnote{In the notation of \citet{edmunds1996function}, take $A=F$, $p_1=p_2=q_1=q_2=2$, $s_1=s$, $s_2=0$, and dimension $m$. Their approximation-number estimate then gives $a_j\asymp_{s,m}j^{-s/m}$. For compact operators between Hilbert spaces, approximation numbers coincide with singular values, while in our decomposition the corresponding singular value is $w_j^{-1}$.} Let $n$ satisfy $\frac{\alpha^2}{4}\sum_{j=1}^n w_j^2\le1$, and take $g_j=\phi_j \in \H$, $j=1,\ldots,n$, with zero thresholds. For every
$\epsilon\in\{\pm1\}^n$, define $u_\epsilon
    \defi
    \frac{\alpha}{2}
    \sum_{j=1}^n\epsilon_j\phi_j$.
Then $\sum_{j\ge1}w_j^2
    \langle u_\epsilon,\phi_j\rangle_{L_2}^2
    =
    \frac{\alpha^2}{4}\sum_{j=1}^n w_j^2
    \le1$, 
so $u_\epsilon\in\X$, while for every $j\le n$, $\epsilon_j\langle g_j,u_\epsilon\rangle
    =
    \frac{\alpha}{2}$.
Thus $g_1,\ldots,g_n$ are $\alpha$-fat-shattered. Since $w_j\asymp_{s,m}j^{s/m}$, the standard integral comparison for sums of powers gives
\begin{align*}
    \sum_{j=1}^n w_j^2 \asymp_{s,m}
    \sum_{j=1}^n j^{2s/m} \asymp_{s,m}
    \int_0^n x^{2s/m}\,dx=
    \frac{1}{1+2s/m}\,
    n^{1+2s/m} \asymp_{s,m}
    n^{1+2s/m}.
\end{align*}
Therefore the condition on $n$, and the fact that we can have a set of size $n$ that is fat-shattered, implies $\fat[\alpha](\cL[\X,\H])
    \gtrsim_{s,m}
    \alpha^{-\frac{2m}{m+2s}}$.
\medskip

For the sequential upper bound, we cover $\X$ in $L_2(\Omega)$ by radius-$R$ balls, apply uniformly convex localization on each ball, and combine the resulting bounds using the sequential union bound, as in \cref{lem:uniformly-convex-localization,lem:seq_union_bound}. For $f_u(g)=\langle g,u\rangle$ and $f_v(g)=\langle g,v\rangle$, since $\H=B_{L_2(\Omega)}$, we have $\|f_u-f_v\|_\infty=\sup_{g\in\H}|\langle g,u-v\rangle|=\|u-v\|_{L_2}$. Thus a uniform cover of $\cL[\X,\H]$ is equivalent to an $L_2$-cover of $\X$.
Let $N_2(\X,R)$ denote the covering number of $\X$ by radius-$R$ balls in $L_2(\Omega)$. The Sobolev entropy estimate of \citet[Sections~3.3.2 and~3.5]{edmunds1996function} gives
\[
    \log N_2(\X,R)\lesssim_{s,m}R^{-m/s},\qquad 0<R<1.
\]
Cover $\X$ by these balls.
Note that the squared $L^2$ distance to the center of each ball is $(1/4,2)$-uniformly convex %
with range at most $R^2$ (see \cref{def:uniform-convexity-and-range} for the notions of uniform convexity and range) \citep{clarkson1936uniformly}. Hence, for any $R \in (0,1)$, \cref{lem:uniformly-convex-localization,lem:seq_union_bound} give
\begin{equation*}
    \sfat[\alpha](\cL[\X,\H])
    \lesssim_{s,m} \inf_{R \in (0,1)}
    \left\{\frac{R^2}{\alpha^2}
    +
    \log N_2(\X,R)\right\} \lesssim_{s,m} \inf_{R \in (0,1)} \left\{\frac{R^2}{\alpha^2}+R^{-m/s}\right\} \lesssim_{s,m}
    \alpha^{-\frac{2m}{m+2s}},
\end{equation*}
where the last inequality follows by taking $R=\alpha^{2s/(m+2s)}$ for $0<\alpha<1$, so both terms equal $\alpha^{-2m/(m+2s)}$.
\end{proof}

\begin{proof} \linkofproof{prop:lipschitz_bound}
For the ordinary fat-shattering lower bound, take a grid
$x_1,\ldots,x_n\in[1/2,1]^m$ (this is basically the elements of $\H$) with spacing $\alpha$ in each coordinate. Then
$n\asymp_m\alpha^{-m}$ and, for every $i\neq j$,
$\|x_i-x_j\|_\infty\ge\alpha$. For every
$\epsilon\in\{\pm1\}^n$, let $u_\epsilon(0)=0$ and
$u_\epsilon(x_i)=\epsilon_i\alpha/2$. These values are $1$-Lipschitz on
$\{0,x_1,\ldots,x_n\}$, since for $i\neq j$, $|u_\epsilon(x_i)-u_\epsilon(x_j)|
    \le
    \alpha
    \le
    \|x_i-x_j\|_\infty$,
while
$|u_\epsilon(x_i)-u_\epsilon(0)|=\alpha/2\le\|x_i\|_\infty$.
By the McShane extension theorem, $u_\epsilon$ extends to a
$1$-Lipschitz function on $\Omega$. Since $g_i=\delta_{x_i}$, $\epsilon_i\langle \delta_{x_i},u_\epsilon\rangle
    =
    \epsilon_i u_\epsilon(x_i)
    =
    \frac{\alpha}{2}$.
Hence $x_1,\ldots,x_n$ are
$\alpha$-fat-shattered with zero thresholds, and therefore $\fat[\alpha](\cL[\X,\H])
    \gtrsim_m
    \alpha^{-m}$. \medskip

For the sequential upper bound, the entropy estimate for Lipschitz functions gives
\[
    \log \Ninf(\X,\alpha)
    \lesssim_m
    \alpha^{-m},
\]
see \citet[Theorem 2.7.1]{vanderVaart1996}. Since $\H=\{0\}\cup\{\pm\delta_x:x\in\Omega\}$, the uniform metric on $\cL[\X,\H]$ is exactly the uniform metric on $\X$, as $\sup_{g\in\H}
    |\langle g,u-v\rangle|
    =
    \sup_{x\in\Omega}
    |u(x)-v(x)|
    =
    \|u-v\|_\infty$.

Now suppose a tree of depth $n$ is $\alpha$-shattered, and let
$u_\epsilon\in\X$ be a witness for each path
$\epsilon\in\{\pm1\}^n$. If two distinct paths $\epsilon,\epsilon'$ first
differ at level $t$, they encounter the same point $x_t$ and threshold
$r_t$. Assuming without loss of generality that $\epsilon_t=1$ and
$\epsilon_t'=-1$, the shattering inequalities give
$|u_\epsilon(x_t)-u_{\epsilon'}(x_t)|\ge\alpha$,
and hence
$\|u_\epsilon-u_{\epsilon'}\|_\infty\ge\alpha$.
Thus the $2^n$ witnesses are pairwise $\alpha$-separated, so by triangle inequality any uniform
$\alpha/4$-cover of $\X$ contains at least $2^n$ elements. Therefore
\begin{equation*}
    n\log2
    \le
    \log \Ninf(\X,\alpha/4)
    \lesssim_m
    \alpha^{-m}.
\end{equation*}
Hence
$\sfat[\alpha](\cL[\X,\H])\lesssim_m\alpha^{-m}$. 
\end{proof}

\begin{proof}\linkofproof{prop:basis_bound}
For the ordinary lower bound, let $n\le(2/\alpha)^p$ and consider the fixed functionals $g_1,\ldots,g_n$ with zero thresholds.
For every $(\epsilon_1,\ldots,\epsilon_n)\in\{\pm1\}^n$, let $ u
    =
    \frac{\alpha}{2}
    \sum_{j=1}^n
    \epsilon_j\phi_j$.
Since $n\le(2/\alpha)^p$, 
$\|u\|^p
    =
    n\left(\frac{\alpha}{2}\right)^p
    \le1$,  
so $u\in\X$. Moreover, for each $j\le n$, $\epsilon_j\langle g_j,u\rangle
    =
    \frac{\alpha}{2}$.
Thus $g_1,\ldots,g_n$ are $\alpha$-fat-shattered with zero thresholds, and $\fat[\alpha](\cL[\X,\H])
    \gtrsim_p
    \alpha^{-p}$. \medskip

For the sequential upper bound, consider an $\alpha$-shattered tree of depth $n$. We construct one root-to-leaf path $\varepsilon=(\varepsilon_1,\dots,\varepsilon_n)$ greedily.
Fix a coordinate $g_j$. We prove by induction that, after its $k$-th occurrence along the constructed path, we have $|\langle g_j,u\rangle|
    \ge
    \left(k-\frac12\right)\alpha$ 
for every $u\in\X$ realizing the current partial path, i.e., satisfying all the shattering inequalities corresponding to the branches chosen so far (there is at least one because the tree is $\alpha$-shattered). For $k=1$, choose the branch whose constraint lies away from zero, which gives $|\langle g_j,u\rangle|\ge\alpha/2$.
For the induction step, suppose for instance that after the $k$th occurrence we have $\langle g_j,u\rangle\ge(k-\frac12)\alpha$. At the next occurrence, let $r$ be the threshold. Since the tree is shattered, the negative child is realizable, so there exists some $u^-\in\X$ satisfying the previous constraints together with
$\langle g_j,u^-\rangle\le r-\alpha/2$. By the induction hypothesis, the same $u^-$ satisfies
$\langle g_j,u^-\rangle\ge(k-\frac12)\alpha$. Hence
\[
    r-\frac{\alpha}{2}
    \ge
    \langle g_j,u^-\rangle
    \ge
    \left(k-\frac12\right)\alpha,
\]
so $r\ge k\alpha$. Choosing the positive child then gives
$\langle g_j,u\rangle\ge r+\alpha/2\ge(k+\frac12)\alpha$. The negative case is symmetric.

Thus, if $k_j$ denotes the number of occurrences of $g_j$ along the resulting path, any witness $u\in\X$ for that path satisfies
$|\langle g_j,u\rangle|\ge(k_j-\frac12)\alpha$ whenever $k_j>0$. Since $\sum_j k_j=n$,
\begin{align*}
    1
    \ge \|u\|^p=
    \sum_j |\langle g_j,u\rangle|^p\ge
    \alpha^p\sum_{j:k_j>0}\left(k_j-\frac12\right)^p\ge
    \left(\frac{\alpha}{2}\right)^p\sum_j k_j
    =
    \left(\frac{\alpha}{2}\right)^p n.
\end{align*}
Hence $n\le(2/\alpha)^p$, so
$\sfat[\alpha](\cL[\X,\H])\lesssim_p\alpha^{-p}$.
For the optimization consequence, replace $\H$ by $\{0\}\cup\{\pm g_j:j\ge1\}$. This does not change either shattering dimension.
\end{proof}

\printbibliography
\endgroup

\end{document}